\documentclass[10pt,twoside,a4paper,english,reqno]{amsart}

\usepackage{amscd}
\usepackage{caption}
\usepackage{a4wide}
\usepackage{amssymb,amsthm,amsmath,amsfonts,mathrsfs}
\usepackage{bm,latexsym,enumitem,dsfont,tabularx,graphicx,url}
\usepackage[utf8]{inputenc}
\usepackage{hyperref}
\usepackage{stmaryrd}
\usepackage{graphicx}
\usepackage{subcaption}
\usepackage{algorithm}

\usepackage{algpseudocode}
\hypersetup{
    colorlinks=true,
    linkcolor=blue,
    filecolor=red,
    urlcolor=blue,
    citecolor=red,
}
\usepackage{nicefrac,microtype,upref,ulem,doi,footmisc}
\usepackage[svgnames]{xcolor}
\usepackage{tikz}
\usepackage{orcidlink}

\usepackage{booktabs}
\usepackage{makecell}
\usepackage{pdflscape}
\usepackage{siunitx}

\usepackage{placeins}

\algrenewcommand\algorithmicrequire{\textbf{Input:}}
\algrenewcommand\algorithmicensure{\textbf{Output:}}

\theoremstyle{plain}
\newtheorem{theorem}{Theorem}[section]
\newtheorem{proposition}[theorem]{Proposition}
\newtheorem{lemma}[theorem]{Lemma}

\newtheorem{remark}[theorem]{Remark}
\newtheorem{assumption}[theorem]{Assumption}
\newtheorem{corollary}[theorem]{Corollary}
\numberwithin{equation}{section}

\algrenewcommand\algorithmicrequire{\textbf{Require:}}
\algrenewcommand\algorithmicensure{\textbf{Ensure:}}
\algrenewcommand\algorithmicfor{\textbf{For}}
\algrenewcommand\algorithmicdo{\textbf{do}}
\algrenewcommand\algorithmicif{\textbf{If}}
\algrenewcommand\algorithmicthen{\textbf{then}}
\algrenewcommand\algorithmicend{\textbf{end}}

\usepackage{xcolor}
\definecolor{lateksii_color}{RGB}{155,0,119}

\title[A FEM for inverse fractional conductivity]{A finite element exterior-to-interior reconstruction in the fractional conductivity inverse problem}

\author[M. Dwivedi]{Mukul Dwivedi\textsuperscript{1,*}\,\orcidlink{0000-0002-5891-2788}}
\address{Department of Mathematics, Saarland University, Saarbr\"ucken, Germany}
\email{mukul.dwivedi@uni-saarland.de}

\author[J. Railo]{Jesse Railo\textsuperscript{2}\,\orcidlink{0000-0001-9226-4190}}
\address{Computational Engineering, School of Engineering Sciences, Lappeenranta-Lahti University of Technology, Finland}
\email{jesse.railo@lut.fi}

\author[A. Rupp]{Andreas Rupp\textsuperscript{1}\,\orcidlink{0000-0001-5527-7187}}
\address{Department of Mathematics, Saarland University, Saarbr\"ucken, Germany}
\email{andreas.rupp@uni-saarland.de}

\makeatletter
\g@addto@macro\@setauthors{\par\nobreak\begingroup\centering\footnotesize
\textsuperscript{1}Department of Mathematics, Saarland University, Saarbr\"ucken, Germany\par
\textsuperscript{2}Computational Engineering, School of Engineering Sciences,\par
Lappeenranta-Lahti University of Technology, Finland\par\smallskip
\textsuperscript{*}Corresponding author: Mukul Dwivedi, \href{mailto:mukul.dwivedi@uni-saarland.de}{\texttt{mukul.dwivedi@uni-saarland.de}}\par
Jesse Railo: \href{mailto:jesse.railo@lut.fi}{\texttt{jesse.railo@lut.fi}}; Andreas Rupp: \href{mailto:andreas.rupp@uni-saarland.de}{\texttt{andreas.rupp@uni-saarland.de}}\par\endgroup}
\renewcommand{\@setaddresses}{}
\makeatother


\newcommand{\R}{\mathbb{R}}

\newcommand{\supp}{\operatorname{supp}}

\begin{document}

\begin{abstract}
We develop a regularized finite element method for the fractional conductivity equation, recovering the conductivity on an exterior observation region from noisy Dirichlet-to-Neumann measurements and using this reconstruction to determine the interior conductivity. Mesh-localized inputs and mass-lumped Tikhonov regularization discretize the exterior determination principle of Covi--Railo--Zimmermann (2026, Calc. Var. PDE). We prove exterior \(L^2\)- and \(L^\infty\)-convergence with error estimates, including the logarithmic correction at the critical H\"older exponent. The exterior reconstruction enables the Liouville reduction to the Calder\'on problem for the fractional Schr\"odinger equation. Separated measurements recover the interior potential, after which overlapping measurements determine the exterior potential tail. Together with the recovered exterior conductivity, the tail supplies the exterior Cauchy data for a Tikhonov-regularized finite element least-squares continuation. We prove \(H^s\)-convergence of the transformed conductivity and \(L^2\)-convergence of the interior conductivity, together with logarithmic state-bias and conductivity error estimates for full observation spaces. The admissible reconstructions remain uniformly positive and bounded. We present one- and two-dimensional numerical experiments illustrating the accuracy of the exterior and interior reconstructions and their sensitivity to measurement noise and discretization.
\end{abstract}

\maketitle

\section{Introduction}
The internal structure of a heterogeneous medium leaves a measurable signature in the potential field it generates. The classical Calder\'on problem captures this using the local conductivity equation \(\nabla\!\cdot\!(\gamma\nabla u)=0\), the core model for electrical impedance tomography; see, for instance, \cite{Calderon1980,Uhlmann2009}. Its constitutive law is pointwise: flux depends only on the gradient at the same location. When transport involves long-range jumps or direct interactions between separated points, such locality breaks down, and the conductivity law must encode pairwise exchanges across distance \cite{Covi2020Inverse}. This motivates the fractional conductivity equation
\begin{equation}\label{eq:frac-cond}
    \operatorname{div}_{s}\!\left(\Theta_{\gamma}\nabla^{s}u\right)=0 \;\; \text{in } \Omega,\qquad u=f \;\; \text{in } \Omega_{e}:=\mathbb{R}^{d}\setminus\overline{\Omega},
\end{equation}
where \(d\geq1\), \(0<s<\min\{1,d/2\}\), and \(\Omega\subset\mathbb{R}^{d}\) is open and bounded in one direction. The unknown \(u\) denotes the state of the system, \(f\) is a prescribed exterior value, and \(\gamma\in L^{\infty}(\mathbb{R}^{d})\) is a uniformly positive isotropic conductivity. The corresponding interaction operator is
\begin{equation}
    \Theta_{\gamma}(x,y)
    :=
    \gamma(x)^{1/2}\gamma(y)^{1/2},
    \qquad (x,y)\in\mathbb{R}^{2d},
    \label{eq:Cond_matrix}
\end{equation}
and the fractional gradient is the two-point vector field
\begin{equation}
    \nabla^{s}u(x,y)
    :=
    \sqrt{\frac{C_{d,s}}{2}}\,
    \frac{u(x)-u(y)}
         {|x-y|^{d/2+s+1}}(x-y),
    \qquad (x,y)\in\mathbb{R}^{2d},
    \label{eq:frac_grad}
\end{equation}
where \(C_{d,s}>0\) is the standard fractional-Laplacian normalization constant; see for example \cite{Covi2020Inverse,covi2022global}. The fractional divergence \(\operatorname{div}_{s}\) is defined as the adjoint of \(\nabla^{s}\), and the identity \(\operatorname{div}_{s}\nabla^{s}=(-\Delta)^{s}\) connects this nonlocal vector calculus with the fractional Laplacian. If \(u_f\) denotes the solution with exterior value \(f\), the measurements are represented by the exterior Dirichlet-to-Neumann (DN) map
\begin{equation}\label{eq:sec2-DN}
    \langle\Lambda_{\gamma}f,g\rangle := B_{\gamma}(u_f,g),
    \qquad B_{\gamma}(u,v) := \iint_{\R^{2d}}
    \Theta_{\gamma}(x,y)\nabla^{s}u(x,y)\cdot\nabla^{s}v(x,y)\,dx\,dy,
\end{equation}
where \(\langle \cdot,\cdot \rangle\) is the duality pairing; see Section~\ref{sec:2}. The inverse problem is to recover \(\gamma\), globally or on an exterior set, from full or partial knowledge of these pairings. Here \(\gamma(x)^{1/2}\gamma(y)^{1/2}\) acts as a pairwise conductivity, while the kernel \(|x-y|^{-d-2s}\) controls how the interaction decays with distance. The order \(s\) therefore determines the interaction range, with smaller values giving greater relative weight to distant exchanges. Du et al. \cite{Du2012NonlocalDiffusion,Du2013NonlocalVectorCalculus} developed the nonlocal vector calculus and volume-constrained diffusion framework underlying these two-point operators. Valdinoci \cite{Valdinoci2009LongJump} derived the fractional Laplacian from long-jump random walks, giving the interaction kernel a probabilistic interpretation. Under suitable regularity and normalization assumptions, Covi \cite{Covi2020Inverse} showed that this weighted long-jump model approaches the classical conductivity equation as \(s\uparrow1\).

For the fractional Schr\"odinger equation, Ghosh, Salo, and Uhlmann \cite{Ghosh2020calderon} proved partial-data uniqueness using unique continuation and Runge approximation. Ghosh, R\"uland, Salo, and Uhlmann \cite{Ghosh2020uniqueness} established single-measurement uniqueness and reconstruction; Propositions~6.1 and~6.2 give supported Cauchy stability and instability under their stated geometric assumptions. R\"uland and Salo \cite{RulandSalo2020low} proved low-regularity uniqueness and logarithmic stability, while R\"uland \cite{Ruland2021single} obtained single-measurement logarithmic stability for H\"older potentials supported in a fixed interior subdomain. Covi \cite{Covi2020Inverse} introduced the conductivity equation and Liouville reduction, with uniqueness and reconstruction for exterior-trivial conductivities. Railo and Zimmermann \cite{railo2023fractional} treated domains bounded in one direction. Covi et al.\ \cite{covi2022global} established global uniqueness, exterior reconstruction, and exterior Lipschitz stability for nontrivial exterior conductivities. Railo and Zimmermann \cite{RailoZimmermann2023counterexamples,RailoZimmermann2024low} gave counterexamples for disjoint measurement sets and reduced the regularity requirements. Interior stability and instability are studied in \cite{Covi2024stability,BaersCoviRuland2025instability,CoviKujanpaaRailo2027partial}: the partial-data estimate in \cite{CoviKujanpaaRailo2027partial} is log--log when the conductivity difference is compactly supported away from the measurement set and logarithmic when the conductivities agree throughout the exterior. See \cite{covi2025survey} for a survey. Kow et al.\ \cite{KowNurminenRailo2025Bayesian} proved posterior consistency and rates for noisy discrete single-measurement observations. Rinaldo and Warma \cite{RinaldoWarma2026Reconstruction} studied constructive recovery from localized fractional Cauchy data, including nonlocal Robin inverse problems and numerical continuation. Conforming finite element approximation and singular-integral assembly for the integral fractional Laplacian are studied in \cite{AcostaBorthagaray2017,Acosta2017short}; nonhomogeneous exterior Dirichlet data are treated in \cite{AcostaBorthagarayHeuer2019}. Li \cite{Li2025numerical} developed a finite-difference and Tikhonov reconstruction of the fractional Schr\"odinger potential from a single measurement, while Dwivedi et al.\ \cite{DwivediRailoRuppFCP} developed a Galerkin--Tikhonov reconstruction for that potential problem.

In this paper, we develop a finite element method for reconstructing the conductivity \(\gamma\). The exterior determination formula is explicit and Lipschitz stable but uses ideal concentrating sequences \cite{covi2022global}. Interior recovery has logarithmic-type stability and exponential instability under the respective geometric and admissibility assumptions \cite{Covi2024stability,CoviKujanpaaRailo2027partial}. We develop a finite element reconstruction from noisy, finite-dimensional DN data, combining exterior determination with regularized interior continuation.

Let \(W_0\Subset W\Subset\Omega_e\) be bounded open sets. Write \(W^s(\mathbb R^d)=H^s(\mathbb R^d)\) with energy norm \(\|v\|_{W^s}^2=\|v\|_{L^2(\mathbb R^d)}^2+\|\nabla^sv\|_{L^2(\mathbb R^{2d})}^2\). Mesh-localized exterior inputs normalized in \(W^s(\mathbb R^d)\) provide nodal conductivity samples, which determine a mass-lumped \(H^1\)-Tikhonov reconstruction on \(W_0\). Clipping and local smoothing approximate the Liouville multiplier \(a_\gamma=\gamma^{-1/2}\). The Liouville reduction introduces \(m_\gamma=\sqrt\gamma-1\) and \(q_\gamma=-(-\Delta)^s m_\gamma/(1+m_\gamma)\). Separated measurements recover the Schr\"odinger states and \(Q_\gamma=q_\gamma|_\Omega\) by regularized continuation and residual fitting \cite{Ghosh2020calderon,DwivediRailoRuppFCP}. On an open set \(W_c\Subset W_0\), overlapping measurements recover \(q_\gamma|_{W_c}\), hence \((-\Delta)^s m_\gamma|_{W_c}\). A final regularized Cauchy problem reconstructs \(m_\gamma\), and admissible projection yields the interior conductivity.

The main contributions are as follows; \(h\) denotes the mesh size and \(\delta\) the total data discrepancy.
\begin{enumerate}[label=(\roman*),leftmargin=*,itemsep=0.35em,topsep=0.25em]
\item \textbf{Finite element concentration.} We construct normalized, mesh-localized inputs and quantify their concentration error, including the critical H\"older logarithm and the interpolation contribution.
\item \textbf{Exterior reconstruction.} Mass-lumped \(H^1\)-Tikhonov regularization gives a symmetric positive definite system and \(L^2(W_0)\)- and \(L^\infty(W_0)\)-error bounds separating concentration, approximation, regularization, and data discrepancy.
\item \textbf{Interior reconstruction.} Separated measurements recover the Schr\"odinger states and interior potential, while overlapping measurements recover the exterior tail. Regularized Liouville--Cauchy continuation then reconstructs the conductivity.
\item \textbf{Error propagation and convergence.} We propagate the errors between reconstruction stages and prove convergence under explicit approximation and parameter conditions. Full observation spaces additionally give logarithmic state-bias and conductivity estimates.
\end{enumerate}

The following two theorems summarize the exterior and interior conclusions. We use fitted, shape-regular, quasi-uniform meshes on a bounded polyhedral Lipschitz computational domain \(\Omega_R\) containing \(\overline\Omega\cup\overline W\), with the finite element spaces specified in Sections~\ref{sec3.1} and~\ref{sec:interior-cauchy}.

For the exterior reconstruction, let \(k\ge1\) be the degree used to interpolate the concentrating inputs, let \(\sigma_I>0\) satisfy \(d/2<s+\sigma_I\le k+1\), and choose \(N(h)=\lfloor h^{-\theta_N}\rfloor\), \(0<\theta_N<1\). At each reconstruction node \(x_i\in\overline{W_0}\), \(i=1,\ldots,N_h\), let \(\phi_{i,N(h)}\) be the \(W^s(\mathbb R^d)\)-normalized smooth concentrating profile of Lemma~\ref{lem:uniform-crz-concentrating}; \(N_h\) is the number of reconstruction nodes. For the H\"older exponent \(0<\beta_c\le1\), define
\begin{equation}\label{eq:concentration-mesh-modulus}
\lambda_c:=\theta_N\min\{\beta_c,2s\},\qquad \mathfrak c_h:=\begin{cases}h^{\lambda_c}(1+|\log h|),&\beta_c=2s,\\ h^{\lambda_c},&\beta_c\ne2s,\end{cases}
\end{equation}
and define the ideal-profile concentration error by
\begin{equation}\label{eq:ideal-profile-concentration-error}
\eta_\phi(h):=\max_{1\le i\le N_h}\left|\langle\Lambda_\gamma\phi_{i,N(h)},\phi_{i,N(h)}\rangle-\gamma(x_i)\right|.
\end{equation}

\begin{theorem}[Exterior reconstruction and maximum-norm convergence]\label{cor:linfty-fully-discrete-rate}
Assume \(0<\gamma_0\le\gamma\le\gamma_1<\infty\) almost everywhere in \(\mathbb R^d\), \(\gamma\in C^{0,\beta_c}(\overline W)\), and \(\gamma|_{W_0}\in H^2(W_0)\cap W^{1,\infty}(W_0)\). Assume the interpolation estimates of Assumption~\ref{ass:interp-concentrating-data} and the sample discrepancy bound \eqref{eq:sample-noise-Dih}. Let \(\gamma_h^\delta\) be the mass-lumped Tikhonov minimizer \eqref{eq:def-recon-minimizer}, with parameter \(\alpha>0\). If \(\alpha C_Bh^{-2}\le q<1\) for fixed \(q\), where \(C_B\) is the mesh-family stability constant, independent of \(h\), from Lemma~\ref{lem:linfty-stab}, then, for sufficiently small \(h\),
\begin{equation*}
\|\gamma_h^\delta-\gamma\|_{L^\infty(W_0)}\le C\bigl(\eta_\phi(h)+\mathfrak c_hh^{(1-\theta_N)\sigma_I}+\delta+\alpha h^{-1}+h^{\beta_c}\bigr).
\end{equation*}
The ideal-profile error satisfies \(\eta_\phi(h)\le C\mathfrak c_h\). In particular, if \(0<\alpha(h)\le c_\alpha h^2\) with a fixed constant \(c_\alpha>0\) satisfying \(c_\alpha C_B<1\), then
\begin{equation}\label{eq:linfty_rate_sim}
\|\gamma_h^\delta-\gamma\|_{L^\infty(W_0)}\le C\bigl(\mathfrak c_h+\delta+h^{\beta_c}\bigr).
\end{equation}
Consequently, \(\gamma_h^\delta\to\gamma\) in \(L^\infty(W_0)\) whenever \(h\to0\) and \(\delta(h)\to0\). If \(\delta=\mathcal O(h^{\lambda_\delta})\), \(\lambda_\delta>0\), then the error is \(\mathcal O(\mathfrak c_h+h^{\lambda_\delta})\). The constants are independent of \(h\) and \(\delta\), but may depend on the fixed degree \(k\) and interpolation gain \(\sigma_I\).
\end{theorem}

For the interior reconstruction, let \(V_h^R\) be the conforming degree-\(k\) finite element space with zero boundary values on \(\Omega_R\), extended by zero to \(\mathbb R^d\), and let \(\Pi_h^c\) be the \(L^2(W_c)\)-orthogonal projection onto the full fitted Cauchy observation space \(V_h^c\) defined in Section~\ref{sec:interior-cauchy}. The clipped potential \(\widehat Q_h^\Omega\), exterior value \(m_h^W\), and recovered fractional-Laplacian datum \(\tau_h^c\) approximate \(Q_\gamma\), \(m_\gamma|_{W_c}\), and \(\tau_\gamma:=(-\Delta)^sm_\gamma|_{W_c}\), respectively. The constructions are given in \eqref{eq:tail-clipped-potential}, \eqref{eq:ah-mh}, and \eqref{eq:h-c-definition}. With fixed positive Cauchy residual weights \(\alpha_v\) and \(\alpha_\tau\), define
\begin{equation}\label{eq:final-combined-input-error}
\begin{aligned}E_h^{\rm in}:=\|\widehat Q_h^\Omega-Q_\gamma\|_{L^{d/(2s)}(\Omega)}+\sqrt{\alpha_v}\,\|\Pi_h^c((m_h^W-m_\gamma)|_{W_c})\|_{L^2(W_c)}+\sqrt{\alpha_\tau}\,\|\tau_h^c-\Pi_h^c\tau_\gamma\|_{L^2(W_c)}.\end{aligned}
\end{equation}
Let \(m_h\) minimize the Cauchy functional \eqref{eq:m-cauchy-functional} with coercive parameter \(\alpha_m>0\), and let \(\gamma_h^\Omega\) be the clipped and squared reconstruction \eqref{eq:gamma-interior-reconstruction}. 

\begin{theorem}[Interior convergence and logarithmic error estimate]\label{cor:final-interior-logarithmic}
Suppose that Assumption~\ref{ass:liouville-regularity} and the approximation properties \eqref{eq:cauchy-space-density}, \eqref{eq:cauchy-projection-consistency}, and \eqref{eq:VR-density} hold, and use the full observation spaces of Proposition~\ref{prop:discrete-cauchy-logarithmic}. Choose \(m_{\gamma,h}\in V_h^R\) and set
\[
a_{m,h}:=\|m_{\gamma,h}-m_\gamma\|_{W^s(\mathbb R^d)},\qquad \eta_{{\rm C},h}:=a_{m,h}+E_h^{\rm in}+\sqrt{\alpha_m}+h^s.
\]
If \(\alpha_m\to0\), \(a_{m,h}^2/\alpha_m\to0\), and \(E_h^{\rm in}/\sqrt{\alpha_m}\to0\) as \(h\to0\), then \(m_h\to m_\gamma\) in \(W^s(\mathbb R^d)\) and \(\gamma_h^\Omega\to\gamma\) in \(L^2(\Omega)\). Moreover, for sufficiently small \(h\),
\begin{equation}\label{eq:final-interior-logarithmic}
\|\gamma_h^\Omega-\gamma\|_{L^2(\Omega)}\le C[\log(2+\eta_{{\rm C},h}^{-1})]^{-\sigma}.
\end{equation}
The constants \(C,\sigma>0\) are independent of \(h,\delta,\alpha_m\) along the parameter coupling; the logarithmic exponent \(\sigma\) is distinct from the interpolation gain \(\sigma_I\).
\end{theorem}

The exterior and interior proofs are connected through the reconstructed Liouville multiplier and exterior Cauchy value. Proposition~\ref{prop:mesh-resolved-concentration} and Corollary~\ref{cor:eta-c-explicit-rate} control concentration; the sample estimate and maximum-norm stability then prove Theorem~\ref{cor:linfty-fully-discrete-rate}. Clipping and local smoothing give the multiplier consistency of Lemma~\ref{lem:discrete-multiplier-consistency}. Separated measurements yield the state and interior-potential estimates, while overlapping measurements yield the exterior-tail estimate. The potential, tail, and exterior value errors form \(E_h^{\rm in}\), which enters the final Cauchy comparison and logarithmic estimate. The diagram summarizes the dependencies under the stated approximation, source, visibility, and parameter conditions.
\begin{center}
\begin{tikzpicture}[>=stealth,every node/.style={font=\small,align=center,text width=3.4cm,inner sep=3pt},every path/.style={line width=0.45pt}]
\node (sample) at (-4.3,1.35) {Sample error\\Theorem~\ref{thm:measured-sample-error-rate}}; \node (maximum) at (0,1.35) {Maximum-norm\\stability\\Theorem~\ref{thm:linfty-regularized-reconstruction}}; \node (exterior) at (0,0) {\textbf{Exterior}\\\textbf{reconstruction}\\Theorem~\ref{cor:linfty-fully-discrete-rate}}; \node (multiplier) at (4.3,0) {Liouville multiplier\\Lemma~\ref{lem:discrete-multiplier-consistency}};
\node (potential) at (4.3,-1.6) {Separated data and potential\\Lemma~\ref{lem:reduced-separated-data-consistency}, Theorem~\ref{thm:Q-whole-domain-error-regularized}}; \node (tail) at (4.3,-3.2) {Exterior Liouville tail\\Theorem~\ref{thm:tail-reconstruction-error}}; \node (input) at (0,-3.2) {Combined input error\\\(E_h^{\rm in}\), \eqref{eq:final-combined-input-error}};
\node (cauchy) at (-4.3,-1.6) {Cauchy comparison\\and stability\\Theorem~\ref{thm:final-interior-conductivity-error}, Proposition~\ref{prop:discrete-cauchy-logarithmic}}; \node (interior) at (-4.3,-3.2) {\textbf{Interior}\\\textbf{reconstruction}\\Theorem~\ref{cor:final-interior-logarithmic}};
\draw[->] (sample.south east) -- (exterior.north west); \draw[->] (maximum) -- (exterior); \draw[->] (exterior) -- (multiplier); \draw[->] (multiplier) -- (potential); \draw[->] (potential) -- (tail); \draw[->] (potential.south west) -- (input.north east); \draw[->] (tail) -- (input); \draw[->] (exterior) -- node[right,text width=1.9cm,font=\footnotesize] {Exterior value} (input); \draw[->] (input) -- (interior); \draw[->] (cauchy) -- (interior); \draw[->] (multiplier.east) -- (6.35,0) -- node[right,text width=1.55cm,font=\footnotesize] {Overlapping data} (6.35,-3.2) -- (tail.east);
\end{tikzpicture}
\end{center}

The remainder of the paper is organized as follows. Section~\ref{sec:2} recalls the function spaces, forward problem, and exterior energy estimate. Sections~\ref{sec:ext} and~\ref{sec:interior-cauchy} develop the exterior reconstruction and interior Liouville--Cauchy method, respectively. Section~\ref{sec:num} presents the numerical experiments. Appendix~\ref{app:quadratic-regularization} collects the quadratic regularization estimates, and Appendix~\ref{app:calibrated-dn} treats calibrated and extrapolated DN samples.

\section{Preliminaries and continuous background}\label{sec:2}

Throughout the discrete analysis, \(\Omega\subset\R^d\) is a bounded Lipschitz domain, \(0<s<\min\{1,d/2\}\), and \(W\Subset\Omega_e:=\R^d\setminus\overline\Omega\) is a bounded nonempty open set. An open set is said to be bounded in one direction if, after a rigid Euclidean motion, it is contained in a cylinder \(\R^{d-k}\times\omega\) for some \(1\le k\le d\) and some bounded open set \(\omega\subset\R^k\); see \cite[Def. 2.1]{railo2023fractional}. In particular, every bounded domain has this property. The forward and exterior-determination statements below remain valid in this larger geometric class. We use the low-regularity theory of \cite{RailoZimmermann2024low}, the fractional Poincar\'e inequalities of \cite{railo2023fractional}, and the concentrating-sequence construction of \cite{covi2022global}.

\subsection{Function spaces and fractional operators}
\label{subsec:function-spaces}
Let \(\mathcal S(\R^d)\) and \(\mathcal S'(\R^d)\) denote the Schwartz functions and tempered distributions, respectively, and let \(\mathcal D'(U)\) denote the distributions on an open set \(U\subset\R^d\). With \(\langle\xi\rangle:=(1+|\xi|^2)^{1/2}\), the Bessel-potential space of order \(t\in\R\) and integrability \(1<p<\infty\) is
\begin{equation*} 
H^{t,p}(\R^d) := \left\{u\in\mathcal S'(\R^d):\mathcal F^{-1}\!\left(\langle\xi\rangle^t\widehat u\right)\in L^p(\R^d)\right\}, \end{equation*}
equipped with the norm $\lVert u\rVert_{H^{t,p}(\mathbb R^d)}:=\lVert \mathcal F^{-1}(\langle\xi\rangle^t\widehat u)\rVert_{L^p(\mathbb R^d)}$. We write \(H^t:=H^{t,2}\), so that \( H^{-s}(\R^d)=\bigl(H^s(\R^d)\bigr)^*\). For an open set \(U\subset\R^d\), we use the restriction and zero-exterior spaces
\[
H^{t,p}(U)
:= \{u|_U:u\in H^{t,p}(\R^d)\}, \qquad \widetilde H^{t,p}(U)
:= \overline{C_c^\infty(U)}^{\,H^{t,p}(\R^d)},
\]
where \(H^{t,p}(U)\) carries the quotient norm. The integrability index is
omitted when \(p=2\). These conventions follow \cite[Sec.~2]{RailoZimmermann2024low}; in particular, \(H^{s,d/s}(\R^d)\), used later in the Liouville reduction, is a Bessel-potential space.

For \(0<\theta<1\) and \(1\le p<\infty\), the Slobodeckij space \(W^{\theta,p}(U)\) is equipped with
\begin{equation*} 
\|v\|_{W^{\theta,p}(U)}^p := \|v\|_{L^p(U)}^p+\int_U\int_U\frac{|v(x)-v(y)|^p}{|x-y|^{d+\theta p}}\,dx\,dy. 
\end{equation*}
For \(\theta=1\), we use the usual first-order Sobolev space. Thus, the spaces \(W^{m,p}\) appearing in the approximation assumptions below are Slobodeckij spaces when \(0<m<1\). The Bessel-potential and Slobodeckij scales generally do not coincide when \(p\ne2\), and this notation should also be distinguished from the energy-normalized space \(W^s(\R^d)\) fixed in the introduction; see \cite{Di2012Hitchhiker,RailoZimmermann2024low}.

For $\sigma\ge0$ and $u\in\mathcal S(\mathbb R^d)$, define $(-\Delta)^\sigma u:=\mathcal F^{-1}(\vert \xi\vert ^{2\sigma}\widehat u)\in\mathcal S'(\mathbb R^d)$. In particular, $(-\Delta)^s$ extends continuously from $H^s(\mathbb R^d)$ to $H^{-s}(\mathbb R^d)$. With the normalization in
\eqref{eq:frac_grad}, we have
\[
    \|\nabla^s u\|_{L^2(\R^{2d})}
    =  \|(-\Delta)^{s/2}u\|_{L^2(\R^d)}, \qquad u\in H^s(\R^d).
\]
The energy norm fixed in the introduction is induced by the inner product
\begin{equation}\label{eq:Ws-norm-convention} (u,v)_{W^s(\R^d)} := (u,v)_{L^2(\R^d)}+\bigl((-\Delta)^{s/2}u,(-\Delta)^{s/2}v\bigr)_{L^2(\R^d)}. \end{equation}
The fractional divergence is the adjoint map \(\operatorname{div}_s:
    L^2(\R^{2d};\R^d)\longrightarrow H^{-s}(\R^d)\) defined by
\[
    \langle\operatorname{div}_s F,v\rangle_{H^{-s},H^s} = (F,\nabla^s v)_{L^2(\R^{2d})},
\]
and \(\operatorname{div}_s\nabla^s v=(-\Delta)^s v\) in \(H^{-s}(\R^d)\). See \cite{Covi2020Inverse,covi2022global} for these identities.

For the exterior value problem, we set
\begin{equation*} 
X := H^s(\R^d)/\widetilde H^s(\Omega), \qquad \|[f]\|_X := \inf_{\phi\in\widetilde H^s(\Omega)}\|f-\phi\|_{H^s(\R^d)}. 
\end{equation*}
Its dual is denoted by \(X^*\), with duality pairing \(\langle\cdot,\cdot\rangle_{X^*, X}\). We write \(\mathcal L(E,F)\) for the space of bounded linear maps between Banach spaces \(E\) and \(F\).

The following analytic facts will be used repeatedly. Since \(\Omega\) is
bounded in one direction, the fractional Poincar\'e inequality gives
\begin{equation}\label{eq:sec2-fractional-poincare}
\|v\|_{L^2(\R^d)} \le C_P(\Omega) \|(-\Delta)^{s/2}v\|_{L^2(\R^d)},
\qquad v\in\widetilde H^s(\Omega),
\end{equation}
see \cite[Thm.~2.2]{railo2023fractional}. Moreover, with \(p_s^* := \frac{2d}{d-2s}\), the fractional Sobolev embedding yields \( H^s(\R^d)\hookrightarrow L^{p_s^*}(\R^d)\). If \(U\) is bounded and Lipschitz, then \( \widetilde H^s(U)\hookrightarrow L^2(U)\) is compact; see \cite[Secs.~6--7]{Di2012Hitchhiker}. We also use the unique continuation property
\begin{equation}\label{eq:sec2-ucp}
w=(-\Delta)^s w=0 \quad\text{in a nonempty open set}
\qquad\Longrightarrow\qquad w\equiv0 \quad\text{in }\R^d,
\end{equation}
for \(w\in H^s(\R^d)\); see \cite[Thm.~1.2]{Ghosh2020calderon}.

A distribution \(q\in\mathcal D'(\R^d)\) belongs to the multiplier space \(\mathcal M(H^s\to H^{-s})\) if
\begin{equation*}
\|q\|_{s,-s} := \sup_{\substack{u,v\in C_c^\infty(\R^d)\\ \|u\|_{H^s(\R^d)}=\|v\|_{H^s(\R^d)}=1}}|\langle q,uv\rangle|<\infty. \end{equation*}
Then \((u,v)\longmapsto\langle qu,v\rangle\) extends uniquely to a bounded bilinear form on \(H^s(\R^d)\times H^s(\R^d)\). For an open set \(U\subset\R^d\), we also use the strong local multiplier space
\begin{equation*} 
\mathcal M_s(U) := \left\{b:U\to\R\text{ measurable}:\|b\|_{\mathcal M_s(U)} := \sup_{0\ne\phi\in\widetilde H^s(U)}\frac{\|b\phi\|_{H^s(\R^d)}}{\|\phi\|_{H^s(\R^d)}}<\infty\right\}, 
\end{equation*}
where multiplication is understood on \(U\), and \(b\phi\) is extended by
zero outside \(U\). 

These two multiplier notions play different roles. The Liouville
potential is interpreted in \(\mathcal M(H^s\to H^{-s})\), whereas the
exterior factor \(\gamma^{-1/2}\) is used as an element of
\(\mathcal M_s(U)\). In particular, if \(\gamma\in L^\infty(\R^d)\) is uniformly positive and \(m_\gamma=\sqrt\gamma-1\in H^{s,d/s}(\R^d)\), then the fractional Liouville reduction and its multiplier potential are well defined; see \cite[Secs.~2.3 and 3.2]{RailoZimmermann2024low}. This additional Bessel-potential regularity is not required for the forward conductivity problem below.

\subsection{The forward problem and the exterior Dirichlet-to-Neumann map}
\label{subsec:forward-DN}
Let \(\gamma\in L^\infty(\R^d)\) satisfy \(
0<\gamma_0 \le \gamma(x) \le \gamma_1 <\infty\) for almost every \(x\in\R^d\).
We associate with \(\gamma\) the scalar coefficient kernel $\Theta_\gamma$
\eqref{eq:Cond_matrix} and define the bilinear energy form by \eqref{eq:sec2-DN}.
Equivalently,
\begin{equation*} 
B_\gamma(u,v) = \frac{C_{d,s}}{2}\iint_{\R^{2d}}\gamma(x)^{1/2}\gamma(y)^{1/2}\frac{(u(x)-u(y))(v(x)-v(y))}{|x-y|^{d+2s}}\,dx\,dy.
\end{equation*}
The form is continuous on \(H^s(\R^d)\), and
\eqref{eq:sec2-fractional-poincare} gives its coercivity on
\(\widetilde H^s(\Omega)\).

Given \([f]\in X\), choose a representative \(f\in H^s(\R^d)\). A
function \(u_f\in H^s(\R^d)\) is a weak solution of
\[
\operatorname{div}_s(\Theta_\gamma\nabla^s u)=0
\quad\text{in }\Omega, \qquad u=f
\quad\text{in }\Omega_e,
\]
if
\begin{equation*} 
u_f-f\in\widetilde H^s(\Omega), \qquad B_\gamma(u_f,\phi)=0\quad\text{for every }\phi\in\widetilde H^s(\Omega). 
\end{equation*}
The solution depends only on the class \([f]\); as usual, we suppress brackets when no ambiguity arises. The exterior Dirichlet-to-Neumann map is \eqref{eq:sec2-DN}.  Since \(u_f-f\in\widetilde H^s(\Omega)\), the weak equation also gives the energy identity
\begin{equation}\label{eq:sec2-DN-energy-identity}
\langle\Lambda_\gamma[f],[f]\rangle_{X^*, X} = B_\gamma(u_f,f)
= B_\gamma(u_f,u_f).
\end{equation}
See \cite[Def.~2.7 and Lem.~2.8]{RailoZimmermann2024low}.

\begin{lemma}[Well-posedness, DN map, and exterior energy estimate]
\label{lem:crz-wellposed-energy}
Let \(\Omega\subset\R^d\) be open and bounded in one direction, let
\(0<s<1\), and assume \( 0<\gamma_0 \le \gamma(x) \le \gamma_1 <\infty\) for almost every \(x\in\R^d\). Then, for every \(f\in X\), the fractional conductivity exterior value problem has a unique weak solution. The map \(\Lambda_\gamma:X\to X^*\) is bounded, self-adjoint, and nonnegative.

Let \(U\Subset\Omega_e\) be bounded. For every
\(f\in\widetilde H^s(U)\), the corresponding solution satisfies
\begin{equation}\label{eq:crz-exterior-energy}
\|u_f-f\|_{H^s(\R^d)} \le
C_E\|f\|_{L^2(U)}.
\end{equation}
The constant \(C_E\) is independent of \(f\) and depends only on
\(d\), \(s\), \(\gamma_0\),
\(\|\gamma\|_{L^\infty(\Omega\cup U)}\), \(|U|\),
\(\operatorname{dist}(U,\Omega)\), and the fractional Poincar\'e
constant of \(\Omega\).
\end{lemma}

The existence, uniqueness, and boundedness of the DN map follow from \cite[Lem.~2.8]{RailoZimmermann2024low}. The estimate \eqref{eq:crz-exterior-energy} is proved for \(C_c^\infty(U)\) in \cite[Lem.~3.1]{RailoZimmermann2024low} and extends to \(\widetilde H^s(U)\) by density and continuous dependence. The self-adjointness and nonnegativity of \(\Lambda_\gamma\) follow from the symmetry of \(B_\gamma\) and the energy identity \eqref{eq:sec2-DN-energy-identity}. Note that the condition \(m_\gamma\in H^{s,d/s}(\R^d)\) is needed for the Liouville-transformed Schr\"odinger problem used in Section~\ref{sec:interior-cauchy}; see \cite[Rem.~2.9]{RailoZimmermann2024low}. Although Lemma~\ref{lem:crz-wellposed-energy} is stated in the full forward well-posedness range \(0<s<1\), all subsequent exterior-determination, Liouville-transform, and numerical arguments are understood under the standing assumption \(0<s<\min\{1,d/2\}\) fixed in the introduction.

\section{Discrete method and regularized reconstruction on the exterior set}\label{sec:ext}

\subsection{Finite element spaces and localized exterior data}\label{sec3.1}

Let \(\Omega\subset\R^d\) be a bounded Lipschitz domain, and let \(W\subset\subset\Omega_e:=\R^d\setminus\overline\Omega\) be a bounded nonempty open set. We assume access to the restricted DN pairing $(f,g)\mapsto\langle\Lambda_\gamma f,g\rangle$ for $f,g\in\widetilde H^s(W)$. 

We fix an open reconstruction set \( W_0\subset\subset W\). The inverse step reconstructs \(\gamma\) on \(W_0\). The inclusion $W_0\Subset W$ gives $\operatorname{dist}(\overline{W_0},\mathbb R^d\setminus W)>0$; this buffer permits localized exterior inputs centered at the nodes of $\overline{W_0}$ to have support in $W$, uniformly in the node. We set
\begin{equation*} 
d_W:=\frac12\operatorname{dist}(\overline{W_0},\R^d\setminus W)>0. \end{equation*}
Choose \(R>0\) such that
\[
\overline\Omega\cup \overline W \subset \Omega_R:=(-R,R)^d.
\]
For the finite element construction, we assume that \(\Omega\), \(W\), and \(W_0\) are polygonal/polyhedral Lipschitz sets. The cube $\Omega_R=(-R,R)^d$ is chosen for definiteness. The construction also applies to a bounded polyhedral Lipschitz computational domain $\Omega_R$ with $\overline\Omega\cup\overline W\Subset\Omega_R$, provided that the fitted triangulations used below exist.

Let \(\mathcal T_h\) be a shape-regular, quasi-uniform simplicial triangulation of
\(\Omega_R\), fitted to \(\Omega\), \(W\), and \(W_0\). Let
\[
X_h^k(\Omega_R)
:=
\bigl\{
v_h\in C^0(\overline{\Omega_R}) : v_h|_K\in\mathbb P_k(K)
\ \text{for all }K\in\mathcal T_h
\bigr\},
\]
where \(\mathbb P_k(K)\) is the space of polynomials of degree at most \(k\) on \(K\), with \(k\geq1\).
We define the interior space
\[
V_h^\Omega := \bigl\{v_h\in X_h^k(\Omega_R): v_h=0 \text{ on }\Omega_R\setminus\Omega
\bigr\}.
\]
Equivalently, the zero extension of every \(v_h\in V_h^\Omega\) belongs to \(H_0^1(\Omega)\). Hence
\(
V_h^\Omega\subset \widetilde H^s(\Omega)\), \(0<s<\min\{1,d/2\}\).

Let \(k\ge1\) be fixed, and let \(\mathcal Z_h\) denote the set of
Lagrange nodes of \(X_h^k(\Omega_R)\). For \(z\in\mathcal Z_h\), let
\(\zeta_z\in X_h^k(\Omega_R)\) be the associated nodal basis function, so that \(\zeta_z(z')=\delta_{zz'}\) for all \(z,z'\in\mathcal Z_h\). We define the exterior node set \( \mathcal Z_h^W := \bigl\{
z\in\mathcal Z_h: \operatorname{supp}(\zeta_z)\Subset W \bigr\}\),
and set
\[
V_h^W
:=
\operatorname{span}\{\zeta_z:z\in\mathcal Z_h^W\}
\subset X_h^k(\Omega_R).
\]
Thus every function in \(V_h^W\), extended by zero to \(\mathbb R^d\), is compactly supported in \(W\). Let
\[
\mathcal N_h^0=\{x_i\}_{i=1}^{N_h}\subset\overline{W_0}
\]
be the reconstruction nodes. Here \(N_h\) is the number of reconstruction nodes in the exterior. For each reconstruction node \(x_i\in\mathcal N_h^0\), we use the exterior concentrating functions from the continuous reconstruction formula. The construction below is a uniform version, with respect to the center \(x\in\overline{W_0}\), of the functions constructed in \cite[Lem.~5.5, Rem.~5.6 and App.~A]{covi2022global}.

\begin{lemma}[Uniform exterior concentrating functions]
\label{lem:uniform-crz-concentrating}
Let \(W_0\Subset W\Subset\Omega_e\) and \( d_W>0\). Fix \(M\in\mathbb N\). Then there exist constants \(C>0\) and \(N_\ast\in\mathbb N\) such that, for every \(x\in\overline{W_0}\) and every integer \(N\ge N_\ast\), there is a function \( \phi_{x,N}\in C_c^\infty(W)\) satisfying
\begin{equation*} 
\|\phi_{x,N}\|_{W^s(\mathbb R^d)}=1, 
\end{equation*}
and
\begin{equation}\label{eq:phi-support-rate}
\operatorname{supp}(\phi_{x,N})
\subset B(x,CN^{-1})
\subset W .
\end{equation}
Moreover, for every fixed real number \(\rho\ge -M\), there exists a constant \(C_\rho>0\), independent of \(x\) and \(N\), such that
\begin{equation}\label{eq:phi-Hst-rate}
\|\phi_{x,N}\|_{H^{s+\rho}(\mathbb R^d)}
\le C_\rho N^\rho .
\end{equation}
In particular,
\begin{equation}\label{eq:phi-L2-rate}
\|\phi_{x,N}\|_{L^2(\mathbb R^d)} \le C N^{-s}.
\end{equation}
\end{lemma}

\begin{proof}
Let \(\eta_N\) be the \(W^s(\mathbb R^d)\)-normalized model concentrating sequence constructed in \cite[Lem.~5.5, Rem.~5.6 and App.~A]{covi2022global}. That construction uses a tensor-product profile generated by a nonzero \(\psi\in C_c^\infty(-1,1)\) satisfying \(\int_{-1}^1 t^j\psi(t)\,dt=0\) for \(j=0,\ldots,M-1\); these cancellations give the negative-order estimates below. Then
\[
    \|\eta_N\|_{W^s(\mathbb R^d)}=1,\qquad
    \operatorname{supp}(\eta_N)\subset B(0,CN^{-1}),
    \qquad \|\eta_N\|_{H^{s+\rho}(\mathbb R^d)}\le C_\rho N^\rho
\]
for every fixed \(\rho\ge -M\). Choose \(N_\ast\) so large that \(CN^{-1}\le d_W\) for all \(N\ge N_\ast\), and set
\[
    \phi_{x,N}(y):=\eta_N(y-x),\qquad x\in\overline{W_0}.
\]
Since \(\operatorname{dist}(\overline{W_0},\mathbb R^d\setminus W)=2d_W\), we have \(\operatorname{supp}(\phi_{x,N})\subset B(x,d_W)\subset W\). The estimates follow from translation invariance. Taking \(\rho=-s\) gives the \(L^2\)-bound.
\end{proof}

Let \( I_h:C(\overline{\Omega_R})\to X_h^k(\Omega_R)\) denote the global Lagrange interpolation operator on \(\mathcal T_h\).
For exterior functions we use the localized nodal interpolant
\begin{equation*}
I_h^W\varphi := \sum_{z\in\mathcal Z_h^W}\varphi(z)\zeta_z. 
\end{equation*}
Here \(\varphi\) is restricted to \(\Omega_R\), and exterior functions supported in \(W\) are identified with their zero extensions to \(\Omega_R\). By construction, \(I_h^W\varphi\in V_h^W\) and \( \operatorname{supp}(I_h^W\varphi)\Subset W \). We shall apply \(I_h^W\) only to functions whose supports stay a fixed positive distance away from \(\mathbb R^d\setminus W\). More precisely, by shape regularity there is a constant \(c_\zeta>0\), independent of \(h\), such that \(
\operatorname{supp}\zeta_z\subset B(z,c_\zeta h)\) for all  \(z\in\mathcal Z_h \). Hence, if
\[
\operatorname{dist}(\operatorname{supp}\varphi,\mathbb R^d\setminus W) > c_\zeta h,
\]
then every node \(z\) with \(\varphi(z)\ne0\) belongs to \(\mathcal Z_h^W\). Consequently,
\begin{equation*}
I_h^W\varphi=I_h\varphi. 
\end{equation*}
For the concentrating functions \(\phi_{x,N}\) from Lemma~\ref{lem:uniform-crz-concentrating}, after increasing \(N_\ast\) if necessary, their supports are contained in a fixed compact subset of \(W\), uniformly for \(x\in\overline{W_0}\) and \(N\ge N_\ast\). Therefore \( I_h^W\phi_{x,N}=I_h\phi_{x,N}\)
for all \(x\in\overline{W_0}\), all \(N\ge N_\ast\), and all sufficiently small \(h\).

\begin{assumption}[Interpolation estimates for concentrating exterior data] \label{ass:interp-concentrating-data}
Let \(k\ge1\) be the polynomial degree in \(X_h^k(\Omega_R)\). Choose \(\sigma_I>0\) such that \(d/2 < s+\sigma_I\le k+1 \).
 We assume that the localized interpolant \(I_h^W\) satisfies the following estimates for the concentrating functions
\(\phi_{x,N}\) of Lemma~\ref{lem:uniform-crz-concentrating}:
\begin{align}
\|I_h^W\phi_{x,N}-\phi_{x,N}\|_{H^s(\mathbb R^d)}
&\le
C_I h^{\sigma_I}
\|\phi_{x,N}\|_{H^{s+\sigma_I}(\mathbb R^d)},
\label{eq:interp-concentrating-Hs}
\\
\|I_h^W\phi_{x,N}-\phi_{x,N}\|_{L^2(\mathbb R^d)}
&\le
C_I h^{s+\sigma_I}
\|\phi_{x,N}\|_{H^{s+\sigma_I}(\mathbb R^d)}.
\label{eq:interp-concentrating-L2}
\end{align}
The constant \(C_I\) is independent of \(h\), \(x\in\overline{W_0}\), and
\(N\ge N_\ast\).
\end{assumption}
The condition \(d/2 < s+\sigma_I\le k+1 \) is required only for the nodal interpolation of the concentrating exterior data. For functions supported a fixed positive distance away from \(\mathbb R^d\setminus W\), \eqref{eq:interp-concentrating-Hs} and \eqref{eq:interp-concentrating-L2} are the standard Lagrange interpolation estimates on shape-regular quasi-uniform meshes; see, for example, \cite[Ch.~3--4]{ciarlet1978finite} and \cite[Ch.~4]{brenner2008mathematical}. The passage to fractional Sobolev orders follows by interpolation of Sobolev spaces; see also \cite{Di2012Hitchhiker}.

We choose an integer-valued concentration index \(N(h)\) such that
\begin{equation}\label{eq:concentration-index-choice}
N(h)\to\infty, \qquad hN(h)\to0 \qquad\text{as }h\to0.
\end{equation}
For each \(x_i\in\mathcal N_h^0\), define
\begin{equation}\label{eq:def-gih}
g_{i,h} := \frac{I_h^W\phi_{i,N(h)}}
{\|I_h^W\phi_{i,N(h)}\|_{W^s(\R^d)}}, \qquad \text{ where}\quad \phi_{i,N(h)}:= \phi_{x_i,N(h)}.
\end{equation}

\begin{proposition}
\label{prop:mesh-resolved-concentration}
Assume the hypotheses of Lemma \ref{lem:uniform-crz-concentrating}, Assumption \ref{ass:interp-concentrating-data}, and \eqref{eq:concentration-index-choice}. Then, for all sufficiently small \(h\), the function \(g_{i,h}\) in \eqref{eq:def-gih} is well defined, belongs to \(V_h^W\), and satisfies
\[
\|g_{i,h}\|_{W^s(\R^d)}=1.
\]
Define
\begin{equation}\label{eq:def-rh-ellh}
r_h := \max_{1\le i\le N_h} \sup_{z\in\supp(g_{i,h})}|z-x_i|,
\qquad \ell_h := \max_{1\le i\le N_h} \|g_{i,h}\|_{L^2(\R^d)}.
\end{equation}
Then
\begin{align}
r_h &\le C\bigl(N(h)^{-1}+h\bigr), \label{eq:rh-rate} \\ \ell_h &\le C\left(
N(h)^{-s} + h^{s+\sigma_I}N(h)^{\sigma_I} \right).
\label{eq:ellh-rate}
\end{align}
The constants may depend on \(d,s,k,\sigma_I,W_0,W\) and on the shape-regularity of the mesh family, but are independent of \(h\), \(i\), \(N_h\), and \(N(h)\).
\end{proposition}

\begin{proof}
By \eqref{eq:interp-concentrating-Hs} and \eqref{eq:phi-Hst-rate}, we have
\[
\|I_h^W\phi_{i,N(h)}-\phi_{i,N(h)}\|_{H^s(\R^d)}
\le C h^{\sigma_I}N(h)^{\sigma_I} = C(hN(h))^{\sigma_I}.
\]
Since \(hN(h)\to0\) and \(\|\phi_{i,N(h)}\|_{W^s(\R^d)}=1\), the denominator in
\eqref{eq:def-gih} is bounded above and below by positive constants, uniformly in \(i\), for all sufficiently small \(h\). Hence, \(g_{i,h}\) is well defined and
\[
\|g_{i,h}\|_{W^s(\R^d)}=1.
\]
For the \(L^2\)-bound, the uniform lower bound on the denominator in \eqref{eq:def-gih} gives
\[
\|g_{i,h}\|_{L^2(\mathbb R^d)} \le C\|I_h^W\phi_{i,N(h)}\|_{L^2(\mathbb R^d)}.
\]
By the triangle inequality,
\[
\|I_h^W\phi_{i,N(h)}\|_{L^2(\mathbb R^d)} \le \|\phi_{i,N(h)}\|_{L^2(\mathbb R^d)}
+ \|I_h^W\phi_{i,N(h)}-\phi_{i,N(h)}\|_{L^2(\mathbb R^d)}.
\]
Using \eqref{eq:phi-L2-rate}, \eqref{eq:interp-concentrating-L2}, and
\eqref{eq:phi-Hst-rate} with \(\rho=\sigma_I\), we obtain
\[
\|g_{i,h}\|_{L^2(\mathbb R^d)}
\le
C\left(
N(h)^{-s}
+
h^{s+\sigma_I}N(h)^{\sigma_I}
\right).
\]
Taking the maximum over \(i=1,\dots,N_h\) proves \eqref{eq:ellh-rate}. For the support estimate, locality of the nodal interpolant gives
\[
\supp(I_h^W\phi_{i,N(h)})
\subset
\bigcup\{K\in\mathcal T_h:K\cap\supp(\phi_{i,N(h)})\ne\emptyset\}.
\]
Using \eqref{eq:phi-support-rate} and the shape regularity of the mesh, we obtain
\[
\supp(g_{i,h})
\subset
B(x_i,CN(h)^{-1}+Ch),
\]
which proves \eqref{eq:rh-rate}.
\end{proof}

\subsection{Measured DN map}
\label{subsec:measured-dn}

For the functions $g_{i,h}\in V_h^W\subset\widetilde H^s(W)$ defined above, set
\[
D_{i,h} := \big\langle \Lambda_\gamma g_{i,h},g_{i,h}\big\rangle_{X^*, X}.
\]
For notational convenience, we write \(\langle \Lambda_\gamma \cdot,\cdot \rangle_{X^*, X} := \langle \Lambda_\gamma \cdot,\cdot \rangle \). The diagonal measured exterior data are the scalar DN samples \( D_{i,h}^{\delta}\) for \(i=1,\dots,N_h\), which are assumed to satisfy the sample-wise noise bound
\begin{equation}\label{eq:sample-noise-Dih}
|D_{i,h}^{\delta}-D_{i,h}|\le \delta,
\qquad i=1,\dots,N_h .
\end{equation}
The exterior determination requires only the diagonal samples \eqref{eq:sample-noise-Dih}. A stronger finite-dimensional DN noise model is the following. Suppose that the measured data are generated by a bilinear form \(
\mathcal D_h^\delta:V_h^W\times V_h^W\to\mathbb R\)
such that
\begin{equation}\label{eq:dn-noise-discrete}
\big|
\mathcal D_h^\delta(f_h,g_h)
- \langle \Lambda_\gamma f_h,g_h \rangle 
\big| \le \delta \|f_h\|_{W^s(\mathbb R^d)} \|g_h\|_{W^s(\mathbb R^d)} \qquad \forall f_h,g_h\in V_h^W .
\end{equation}
If \( D_{i,h}^{\delta}:=\mathcal D_h^\delta(g_{i,h},g_{i,h})\), then \eqref{eq:sample-noise-Dih} follows from \eqref{eq:dn-noise-discrete}, because \(\|g_{i,h}\|_{W^s(\mathbb R^d)}=1 \).

If \(\mathcal D_h^\delta\) is obtained from a numerical forward solver, condition~\eqref{eq:dn-noise-discrete} is understood in terms of its total discrepancy from the exact DN form. Let $\mathcal D_h^0$ be the noiseless assembled form. Assume that, for all $f_h,g_h\in V_h^W$, $\vert \mathcal D_h^\delta(f_h,g_h)-\mathcal D_h^0(f_h,g_h)\vert \le\delta_m\lVert f_h\rVert_{W^s(\mathbb R^d)}\lVert g_h\rVert_{W^s(\mathbb R^d)}$ and $\vert \mathcal D_h^0(f_h,g_h)-\langle\Lambda_\gamma f_h,g_h\rangle\vert \le\varepsilon_f(h)\lVert f_h\rVert_{W^s(\mathbb R^d)}\lVert g_h\rVert_{W^s(\mathbb R^d)}$. Then \eqref{eq:dn-noise-discrete} holds with $\delta_t(h):=\delta_m+\varepsilon_f(h)$ in place of $\delta$. The convergence statements require \(\delta_{t}(h)\to0\); below we continue to write \(\delta\) for this total discrepancy, for example see \cite{DwivediRailoRuppFCP}.

\begin{lemma}[Consistency of the exact Galerkin DN form]
\label{lem:exact-galerkin-dn-consistency}
Assume the ellipticity hypotheses of Lemma~\ref{lem:crz-wellposed-energy} and use the full fitted spaces \(V_h^\Omega\) of Section~\ref{sec3.1}. For \(f\in\widetilde H^s(W)\), let \(u_{f,h}\in f+V_h^\Omega\) satisfy \(B_\gamma(u_{f,h},v_h)=0\) for all \(v_h\in V_h^\Omega\), and set \(\mathcal D_h^{\rm G}(f,g):=B_\gamma(u_{f,h},u_{g,h})\). There is a sequence \(\kappa_{G,h}\to0\) such that
\[
 |\mathcal D_h^{\rm G}(f,g)-\langle\Lambda_\gamma f,g\rangle|
 \le C\kappa_{G,h}^2\|f\|_{W^s(\mathbb R^d)}\|g\|_{W^s(\mathbb R^d)},
 \qquad f,g\in\widetilde H^s(W).
\]
\end{lemma}

\begin{proof}
Let \(Sf:=u_f-f\), and let \(P_h\) be the \(B_\gamma\)-orthogonal projection onto \(V_h^\Omega\). Estimate \eqref{eq:crz-exterior-energy} extends \(S\) boundedly from \(L^2(W)\) to \(\widetilde H^s(\Omega)\). The compact embedding \(\widetilde H^s(W)\hookrightarrow L^2(W)\) makes \(S:\widetilde H^s(W)\to\widetilde H^s(\Omega)\) compact. Density of the fitted spaces and uniform boundedness of \(P_h\) imply \(P_hv\to v\) in \(W^s\) and hence $\kappa_{G,h}:=
\sup_{0\ne f\in\widetilde H^s(W)}
\frac{\|(I-P_h)Sf\|_{W^s(\mathbb R^d)}} {\|f\|_{W^s(\mathbb R^d)}}\longrightarrow0$. Galerkin orthogonality gives
\[
 \mathcal D_h^{\rm G}(f,g)-\langle\Lambda_\gamma f,g\rangle
 =B_\gamma((I-P_h)Sf,(I-P_h)Sg),
\]
which proves the estimate.
\end{proof}

If, for every $f_h,g_h\in V_h^W$, $\vert \mathcal D_h^0(f_h,g_h)-\mathcal D_h^{\rm G}(f_h,g_h)\vert \le\varepsilon_{\rm ass}(h)\lVert f_h\rVert_{W^s(\mathbb R^d)}\lVert g_h\rVert_{W^s(\mathbb R^d)}$, then $\varepsilon_f(h)\le C\kappa_{G,h}^2+\varepsilon_{\rm ass}(h)$. Convergence of the total discrepancy therefore also requires vanishing assembly and outer-tail errors.

\subsection{Regularized reconstruction of the conductivity on the exterior domain }

We reconstruct the conductivity on \(W_0\) from the measured DN samples
\(D_{i,h}^{\delta}\), \(i=1,\dots,N_h\), associated with the reconstruction
nodes \(\mathcal N_h^0=\{x_i\}_{i=1}^{N_h}\subset\overline{W_0}\).

We assume that
\begin{equation}\label{eq:gamma-reg-recon}
\gamma\in C^{0,\beta_c}(\overline W),
\qquad \gamma|_{W_0}\in H^2(W_0), \qquad 0<\beta_c\le1, \qquad
0<\gamma_0\le \gamma\le\gamma_1 \quad\text{on }\overline W ,
\end{equation}
where \( [\gamma]_{C^{0,\beta_c}(\overline{W_0})}
:=\sup_{\substack{x,y\in\overline{W_0}\\x\ne y}}
\frac{|\gamma(x)-\gamma(y)|}{|x-y|^{\beta_c}}\).

Let \(\mathcal T_h^{W_0}\) be the mesh induced on \(W_0\) by \(\mathcal T_h\). Define
\[
Z_h
:= \bigl\{ c_h\in C^0(\overline{W_0}) : c_h|_K\in \mathbb P_1(K)
\ \text{for all }K\in\mathcal T_h^{W_0}
\bigr\}
\subset H^1(W_0),
\]
and denote its nodal basis by \(\{\psi_i\}_{i=1}^{N_h}\), so that \(\psi_i(x_j)=\delta_{ij}\). We take \(\mathcal N_h^0\) to be the nodal set of the \(\mathbb P_1\) reconstruction mesh \(\mathcal T_h^{W_0}\). For each node \(x_i\), define the lumped mass weight
\[
\mu_i:=\int_{W_0}\psi_i(x)\,dx = \frac1{d+1}
\sum_{K\in\mathcal T_h^{W_0}:\,x_i\in\overline K}|K|>0.
\]
For \(\mathbf z=(z_1,\dots,z_{N_h})^\top,\,\mathbf{\bar z}=(\bar z_1,\dots,\bar z_{N_h})^\top\in\mathbb R^{N_h}\), we write
\[
\|\mathbf z\|_{\mu,h}^2 := \sum_{i=1}^{N_h}\mu_i |z_i|^2 \qquad \text{ and } \qquad (\mathbf z,\mathbf{\bar z} )_{\mu,h}
:=
\sum_{i=1}^{N_h}\mu_i z_i \bar z_i.
\]

For \(c_h\in Z_h\), we write \( \mathbf c_h:= (c_1,\dots,c_{N_h})^\top=(c_h(x_1),\dots,c_h(x_{N_h}))^\top\in\mathbb R^{N_h} \), and the row-sum lumped mass matrix is spectrally equivalent to the consistent mass matrix; in particular, 
\begin{equation}\label{eq:lumped-L2} 
\|c_h\|_{L^2(W_0)}^2 \le \sum_{i=1}^{N_h}\mu_i c_i^2 \le (d+2)\|c_h\|_{L^2(W_0)}^2, \qquad c_h=\sum_i c_i\psi_i, \qquad \forall c_h\in Z_h,
\end{equation}
see \cite[Lem.~1]{LangerLoescherSteinbachYang2024}.

Let us set
\[
\mathbf d_h^\delta := (D_{1,h}^{\delta},\dots,D_{N_h,h}^{\delta})^\top,
\qquad \mathbf d_h := (D_{1,h},\dots,D_{N_h,h})^\top, \quad 
\text{ and }\quad \boldsymbol\gamma_h^\ast :=(\gamma(x_1),\dots,\gamma(x_{N_h}))^\top.
\]

For a regularization parameter \(\alpha>0\), we define
\begin{equation}\label{eq:def-J-h-lambda}
\mathcal J_{h,\alpha}^\delta(c_h)
:= \frac12 \|\mathbf c_h-\mathbf d_h^\delta\|_{\mu,h}^2
+ \frac{\alpha}{2} \|c_h\|_{H^1(W_0)}^2, \qquad c_h\in Z_h .
\end{equation}
The experiments can also use the calibrated and extrapolated samples from Appendix~\ref{app:calibrated-dn} in the data vector of \eqref{eq:def-J-h-lambda}. In this case, Lemma~\ref{lem:calibrated-extrapolated-sample} bounds their nodal error; Theorems~\ref{thm:recon-error} and~\ref{thm:linfty-regularized-reconstruction} then apply with error \(E_h^\delta\), defined below in \eqref{eq:def-Eh-delta}, replaced by \(\widehat E_h^\delta\) from \eqref{eq:def-Eh-cal-delta}.

Apply Lemma~\ref{lem:abstract-shifted-tikhonov} with $\mathsf X=E=Z_h$ in the $H^1(W_0)$-norm, $\mathsf Y=\mathbb R^{N_h}$ in the norm $\|\cdot\|_{\mu,h}$, $T c_h=(c_h(x_i))_{i=1}^{N_h}$, $b=\mathbf d_h^\delta$, and $x_c=0$. For every $\alpha>0$, the functional \eqref{eq:def-J-h-lambda} has a unique minimizer
\begin{equation}\label{eq:def-recon-minimizer}
\gamma_h^\delta=\operatorname*{argmin}_{c_h\in Z_h}\mathcal J_{h,\alpha}^\delta(c_h),
\end{equation}
characterized by
\[
\sum_{i=1}^{N_h}\mu_i\bigl(\gamma_h^\delta(x_i)-D_{i,h}^\delta\bigr)v_h(x_i)+\alpha(\gamma_h^\delta,v_h)_{H^1(W_0)}=0,\qquad v_h\in Z_h.
\]
Writing $\gamma_h^\delta=\sum_jc_j^\delta\psi_j$, the coefficient vector solves the symmetric positive definite system
\begin{equation}\label{eq:recon-linear-system}
(D_h+\alpha K_h)\mathbf c^\delta=D_h\mathbf d_h^\delta,
\end{equation}
where $D_h=\operatorname{diag}(\mu_1,\ldots,\mu_{N_h})$ and $(K_h)_{ij}=\int_{W_0}(\psi_i\psi_j+\nabla\psi_i\cdot\nabla\psi_j)\,dx$.

\subsection{Error analysis of the reconstructed conductivity}
\label{sec:error}

For each \(i=1,\dots,N_h\), let \(u^e_{g_i}\in H^s(\mathbb R^d)\) denote
the exact weak solution of \eqref{eq:frac-cond} with exterior value
\(g_{i,h}\). Thus
\begin{equation}\label{eq:weakstate-gi}
u^e_{g_i}-g_{i,h}\in \widetilde H^s(\Omega),
\qquad
B_\gamma(u^e_{g_i},\varphi)=0
\quad
\forall \varphi\in \widetilde H^s(\Omega).
\end{equation}
The DN definition gives
\begin{equation*} 
D_{i,h} = B_\gamma(u^e_{g_i},g_{i,h}) = B_\gamma(u^e_{g_i},u^e_{g_i}). \end{equation*}
The last equality follows from
\(u^e_{g_i}-g_{i,h}\in\widetilde H^s(\Omega)\) and the weak equation. We define the continuous concentration error by
\begin{equation*} 
\eta_c(h) := \max_{1\le i\le N_h}|D_{i,h}-\gamma(x_i)| =: \|\mathbf d_h-\boldsymbol\gamma_h^\ast\|_{\ell^\infty}. 
\end{equation*}
We also define the measured sample error by
\begin{equation}\label{eq:def-Eh-delta}
E_h^\delta
:= \max_{1\le i\le N_h} |D_{i,h}^{\delta}-\gamma(x_i)| =: \|\mathbf d_h^\delta - \boldsymbol\gamma_h^\ast\|_{\ell^\infty}.
\end{equation}
Thus, we have \(E_h^\delta \le \eta_c(h)+\delta\). We now estimate the deterministic error \(\eta_c(h)\) produced by the localized concentrating data. We start with the following tail estimate for a cutoff function.

\begin{lemma}[Fractional cutoff tail estimate]
\label{lem:fractional-cutoff-tail}
Let \(x_i\in\overline{W_0}\), let \(0<r\le t/4\) for some $t>0$, and let
\(\chi_{i,t}\in C_c^\infty(B(x_i,t))\) satisfy
\[
0\le \chi_{i,t}\le 1,
\qquad
\chi_{i,t}\equiv 1
\quad\text{on }B(x_i,t/2).
\]
Then
\begin{equation*} 
\sup_{x\in B(x_i,r)}\int_{\R^d}|\nabla^s\chi_{i,t}(x,y)|^2\,dy \le Ct^{-2s}, 
\end{equation*}
where \(C>0\) depends only on \(d\) and \(s\).
\end{lemma}

\begin{proof}
Fix \(x\in B(x_i,r)\). Since \(r\le t/4\), we have 
\(\chi_{i,t}(x)=1\). The integrand vanishes for \(y\in B(x_i,t/2)\). On the remaining set,
\( |x-y|\ge t/4\).
Therefore, using the definition of the fractional gradient and \(0\le \chi_{i,t}\le1\),
\[
\int_{\R^d}|\nabla^s\chi_{i,t}(x,y)|^2\,dy
\le
C\int_{\{|x-y|\ge t/4\}}\frac{dy}{|x-y|^{d+2s}}
\le
C t^{-2s}.
\]
This completes the proof.
\end{proof}

\begin{theorem}[Uniform quantitative concentration]
\label{thm:uniform-concentration}
Assume \( 0<\gamma_0 \le \gamma(x) \le \gamma_1 <\infty\) for almost every \(x\in\R^d\) and \(\sqrt\gamma\in C(\overline W)\). Let \(u^e_{g_i}\) be the exact weak solution satisfying \eqref{eq:weakstate-gi} with exterior data \(g_{i,h}\) defined by \eqref{eq:def-gih}. Let \(r_h\) and \(\ell_h\) be defined by \eqref{eq:def-rh-ellh}. For all sufficiently small \(h\), assume \(4r_h<d_W\). Then for every \(t\in[4r_h,d_W]\),
\begin{equation}\label{eq:uniform-concentration-t}
\eta_c(h)
\le
C\left(
\omega_{\sqrt\gamma}(t)
+
t^{-2s}\ell_h^2
+
\ell_h^2
\right),
\end{equation}
where
\[
\omega_{\sqrt\gamma}(t)
:=
\sup_{\substack{x,y\in\overline W\\ |x-y|\le t}}
|\sqrt\gamma(x)-\sqrt\gamma(y)|.
\]
The constant \(C>0\) is independent of \(h\), \(i\), \(N_h\), and \(t\).
\end{theorem}

\begin{proof}
Fix \(i\in\{1,\dots,N_h\}\), and for a moment, write
\( u_i:=u^e_{g_i}\) and \(g_i:=g_{i,h}\).
We split
\begin{equation*} 
|B_\gamma(u_i,u_i)-\gamma(x_i)| \le |B_\gamma(u_i,u_i)-B_\gamma(g_i,g_i)|+|B_\gamma(g_i,g_i)-\gamma(x_i)|.
\end{equation*}
For the first term, let \(w_i = u_i-g_i \in \widetilde H^s(\Omega)\); then \(B_\gamma(u_i,w_i) =0\) and 
\[
B_\gamma(g_i,g_i) = B_\gamma(u_i - w_i, u_i - w_i) = B_\gamma(u_i, u_i) + B_\gamma(w_i, w_i),
\]
which, together with the boundedness of \(B_\gamma\) and the exterior energy estimate \eqref{eq:crz-exterior-energy}, implies
\begin{equation}\label{eq:eta-c-first-term}
|B_\gamma(u_i,u_i)-B_\gamma(g_i,g_i)| = B_\gamma(w_i, w_i) \leq C \|w_i\|^2_{H^s(\R^d)} = C \|u_i- g_i\|^2_{H^s(\R^d)}\leq C \|g_i\|^2_{L^2(W)} \leq C\ell_h^2.
\end{equation}
We now estimate \(B_\gamma(g_i,g_i)-\gamma(x_i)\). Since \(
\|g_i\|_{W^s(\R^d)}^2 = \|g_i\|_{L^2(\R^d)}^2 + \|\nabla^s g_i\|_{L^2(\R^{2d})}^2
= 1\), we have \( \|\nabla^s g_i\|_{L^2(\R^{2d})}^2 = 1-\|g_i\|_{L^2(\R^d)}^2\). Therefore,
\begin{align}
B_\gamma(g_i,g_i)-\gamma(x_i) = \iint_{\R^{2d}}
\bigl(\gamma(x)^{1/2}\gamma(y)^{1/2}-\gamma(x_i)\bigr)
|\nabla^s g_i(x,y)|^2\,dx\,dy - \gamma(x_i)\|g_i\|_{L^2(\R^d)}^2 .
\label{eq:Bgg-minus-gamma}
\end{align}
The last term in \eqref{eq:Bgg-minus-gamma} is bounded by \(C\ell_h^2\). Since \(\sqrt{\gamma}\) is bounded,
\[
|\gamma(x)^{1/2}\gamma(y)^{1/2}-\gamma(x_i)|
\le
C\bigl(|\gamma(x)^{1/2}-\gamma(x_i)^{1/2}|+|\gamma(y)^{1/2}-\gamma(x_i)^{1/2}|\bigr).
\]
By symmetry, it is enough to estimate
\[
J_i := \iint_{\R^{2d}}
|\gamma(y)^{1/2}-\gamma(x_i)^{1/2}|\,|\nabla^s g_i(x,y)|^2\,dx\,dy.
\]
Fix \(t\in[4r_h,d_W]\). Since \(t\le d_W\), we have \(B(x_i,t)\subset\subset W\). Choose
\(\chi_{i,t}\in C_c^\infty(B(x_i,t))\) such that
\[
0\le \chi_{i,t}\le1,
\qquad
\chi_{i,t}\equiv1
\quad\text{on }B(x_i,t/2).
\]
Since \(\operatorname{supp}(g_i)\subset\overline{B(x_i,r_h)}\) and \(r_h\le t/4\), we have \(g_i\chi_{i,t}=g_i\). Using the product rule for the fractional gradient, we get
\[
|\nabla^s g_i(x,y)|^2 = |\nabla^s (g_i\chi_{i,t})(x,y)|^2 \le
2|\chi_{i,t}(y)|^2|\nabla^s g_i(x,y)|^2 + 2|g_i(x)|^2|\nabla^s\chi_{i,t}(x,y)|^2.
\]
Hence \(J_i\le 2J_{i,1}+2J_{i,2}\), where
\[
J_{i,1} := \iint_{\R^{2d}} |\gamma(y)^{1/2}-\gamma(x_i)^{1/2}|\,|\chi_{i,t}(y)|^2|\nabla^s g_i(x,y)|^2\,dx\,dy
\]
and
\[
J_{i,2} := \iint_{\R^{2d}} |\gamma(y)^{1/2}-\gamma(x_i)^{1/2}|\,|g_i(x)|^2|\nabla^s\chi_{i,t}(x,y)|^2\,dx\,dy.
\]
For \(J_{i,1}\), the cutoff implies \(y\in B(x_i,t)\). Hence \(|\gamma(y)^{1/2}-\gamma(x_i)^{1/2}|\le \omega_{\sqrt\gamma}(t)\), and therefore
\[
J_{i,1} \le
\omega_{\sqrt\gamma}(t)
\|\nabla^s g_i\|_{L^2(\R^{2d})}^2 \le \omega_{\sqrt\gamma}(t).
\]
For \(J_{i,2}\), boundedness of \(\sqrt{\gamma}\), the support of \(g_i\), and
Lemma \ref{lem:fractional-cutoff-tail} give
\[
J_{i,2}
\le
C
\int_{\operatorname{supp}(g_i)}
|g_i(x)|^2
\left(
\int_{\R^d}|\nabla^s\chi_{i,t}(x,y)|^2\,dy
\right)dx \le C t^{-2s}\|g_i\|_{L^2(\R^d)}^2 \le C t^{-2s}\ell_h^2.
\]
Thus
\[
J_i \le C\left( \omega_{\sqrt\gamma}(t)+t^{-2s}\ell_h^2 \right).
\]
The same bound holds for the term containing \(|\gamma(x)^{1/2}-\gamma(x_i)^{1/2}|\). Combining this with
\eqref{eq:Bgg-minus-gamma}, we obtain
\[
|B_\gamma(g_i,g_i)-\gamma(x_i)|
\le
C\left(
\omega_{\sqrt\gamma}(t) + t^{-2s}\ell_h^2 + \ell_h^2
\right).
\]
Taking the maximum over \(i=1,\dots,N_h\) proves \eqref{eq:uniform-concentration-t}.
\end{proof}
Under H\"older regularity, integration over annuli sharpens Theorem~\ref{thm:uniform-concentration}.

\begin{corollary}[H\"older concentration rate]
\label{cor:eta-c-explicit-rate}
Let \(\Omega,W_0,W\) be as in Section~\ref{sec3.1}, assume \(0<\gamma_0\le\gamma\le\gamma_1<\infty\) almost everywhere in \(\mathbb R^d\), and let \(\gamma\in C^{0,\beta_c}(\overline W)\), \(0<\beta_c\le1\). Define
\begin{equation*} 
\Psi_{\beta_c,s}(r):=
\begin{cases} r^{\beta_c},&\beta_c<2s,\\ r^{2s}(1+|\log r|),&\beta_c=2s,\\ r^{2s},&\beta_c>2s. 
\end{cases} 
\end{equation*}
For \(0<r<\min\{1,d_W/4\}\), every \(g\in\widetilde H^s(W)\) with \(\|g\|_{W^s(\mathbb R^d)}=1\) and \(\operatorname{supp}g\subset B(x,r)\), \(x\in\overline{W_0}\), satisfies
\begin{equation}\label{eq:localized-holder-concentration}
 |\langle\Lambda_\gamma g,g\rangle-\gamma(x)|\le C\Psi_{\beta_c,s}(r).
\end{equation}
The same estimate holds with \(\mathcal D_h^{\rm G}(g,g)\) from Lemma~\ref{lem:exact-galerkin-dn-consistency} in place of \(\langle\Lambda_\gamma g,g\rangle\), uniformly in \(h\). Under the hypotheses of Proposition~\ref{prop:mesh-resolved-concentration}, for sufficiently small \(h\),
\begin{equation}\label{eq:eta-c-radius-rate}
 \eta_c(h)\le C\Psi_{\beta_c,s}(r_h).
\end{equation}
For \(N(h)=\lfloor h^{-\theta_N}\rfloor\), \(0<\theta_N<1\), and \(\mathfrak c_h\) defined in \eqref{eq:concentration-mesh-modulus}, we obtain
\begin{equation*}
\eta_c(h)\le C\mathfrak c_h.
\end{equation*}
The constants are independent of \(g,x,r,h\).
\end{corollary}

\begin{proof}
Write \(a=\sqrt\gamma\). Restricting the fractional seminorm to \(B(x,r)\times(B(x,3r)\setminus B(x,2r))\) gives $\|g\|_{L^2(\mathbb R^d)}^2
\le Cr^{2s}\|\nabla^sg\|_{L^2(\mathbb R^{2d})}^2
\le Cr^{2s}$. By \eqref{eq:eta-c-first-term} and \eqref{eq:Bgg-minus-gamma}, it remains to bound
\[
 \iint_{\mathbb R^{2d}}|a(y)a(z)-\gamma(x)|\,|\nabla^sg(y,z)|^2\,dy\,dz.
\]
On \(B(x,2r)^2\), the coefficient factor is at most \(Cr^{\beta_c}\). Elsewhere the integrand is nonzero only if one variable belongs to \(\operatorname{supp}g\) and the other lies outside \(B(x,2r)\). H\"older continuity within distance \(d_W\) of \(x\), global boundedness, and symmetry therefore give
\[
 |\langle\Lambda_\gamma g,g\rangle-\gamma(x)|
 \le C\left[r^{\beta_c}+r^{2s}\left(1+\int_{2r}^{d_W}t^{\beta_c-2s-1}\,dt\right)\right]
 \le C\Psi_{\beta_c,s}(r).
\]
For the Galerkin estimate, let \(w=u_g-g\) and let \(P_h\) be the \(B_\gamma\)-orthogonal projection onto \(V_h^\Omega\). Galerkin orthogonality gives \(\mathcal D_h^{\rm G}(g,g)=B_\gamma(g,g)-B_\gamma(P_hw,P_hw)\), while \(B_\gamma(P_hw,P_hw)\le B_\gamma(w,w)\le C\|g\|_{L^2(\R^{d})}^2\) by \eqref{eq:crz-exterior-energy}. Apply \eqref{eq:localized-holder-concentration} to \(g_{i,h}\) with radius \(2r_h\), using \(\Psi_{\beta_c,s}(2r_h)\le C\Psi_{\beta_c,s}(r_h)\) for small \(h\), and then use \(r_h\le Ch^{\theta_N}\) from \eqref{eq:rh-rate}.
\end{proof}

Estimate \eqref{eq:localized-holder-concentration} includes normalized finite element sources supported in balls of radius at most $Ch$, with $C$ independent of $h$. The condition \(hN(h)\to0\) in Proposition~\ref{prop:mesh-resolved-concentration} ensures interpolation of the prescribed smooth profiles; it is not needed for this support-based estimate. Assembly and outer-tail errors remain part of the total discrepancy in Section~\ref{subsec:measured-dn}.

\begin{theorem}[Measured-sample error]
\label{thm:measured-sample-error-rate}
Assume the hypotheses of Corollary~\ref{cor:eta-c-explicit-rate} and Proposition~\ref{prop:mesh-resolved-concentration}, and the noise model \eqref{eq:sample-noise-Dih}. Then
\begin{equation*} 
E_h^\delta\le C\Psi_{\beta_c,s}(r_h)+\delta. 
\end{equation*}
For \(N(h)=\lfloor h^{-\theta_N}\rfloor\), \(0<\theta_N<1\), the ideal-profile concentration error \(\eta_\phi(h)\) from \eqref{eq:ideal-profile-concentration-error} satisfies, for sufficiently small \(h\),
\begin{equation}\label{eq:profile-concentration-comparison}
|\eta_c(h)-\eta_\phi(h)|\le C\mathfrak c_hh^{(1-\theta_N)\sigma_I},\qquad \eta_\phi(h)\le C\mathfrak c_h.
\end{equation}
Consequently,
\begin{equation}\label{eq:measured-sample-h-rate}
E_h^\delta\le\eta_\phi(h)+C\mathfrak c_hh^{(1-\theta_N)\sigma_I}+\delta\le C\mathfrak c_h+\delta.
\end{equation}
The constants may depend on the fixed $k$ and $\sigma_I$, but are independent of $h$ and $\delta$.
\end{theorem}
\begin{proof}
The radius estimate follows from $E_h^\delta\le\eta_c(h)+\delta$ and \eqref{eq:eta-c-radius-rate}. Write $\phi_i=\phi_{i,N(h)}$ and $g_i=g_{i,h}$. Both functions have unit $W^s$ norm, and normalization gives
\[
\|g_i-\phi_i\|_{W^s(\mathbb R^d)}\le2\|I_h^W\phi_i-\phi_i\|_{W^s(\mathbb R^d)}\le C(hN(h))^{\sigma_I}.
\]
Choose $r_{\phi,h}=C(N(h)^{-1}+h)$ so that both supports lie in $B(x_i,r_{\phi,h})$. For $v,w\in\widetilde H^s(W)$ supported in this ball, define the symmetric form $\mathcal B_i(v,w):=\langle\Lambda_\gamma v,w\rangle-\gamma(x_i)(v,w)_{W^s(\R^{d})}$. Homogeneity and scaled polarization of \eqref{eq:localized-holder-concentration} give $|\mathcal B_i(v,w)|\le C\Psi_{\beta_c,s}(r_{\phi,h})\|v\|_{W^s(\R^{d})}\|w\|_{W^s(\R^{d})}$. Since $\|g_i\|_{W^s(\R^{d})}=\|\phi_i\|_{W^s(\R^{d})}=1$,
\[
\left|\langle\Lambda_\gamma g_i,g_i\rangle-\langle\Lambda_\gamma\phi_i,\phi_i\rangle\right|=|\mathcal B_i(g_i-\phi_i,g_i+\phi_i)|\le C\Psi_{\beta_c,s}(r_{\phi,h})(hN(h))^{\sigma_I}.
\]
Now $\Psi_{\beta_c,s}(r_{\phi,h})\le C\mathfrak c_h$ and $(hN(h))^{\sigma_I}\le Ch^{(1-\theta_N)\sigma_I}$. Taking the maximum over $i$ proves the first estimate in \eqref{eq:profile-concentration-comparison}; applying \eqref{eq:localized-holder-concentration} to $\phi_i$ proves the second. The noise bound then yields \eqref{eq:measured-sample-h-rate}.
\end{proof}
For the maximal admissible interpolation gain $\sigma_I=k+1-s$, the profile interpolation contribution is $\mathfrak c_hh^{(1-\theta_N)(k+1-s)}$. The degree therefore enters the higher-order interpolation exponent, while the proved leading concentration exponent $\lambda_c=\theta_N\min\{\beta_c,2s\}$ and the critical logarithm remain unchanged. The forward approximation error is included separately in the total discrepancy $\delta$.

Let \(I_h^{W_0}:C(\overline{W_0})\to Z_h\) denote the nodal interpolation operator on \(\mathcal T_h^{W_0}\).  We assume that the nodal interpolant \(I_h^{W_0}\gamma\) satisfies
\begin{equation}\label{eq:interp-estimates-recon}
\|I_h^{W_0}\gamma-\gamma\|_{L^2(W_0)}
\le
C_I h^2\|\gamma\|_{H^2(W_0)},
\qquad
\|I_h^{W_0}\gamma\|_{H^1(W_0)}
\le
C_I\|\gamma\|_{H^2(W_0)}.
\end{equation}
For example, for $d\le3$ the embedding $H^2(W_0)\hookrightarrow C(\overline{W_0})$ makes the nodal interpolant well defined, and the fitted, shape-regular $P_1$ mesh yields \eqref{eq:interp-estimates-recon}. In higher dimensions we retain that estimate as an explicit approximation assumption on $\gamma$; see \cite[Thm.~4.4.20]{brenner2008mathematical}. For the maximum-norm analysis we also use the standard local \(W^{1,\infty}\)-stability of the affine nodal interpolant: for every element patch \(\omega_i=\operatorname{supp}\psi_i\),
\begin{equation}\label{eq:local-W1inf-stability}
\|\nabla I_h^{W_0}\gamma\|_{L^\infty(\omega_i)} \le C \|\nabla\gamma\|_{L^\infty(\omega_i)} .
\end{equation}
Estimate \eqref{eq:local-W1inf-stability} follows by affine scaling on shape-regular simplices; see \cite[Ch.~3]{ciarlet1978finite} or \cite[Ch.~4]{brenner2008mathematical}.

\begin{theorem}[A priori reconstruction \(L^2\) error]
\label{thm:recon-error}
Assume \eqref{eq:gamma-reg-recon} and
\eqref{eq:interp-estimates-recon}. Then there exists \(C>0\), independent of
\(h\), \(\alpha\), and the data error, such that
\begin{equation}\label{eq:recon-error-final}
\|\gamma_h^\delta-\gamma\|_{L^2(W_0)}
\le
C\left(
E_h^\delta
+
(\alpha^{1/2}+h^2)\|\gamma\|_{H^2(W_0)}
\right).
\end{equation}
\end{theorem}

\begin{proof}
Let \(I_h\gamma:=I_h^{W_0}\gamma\in Z_h\). Its nodal vector is \(\boldsymbol\gamma_h^\ast=(\gamma(x_1),\ldots,\gamma(x_{N_h}))^\top\). By the triangle inequality and \eqref{eq:lumped-L2}, we obtain
\[
    \|\gamma_h^\delta-\gamma\|_{L^2(W_0)}
    \le
    \|\mathbf c^\delta-\boldsymbol\gamma_h^\ast\|_{\mu,h}
    +\|I_h\gamma-\gamma\|_{L^2(W_0)} .
\]
Minimality of \(\gamma_h^\delta\) gives
\[
    \|\mathbf c^\delta-\mathbf d_h^\delta\|_{\mu,h}
    \le \|\mathbf d_h^\delta-\boldsymbol\gamma_h^\ast\|_{\mu,h}
    +\alpha^{1/2}\|I_h\gamma\|_{H^1(W_0)}.
\]
Hence
\[
    \|\mathbf c^\delta-\boldsymbol\gamma_h^\ast\|_{\mu,h}
    \le 2\|\mathbf d_h^\delta-\boldsymbol\gamma_h^\ast\|_{\mu,h}
    +\alpha^{1/2}\|I_h\gamma\|_{H^1(W_0)}.
\]
Using \eqref{eq:interp-estimates-recon} and \( \|\mathbf d_h^\delta-\boldsymbol\gamma_h^\ast\|_{\mu,h}^2 \le (E_h^\delta)^2\sum_{i=1}^{N_h}\mu_i =(E_h^\delta)^2|W_0|\) proves \eqref{eq:recon-error-final}.
\end{proof}

Under the hypotheses of Theorems \ref{thm:recon-error} and \ref{thm:measured-sample-error-rate}, with $N(h)=\lfloor h^{-\theta_N}\rfloor$, $0<\theta_N<1$, their estimates give
\begin{equation*} 
\|\gamma_h^\delta-\gamma\|_{L^2(W_0)}\le C\bigl(\eta_\phi(h)+\mathfrak c_hh^{(1-\theta_N)\sigma_I}+\delta+\alpha^{1/2}+h^2\bigr). 
\end{equation*}
By \eqref{eq:profile-concentration-comparison}, the right-hand side is bounded by $C(\mathfrak c_h+\delta+\alpha^{1/2}+h^2)$. For $\rho_h:=\mathfrak c_h+\delta+h^2$, the choice $0<\alpha(h)\le C_\alpha\rho_h^2$ therefore yields $\|\gamma_h^\delta-\gamma\|_{L^2(W_0)}\le C\rho_h$. The term $h^2$ is the approximation error in the piecewise affine conductivity space $Z_h$; $k$ is the degree of the forward and exterior-data spaces.

The preceding estimate controls the averaged reconstruction error. For the interior Liouville step below, however, we also need uniform control of the exterior coefficient, in particular to preserve positivity and to define \(\sqrt{\Gamma_h^\delta}\) and \((\Gamma_h^\delta)^{-1/2}\). We therefore derive an independent maximum-norm estimate at the nodal level. This estimate does not follow from the \(L^2\) bound by an inverse inequality; such an argument would introduce a large factor \(h^{-d/2}\). Instead, we use the row-wise stability of the finite-dimensional regularized reconstruction system.

We first record a sign-free \(\ell^\infty\)-stability estimate for the regularized reconstruction operator. The estimate uses only row-sum bounds for the lumped mass matrix and the stiffness matrix. For vectors and matrices, we use
\[
\|\mathbf v\|_{\ell^\infty}=\max_i |v_i|,
\qquad \|A\|_{\ell^\infty\to\ell^\infty}
:= \sup_{\mathbf v\ne0} \frac{\|A\mathbf v\|_{\ell^\infty}}{\|\mathbf v\|_{\ell^\infty}}
= \max_i\sum_j |A_{ij}|.
\]

\begin{lemma}[Maximum-norm stability of the regularized reconstruction operator]
\label{lem:linfty-stab}
Assume that the reconstruction meshes \(\mathcal T_h^{W_0}\) are fitted, shape-regular, and quasi-uniform, and that \(Z_h\) is the continuous piecewise affine space defined above. Then there exists a constant \(C_B>0\), depending only on the shape-regularity and quasi-uniformity constants, such that, for all sufficiently small \(h\),
\begin{equation}\label{eq:Dh-inv-Kh-linfty-bound}
\|D_h^{-1}K_h\|_{\ell^\infty\to\ell^\infty}
\le
C_B h^{-2}.
\end{equation}
Consequently, if \(0<q<1\) and
\begin{equation}\label{eq:linfty-alpha-smallness}
\alpha C_Bh^{-2}\le q,
\end{equation}
then
\begin{equation}\label{eq:sign-free-smoother-stability}
\left\| (D_h+\alpha K_h)^{-1}D_h\mathbf z \right\|_{\ell^\infty}
\le \frac1{1-q}\|\mathbf z\|_{\ell^\infty} \qquad \forall\,\mathbf z\in\mathbb R^{N_h},
\end{equation}
where \(K_h\) and \(D_h\) are defined in \eqref{eq:recon-linear-system}.
\end{lemma}

\begin{proof}
Write \(K_h=\mathsf M_h+\mathsf S_h\), where \(
(\mathsf M_h)_{ij}:=\int_{W_0}\psi_i\psi_j\,dx\), and 
\((\mathsf S_h)_{ij}:=\int_{W_0}\nabla\psi_i\cdot\nabla\psi_j\,dx\).
Since the \(\mathbb P_1\) nodal basis is nonnegative and forms a partition of unity, 
\[
\sum_{j=1}^{N_h}|(\mathsf M_h)_{ij}| = \sum_{j=1}^{N_h}(\mathsf M_h)_{ij}
= \int_{W_0}\psi_i\sum_{j=1}^{N_h}\psi_j\,dx = \int_{W_0}\psi_i\,dx = \mu_i .
\]
By shape regularity and quasi-uniformity,
\[
|\operatorname{supp}\psi_i|\le Ch^d, \qquad \|\nabla\psi_i\|_{L^\infty(\operatorname{supp}\psi_i)}\le Ch^{-1}, \qquad
\mu_i\ge ch^d,
\]
and only a uniformly bounded number of basis functions overlap with \(\psi_i\). Therefore
\( \sum_{j=1}^{N_h}|(\mathsf S_h)_{ij}| \le Ch^{d-2}\). It follows that
\[
\sum_{j=1}^{N_h}|(D_h^{-1}K_h)_{ij}| \le \frac1{\mu_i} \left( \sum_{j=1}^{N_h}|(\mathsf M_h)_{ij}| + \sum_{j=1}^{N_h}|(\mathsf S_h)_{ij}| \right) \le C(1+h^{-2}) \le C_Bh^{-2}
\]
for all sufficiently small \(h\). This proves \eqref{eq:Dh-inv-Kh-linfty-bound}. Since
\( D_h+\alpha K_h = D_h(I+\alpha D_h^{-1}K_h)\) and \( \qquad (D_h+\alpha K_h)^{-1}D_h
= (I+\alpha D_h^{-1}K_h)^{-1}\), condition \eqref{eq:linfty-alpha-smallness} gives
\[
\|\alpha D_h^{-1}K_h\|_{\ell^\infty\to\ell^\infty}\le q<1.
\]
Hence the Neumann series converges in \(\ell^\infty\to\ell^\infty\) operator norm:
\[
(I+\alpha D_h^{-1}K_h)^{-1} = \sum_{m=0}^{\infty}(-\alpha D_h^{-1}K_h)^m, \qquad \Longrightarrow  \qquad \|(I+\alpha D_h^{-1}K_h)^{-1}\|_{\ell^\infty\to\ell^\infty}
\le \sum_{m=0}^{\infty}q^m = \frac1{1-q}.
\]
This proves \eqref{eq:sign-free-smoother-stability}.
\end{proof}

\begin{lemma}[Row-wise regularization consistency]
\label{lem:regularization-bias-consistency}
Let \( \boldsymbol\gamma_h^\ast = (\gamma(x_1),\dots,\gamma(x_{N_h}))^\top \). Assume \eqref{eq:gamma-reg-recon} and \(\gamma|_{W_0}\in W^{1,\infty}(W_0)\). Then
\begin{equation*} 
\|D_h^{-1}K_h\boldsymbol\gamma_h^\ast\|_{\ell^\infty} \le C\left(\|\gamma\|_{L^\infty(W_0)}+h^{-1}\|\nabla\gamma\|_{L^\infty(W_0)}\right). 
\end{equation*}
\end{lemma}

\begin{proof}
Let \(I_h^{W_0}\gamma = \sum\limits_{j=1}^{N_h}\gamma(x_j)\psi_j\in Z_h\) be the nodal interpolant of \(\gamma\). Since the coefficient vector of \(I_h^{W_0}\gamma\) is \(\boldsymbol\gamma_h^\ast\), we have
\[
(K_h\boldsymbol\gamma_h^\ast)_i
= \int_{W_0}I_h^{W_0}\gamma\,\psi_i\,dx
+ \int_{W_0}\nabla I_h^{W_0}\gamma\cdot\nabla\psi_i\,dx .
\]
The mass part satisfies
\[
\left| \int_{W_0}I_h^{W_0}\gamma\,\psi_i\,dx \right| \le \|\gamma\|_{L^\infty(W_0)}\mu_i,
\]
because the nodal basis is nonnegative and forms a partition of unity. For the gradient part, using the \(W^{1,\infty}\)-stability of the nodal interpolant \eqref{eq:local-W1inf-stability}, \( |\operatorname{supp}\psi_i|\le Ch^d\) and  \( \|\nabla\psi_i\|_{L^\infty(\operatorname{supp}\psi_i)}\le Ch^{-1}\), we obtain
\[
\left| \int_{W_0}\nabla I_h^{W_0}\gamma\cdot\nabla\psi_i\,dx
\right| \le Ch^{d-1}\|\nabla\gamma\|_{L^\infty(W_0)}.
\]
Since \(\mu_i\ge ch^d\), division by \(\mu_i\) gives
\[
\frac{|(K_h\boldsymbol\gamma_h^\ast)_i|}{\mu_i} \le C \left( \|\gamma\|_{L^\infty(W_0)} + h^{-1}\|\nabla\gamma\|_{L^\infty(W_0)} \right).
\]
Taking the maximum over \(i\) proves the claim.
\end{proof}

\begin{theorem}[Maximum-norm consistency of the regularized reconstruction]
\label{thm:linfty-regularized-reconstruction}
Assume \eqref{eq:gamma-reg-recon} and \(\gamma|_{W_0}\in W^{1,\infty}(W_0)\). Let \(\alpha=\alpha(h)>0\) and \(0<q<1\), and suppose that \( \alpha C_Bh^{-2}\le q\), where \(C_B\) is the constant from Lemma~\ref{lem:linfty-stab}. Then, for all
sufficiently small \(h\), we have
\begin{equation}\label{eq:linfty-regulrec}
\begin{aligned}
\|\gamma_h^\delta-\gamma\|_{L^\infty(W_0)}
\le
\frac1{1-q}
\left[
E_h^\delta
+
C\alpha
\left(
\|\gamma\|_{L^\infty(W_0)}
+
h^{-1}\|\nabla\gamma\|_{L^\infty(W_0)}
\right)
\right]
+
Ch^{\beta_c}[\gamma]_{C^{0,\beta_c}(\overline{W_0})}.
\end{aligned}
\end{equation}
In particular, if \(h\to0\), \(E_h^\delta\to0\), and \(\alpha C_Bh^{-2}\le q<1\), then
\[
\gamma_h^\delta\to\gamma
\qquad
\text{in }L^\infty(W_0).
\]
\end{theorem}

\begin{proof}
Let \( \gamma_h^\delta=\sum_{j=1}^{N_h}c_j^\delta\psi_j\) where 
\(\mathbf c^\delta:=(c_1^\delta,\dots,c_{N_h}^\delta)^\top\). By \eqref{eq:recon-linear-system}, we have
\[
(D_h+\alpha K_h)\mathbf c^\delta
=
D_h\mathbf d_h^\delta .
\]
Subtracting
\((D_h+\alpha K_h)\boldsymbol\gamma_h^\ast\) gives
\[
(D_h+\alpha K_h)(\mathbf c^\delta-\boldsymbol\gamma_h^\ast)
=
D_h(\mathbf d_h^\delta-\boldsymbol\gamma_h^\ast)
-
\alpha K_h\boldsymbol\gamma_h^\ast .
\]
Since \( K_h\boldsymbol\gamma_h^\ast = D_h(D_h^{-1}K_h\boldsymbol\gamma_h^\ast)\), we obtain
\[
\mathbf c^\delta-\boldsymbol\gamma_h^\ast = (D_h+\alpha K_h)^{-1}D_h
(\mathbf d_h^\delta-\boldsymbol\gamma_h^\ast)
- \alpha (D_h+\alpha K_h)^{-1}D_h (D_h^{-1}K_h\boldsymbol\gamma_h^\ast).
\]
By Lemma~\ref{lem:linfty-stab}, we get
\[
\|\mathbf c^\delta-\boldsymbol\gamma_h^\ast\|_{\ell^\infty}
\le
\frac1{1-q}
\left[
\|\mathbf d_h^\delta-\boldsymbol\gamma_h^\ast\|_{\ell^\infty}
+
\alpha
\|D_h^{-1}K_h\boldsymbol\gamma_h^\ast\|_{\ell^\infty}
\right].
\]
By definition \eqref{eq:def-Eh-delta}, \(\|\mathbf d_h^\delta-\boldsymbol\gamma_h^\ast\|_{\ell^\infty} = E_h^\delta\). Using Lemma~\ref{lem:regularization-bias-consistency}, we get
\[
\|\mathbf c^\delta-\boldsymbol\gamma_h^\ast\|_{\ell^\infty}
\le \frac1{1-q} \left[ E_h^\delta + C\alpha \left( \|\gamma\|_{L^\infty(W_0)}
+ h^{-1}\|\nabla\gamma\|_{L^\infty(W_0)} \right) \right].
\]
Let \(I_h^{W_0}\gamma\in Z_h\) be the nodal interpolant of \(\gamma\). Since \(
\gamma_h^\delta-I_h^{W_0}\gamma = \sum_{i=1}^{N_h}(c_i^\delta-\gamma(x_i))\psi_i\) and the nodal basis is nonnegative with
\(\sum_{i=1}^{N_h}\psi_i\equiv1\), we have
\[
\|\gamma_h^\delta-I_h^{W_0}\gamma\|_{L^\infty(W_0)}
\le
\|\mathbf c^\delta-\boldsymbol\gamma_h^\ast\|_{\ell^\infty}.
\]
Moreover, \( \|I_h^{W_0}\gamma-\gamma\|_{L^\infty(W_0)} \le
Ch^{\beta_c}[\gamma]_{C^{0,\beta_c}(\overline{W_0})}\). Combining the estimates proves \eqref{eq:linfty-regulrec}. If \(\alpha C_Bh^{-2}\le q\), then
\[
\alpha h^{-1}\le \frac q{C_B}h\to0.
\]
Thus the right-hand side tends to zero whenever \(E_h^\delta\to0\) and
\(h\to0\).
\end{proof}

\begin{proof}[Proof of Theorem~\ref{cor:linfty-fully-discrete-rate}]
Insert \eqref{eq:measured-sample-h-rate} into \eqref{eq:linfty-regulrec}. The parameter rule gives $\alpha h^{-1}\le c_\alpha h\le c_\alpha h^{\beta_c}$. Estimate \eqref{eq:profile-concentration-comparison} and $h^{(1-\theta_N)\sigma_I}\le1$ imply \eqref{eq:linfty_rate_sim}. The final rate follows from $h^{\beta_c}\le\mathfrak c_h$ for sufficiently small $h$.
\end{proof}

\section{Interior reconstruction by discrete Liouville--Cauchy continuation}
\label{sec:interior-cauchy}

In this section, we reconstruct \(\gamma\) in \(\Omega\) through the fractional Liouville transform, recovery of the Schr\"odinger potential, and regularized continuation of \(m_\gamma=\sqrt\gamma-1\).

The continuous change of variables is the fractional Liouville reduction for the fractional conductivity equation; see \cite[Thm.~3.1]{Covi2020Inverse} and, for global conductivities under stronger regularity, \cite[Lem.~4.2]{covi2022global}. Under the regularity assumed here, the reduction and DN pairing identity follow from \cite[Lem.~3.9, 3.11, and 3.13]{RailoZimmermann2024low}. We use the notation
\begin{equation}\label{eq:qgamma_def}
m_\gamma =\sqrt\gamma-1,
\qquad a_\gamma:=\gamma^{-1/2}=(1+m_\gamma)^{-1},
\qquad q_\gamma
:= -\frac{(-\Delta)^s m_\gamma}{1+m_\gamma}
= -\frac{(-\Delta)^s\sqrt\gamma}{\sqrt\gamma}.
\end{equation}
Thus, if \(u\) solves the fractional conductivity equation \eqref{eq:frac-cond} and \( v:=\sqrt\gamma\,u,\) then \(v\) solves the fractional Schr\"odinger equation
\begin{equation}\label{eq:fse}
\bigl((-\Delta)^s+q_\gamma\bigr)v=0
\quad\text{in }\Omega .
\end{equation}
The identity \eqref{eq:qgamma_def} is understood in the weak multiplier sense in the general low-regularity setting.  We write the interior Schr\"odinger potential as
\begin{equation*}
Q_\gamma:=q_\gamma|_\Omega.
\end{equation*}
Moreover, from \eqref{eq:qgamma_def}, \(m_\gamma\) satisfies the inhomogeneous Liouville equation
\begin{equation}\label{eq:fse2}
\bigl((-\Delta)^s+Q_\gamma\bigr)m_\gamma=-Q_\gamma
\quad\text{in }\Omega .
\end{equation}

\subsection*{Standing Liouville assumptions for the interior analysis}

We use the multiplier spaces from Section~\ref{subsec:function-spaces}. The following assumptions ensure the weak Liouville reduction and the additional regularity needed to reconstruct the potential and Cauchy data as functions.  

\begin{assumption}[Standing Liouville regularity]
\label{ass:liouville-regularity}
Throughout the Liouville--Cauchy reconstruction we assume the following.

\begin{enumerate}[label=(\roman*)]
\item \textbf{Ellipticity, normalization, and Liouville admissibility.}
There exist constants \(0<\gamma_0\le\gamma_1<\infty\) such that \(\gamma_0\le\gamma\le\gamma_1\) almost everywhere in \(\mathbb R^d\) and  \( m_\gamma \in \widetilde H^s(\Omega_R)\cap H^{s,d/s}(\mathbb R^d)\). In particular, \(m_\gamma=0\) almost everywhere in \(\mathbb R^d\setminus\Omega_R\), equivalently \(\gamma\equiv1\) in \(\mathbb R^d\setminus\Omega_R\). The \(H^{s,d/s}\)-regularity is the hypothesis used in the weak fractional Liouville reduction.

\item \textbf{\(L^\infty\)-representative of the Liouville potential on the computational domain.} The restriction of the weak Sobolev multiplier \(q_\gamma\) to \(\Omega_R\) is represented by a function in \(L^\infty(\Omega_R)\), still denoted by \(q_\gamma\). More precisely, \(
\langle q_\gamma,\varphi\rangle =\int_{\Omega_R}q_\gamma\varphi\,dx\) for all \(\varphi\in C_c^\infty(\Omega_R)\). We let \( Q_\gamma=q_\gamma|_\Omega\in L^\infty(\Omega)\) and \(q_\gamma|_{W}\in L^\infty(W)\). 

\item \textbf{Exterior \(C^1\)-regularity.}
The conductivity has a \(C^1\)-representative on an open neighborhood
\(\mathcal O_0\) of \(\overline{W_0}\), with
\(\mathcal O_0\Subset W\). Consequently,
\[
    a_\gamma:=\gamma^{-1/2}
    \in W^{1,\infty}(W_0)
    \cap
    \bigcap_{U\Subset W_0}\mathcal M_s(U).
\]
This exterior \(C^1\)-regularity is used only to verify the discrete multiplier consistency estimate \eqref{eq:eps-a-def} below.
\end{enumerate}
\end{assumption}

The exact Schr\"odinger states associated with \(Q_\gamma\) are obtained from the well-posed conductivity problem through the Liouville transform. Hence the transformed Schr\"odinger equation is well-posed, and we do not need to assume separately that zero is not a Dirichlet eigenvalue.

For functions supported in \(\overline{\Omega_R}\), define the standard truncated fractional Laplacian energy form
\begin{align}
A_R(a,b):={}& \frac{C_{d,s}}{2} \int_{\Omega_R}\int_{\Omega_R}
\frac{(a(x)-a(y))(b(x)-b(y))}
{|x-y|^{d+2s}}\,dx\,dy \nonumber\\ &\quad+
C_{d,s}\int_{\Omega_R}a(x)b(x)
\left( \int_{\mathbb R^d\setminus\Omega_R} \frac{dy}{|x-y|^{d+2s}}\right)dx .
\label{eq:AR-standard-form}
\end{align}

For $a,b\in\widetilde H^s(\Omega_R)$, zero extension gives $A_R(a,b)=\langle(-\Delta)^s a,b\rangle$ exactly. Thus \eqref{eq:AR-standard-form} includes the full exterior interaction; approximating its tail integral is a separate numerical operation.

The key point is that the conductivity DN pairing can be converted into a Schr\"odinger DN pairing once the exterior multiplier \(a_\gamma\) is known on the measurement set. For \(Q\in L^\infty(\Omega)\), the weak
Schr\"odinger DN pairing is written as
\begin{equation}\label{eq:inter_DN_schro}
\big\langle\Lambda_Q^\Omega F,G\big\rangle
= A_R(v_F^Q,G)+\int_\Omega Qv_F^QG\,dx , \qquad\qquad \forall~ G\in\widetilde H^s(\Omega_R),
\end{equation}
where \(v_F^Q\) is the solution of the exterior value problem associated with the fractional Schr\"odinger equation \eqref{eq:fse} for the exterior value \(F\) in \(\Omega_e\); see \cite[Sec. 2]{Ghosh2020uniqueness} and \cite[Lem. 2.4]{Ghosh2020calderon}, whenever the associated exterior value problem is uniquely solvable. Throughout this discussion,  \(F\) and \(G\) are exterior functions in \(\Omega_R\), extended by zero to \(\mathbb R^d\).

It is useful to distinguish the global Liouville Schr\"odinger pairing from the interior Schr\"odinger pairing defined above. The low-regularity Liouville identity \cite[Lem.~3.13]{RailoZimmermann2024low} gives, for \(F,G\in\widetilde H^s(W)\) and \(a_\gamma=\gamma^{-1/2}\),
\begin{equation*} 
\big\langle\Lambda_\gamma(a_\gamma F),a_\gamma G\big\rangle=\big\langle\Lambda_{q_\gamma}F,G\big\rangle. 
\end{equation*}
Here \(\Lambda_{q_\gamma}\) uses the global potential \(q_\gamma\). The transformed state \(v=\sqrt\gamma\,u_{a_\gamma F}\) equals \(v_F^{Q_\gamma}\) by uniqueness; since \(G=0\) in \(\Omega\) and \(v=F\) in \(\Omega_e\), comparison with \eqref{eq:inter_DN_schro} leaves only the exterior potential term. Consequently,
\begin{equation}\label{eq:split-intro}
\big\langle\Lambda_\gamma(a_\gamma F),a_\gamma G\big\rangle=\big\langle\Lambda_{Q_\gamma}^{\Omega}F,G\big\rangle+\int_{\Omega_e}q_\gamma FG\,dx.
\end{equation}
This is the interior-exterior split of the global Liouville DN identity. The separate roles of disjoint and overlapping exterior test supports in the global inverse problem were emphasized in \cite[Rem. 4.1]{covi2022global}. In the exterior-trivial setting of \cite[Lem. 3.4]{Covi2020Inverse}, the same tail is written as an integral involving \((-\Delta)^s m_\gamma\); indeed, when \(\gamma=1\) in \(\Omega_e\), \(q_\gamma=-(-\Delta)^s m_\gamma\) in \(\Omega_e\).

Now we introduce the discrete setup and the main idea of the interior conductivity reconstruction.  We choose nonempty mesh-fitted polygonal Lipschitz open sets \(W_a,W_b,W_c\Subset W_0\) and assume
 \begin{equation*} 
 \operatorname{dist}(W_a,W_b)>0. 
 \end{equation*}
Therefore, if \(F\in\widetilde H^s(W_a)\) and \(G\in\widetilde H^s(W_b)\) then \(FG=0\) in \(\Omega_e\). Hence \eqref{eq:split-intro} reduces to
\begin{equation}\label{eq:separated-identity}
\big\langle\Lambda_\gamma(a_\gamma F),a_\gamma G\big\rangle
= \big\langle\Lambda_{Q_\gamma}^{\Omega}F,G\big\rangle .
\end{equation}
If instead \(F,G\in\widetilde H^s(W_c)\), then \(FG\) is supported in \(W_c\), and therefore
\begin{equation}\label{eq:overlap-identity}
\big\langle\Lambda_\gamma(a_\gamma F),a_\gamma G\big\rangle - \big\langle\Lambda_{Q_\gamma}^{\Omega}F,G\big\rangle = \int_{W_c}q_\gamma FG\,dx.
\end{equation}
The separated identity \eqref{eq:separated-identity} is used to reconstruct \(Q_\gamma\). The overlapping identity \eqref{eq:overlap-identity} is used only after \(Q_\gamma\) has been approximated; it then gives the exterior tail \(q_\gamma|_{W_c}\).

We assume that the triangulation \(\mathcal T_h\) is fitted to \(W_a\), \(W_b\), and \(W_c\). We recall the discretization from Section \ref{sec:ext}. For \(U\in\{W_a,W_b\}\), define
\[
     V_h^U:=\operatorname{span}\{\zeta_z:z\in\mathcal Z_h\cap U\}.
\]
Because the mesh is fitted to $U$, these functions belong to $H_0^1(U)\subset\widetilde H^s(U)$ after extension by zero. Their supports lie in $\overline U\Subset W$, so they also belong to the compactly supported exterior finite element space used in Section~\ref{sec:ext}. For the overlapping tail step, we enrich the source and test space on $W_c$ by cell bubbles. First set
\[
    V_{h,0}^c:=\operatorname{span}\{\zeta_z:z\in\mathcal Z_h\cap W_c\}.
\]
Let \(\mathcal T_h^c=\{K_1,\ldots,K_{N_t}\}\) be the elements contained in \(W_c\). For each \(K_m\), let \(b_{m,h}\) be the standard simplex bubble, defined on \(K_m\) by the product of the barycentric coordinates and extended by zero outside \(K_m\). Thus \( b_{m,h}\in H_0^1(K_m)\subset \widetilde H^s(W_c)\), \(b_{m,h}>0\) in the interior of \(K_m\). If \(k\ge d+1\), these bubbles are already contained in the global polynomial space \(X_h^k(\Omega_R)\). If \(k<d+1\), we use them as a local enrichment only for the tail source/test space. Define
\[
    B_h^c:=\operatorname{span}\{b_{m,h}:m=1,\ldots,N_t\}, \qquad V_h^c:=V_{h,0}^c+B_h^c .
\]
We write \( V_h^a:=V_h^{W_a}\) and  \(V_h^b:=V_h^{W_b}\), and let
\[
    \{\chi_{\ell,h}^a\}_{\ell=1}^{N_a},\qquad
    \{\chi_{\ell,h}^b\}_{\ell=1}^{N_b},\qquad
    \{\chi_{\ell,h}^c\}_{\ell=1}^{N_c}
\]
be bases of \(V_h^a,V_h^b,V_h^c\), respectively. The basis of \(V_h^c\) is arbitrary, but in implementations it is convenient to include the bubbles
\(\{b_{m,h}\}_{m=1}^{N_t}\) explicitly. For \(U=W_a,W_b\), let \(I_h^U\) be the localized nodal interpolant. Write \(I_{h,0}^c\) for nodal interpolation into \(V_{h,0}^c\), and let \(b_K=b_{m,h}\) for \(K=K_m\). For \(u\in C(\overline{W_c})\) vanishing on \(\partial W_c\), define, when \(k<d+1\),
\[
 I_h^c u:=I_{h,0}^c u+
 \sum_{K\in\mathcal T_h^c}
 \frac{u(x_K)-I_{h,0}^cu(x_K)}{b_K(x_K)}b_K,
\]
where \(x_K\) is the barycenter of \(K\). Every degree-\(k\) simplex node then lies on the element boundary, so \(I_{h,0}^c b_K=0\). For \(k\ge d+1\), set \(I_h^c=I_{h,0}^c\). Both definitions reproduce \(V_h^c\). Let \(\overline{W_c}\subset W_*\Subset W_0\), \(\omega\in W^{1,\infty}(W_*)\), and \(\phi_h\in V_h^c\). For \(E_h:=I_h^c(\omega\phi_h)-\omega\phi_h\), local reproduction, freezing \(\omega\) on each element, and inverse estimates give
\[
 \|E_h\|_{L^2(\mathbb R^d)}+h\|E_h\|_{H^1(\mathbb R^d)}
 \le Ch\|\nabla\omega\|_{L^\infty(W_*)}\|\phi_h\|_{L^2(\mathbb R^d)}.
\]
Interpolation yields
\begin{equation}\label{eq:tail-enriched-superapprox}
 \|I_h^c(\omega\phi_h)-\omega\phi_h\|_{H^s(\mathbb R^d)}
 \le Ch^{1-s}\|\omega\|_{W^{1,\infty}(W_*)}\|\phi_h\|_{H^s(\mathbb R^d)},
\end{equation}
with \(C\) independent of \(h\), \(\omega\), and \(\phi_h\).

The final Cauchy--Liouville continuation is posed on the computational domain \(\Omega_R\), not only on \(\Omega\). We therefore use \(V_h^R := X_h^k(\Omega_R)\cap H_0^1(\Omega_R)\), where functions are extended by zero to \(\mathbb R^d\). Then \( V_h^\Omega\subset V_h^R\subset \widetilde H^s(\Omega_R)\). We denote the nodal basis of \(V_h^R\) by \(
\{\Phi_{i,h}\}_{i=1}^{N_R}\). The \(L^2(W_c)\)-projection onto the Cauchy observation space is denoted by
\begin{equation}\label{eq:Pi-h-c}
\Pi_h^c:L^2(W_c)\to V_h^c .
\end{equation}

On the full fitted, shape-regular, quasi-uniform observation meshes, define $\mathcal I_h^Uv:=I_{h,0}^U(\varrho_h*v)$ for $U\in\{W_b,W_c\}$ and functions $v$ supported in a fixed compact subset of $U$. Here $\varrho_h$ is a symmetric smooth mollifier of radius $h$, and $I_{h,0}^U$ is nodal interpolation of fixed degree $k\ge1$. For sufficiently small $h$, $\mathcal I_h^Uv$ belongs to the nodal observation subspace and satisfies
\begin{equation}\label{eq:local-averaged-interpolation}
\begin{aligned}
\|\mathcal I_h^Uv\|_{H^s}\le C\|v\|_{H^s},\qquad \|v-\mathcal I_h^Uv\|_{L^2(\R^{d})}\le Ch^s\|v\|_{H^s}, \qquad \|v-\mathcal I_h^Uv\|_{H^s}\le Ch^s\|v\|_{H^{2s}}.
\end{aligned}
\end{equation}
We write $\mathcal I_h^b:=\mathcal I_h^{W_b}$ and $\mathcal I_h^c:=\mathcal I_h^{W_c}$; all norms in \eqref{eq:local-averaged-interpolation} are on $\mathbb R^d$. Local averaging and affine reproduction give $L^2$- and $H^1$-stability and order-$h$ approximation from $H^1$ to $L^2$ and from $H^2$ to $H^1$. Interpolation, localized by a fixed cutoff inside $U$, gives the first two estimates and $\|v-\mathcal I_h^Uv\|_{H^s}\le Ch\|v\|_{H^{1+s}}$; interpolating the latter with $H^s$-stability gives the third estimate.

The exterior reconstruction developed in the previous sections gives a
finite-dimensional approximation \(\gamma_h^\delta\) of \(\gamma\) on
\(W_0\). By Theorem~\ref{cor:linfty-fully-discrete-rate}, under
the hypotheses and parameter choices stated there,
\[
\gamma_h^\delta\to\gamma
\quad\text{in }L^\infty(W_0).
\]
Since the exact conductivity satisfies \( 0<\gamma_0\le \gamma\le \gamma_1\),
this convergence implies that \(\gamma_h^\delta\) is eventually positive in
the asymptotic regime. Nevertheless, the Liouville transform below uses the
finite-dimensional quantities \((\gamma_h^\delta)^{-1/2}\) and
\(\sqrt{\gamma_h^\delta}\) before the limit is taken. To keep
these quantities well-defined and uniformly bounded for every discrete data set, we use the admissible exterior representative
\begin{equation}\label{eq:Gamma-h-delta}
\Gamma_h^\delta
:=
\min\{2\gamma_1,\max\{\gamma_h^\delta,\gamma_0/2\}\}
\quad\text{on }W_0 .
\end{equation}
This projection of the approximation does not change the convergence statement. Indeed, the map \(
T(t):=\min\{2\gamma_1,\max\{t,\gamma_0/2\}\}\) is Lipschitz with constant one, and \(T(\gamma)=\gamma\). Hence
\[
\|\Gamma_h^\delta-\gamma\|_{L^\infty(W_0)}
\le
\|\gamma_h^\delta-\gamma\|_{L^\infty(W_0)}
\to0 .
\]
We define the exterior background deviation by
\begin{equation}\label{eq:ah-mh}
m_h^W:=\sqrt{\Gamma_h^\delta}-1.
\end{equation}
Consequently, $m_h^W\to\sqrt\gamma-1$ and $(\Gamma_h^\delta)^{-1/2}\to\gamma^{-1/2}$ in $L^\infty(W_0)$. The smoothed multiplier $a_h^\delta$ used in the DN pairings is constructed in Lemma~\ref{lem:discrete-multiplier-consistency} below.
In this section, the only noisy measured object is the finite-dimensional bilinear form \(\mathcal D_h^\delta\) from \eqref{eq:dn-noise-discrete}. If the tail space \(V_h^c\) is bubble-enriched, we use the same noise model on the corresponding enlarged exterior data space. Equivalently, in the tail estimates below, the noise bound \eqref{eq:dn-noise-discrete} is understood on \(V_h^W+B_h^c\subset\widetilde H^s(W)\), with \(B_h^c=\{0\}\) if no enrichment is used.

The reconstruction in \(\Omega\) has four steps. First, we use the exterior approximation \(a_h^\delta\) of \(a_\gamma=\gamma^{-1/2}\) and the separated identity \eqref{eq:separated-identity} to convert conductivity measurements with sources in \(W_a\) and tests in \(W_b\) into interior Schr\"odinger DN data. Second, after subtracting the known source contribution, we recover the interior corrections, and hence the Schr\"odinger states, from their exterior fractional-Laplacian observations by regularized continuation; we then approximate \(Q_\gamma\) by fitting the weak equations for these states without dividing by small state values. Third, we use sources and tests in \(W_c\) and subtract the interior Schr\"odinger response computed with the approximate \(Q_\gamma\); by \eqref{eq:overlap-identity}, the remainder consists of moments of \(q_\gamma|_{W_c}\), which we fit in a piecewise-constant space (cell-bubble tests give a sufficient identifiability criterion). Fourth, the exterior reconstruction supplies \(m_\gamma|_{W_c}=\sqrt{\gamma}-1\), while the recovered tail supplies \((-\Delta)^s m_\gamma|_{W_c}=-\sqrt{\gamma}\,q_\gamma|_{W_c}\); we combine these data with \eqref{eq:fse2} in a regularized Cauchy problem on \(\Omega_R\), then clip the reconstructed square root to a positive admissible interval and square it to obtain the interior conductivity.

The details of these four steps are given below. Before starting the reconstruction, we establish a consistency result for the discrete exterior Liouville multiplier \(a_h^\delta\phi_h\). As mentioned earlier, the exact Liouville identities use the exterior multiplier \(a_\gamma=\gamma^{-1/2}\). In the discrete data this multiplier is replaced by an approximation constructed from the exterior reconstruction \(\Gamma_h^\delta\). The convergence
\[
    e_h^\delta:=\|\Gamma_h^\delta-\gamma\|_{L^\infty(W_0)}\to0
\]
is provided by Theorem~\ref{cor:linfty-fully-discrete-rate}. Since the DN map acts on \(H^s\)-based exterior data, this \(L^\infty\)-convergence alone is not sufficient when the multiplier is applied to all functions in the exterior finite element spaces. We therefore smooth \(\Gamma_h^\delta\) locally inside \(W_0\) and prove convergence of the induced multiplier in the operator norm needed for the transformed discrete data.

Under the hypotheses of Theorem~\ref{cor:linfty-fully-discrete-rate}, choose \(0<\alpha(h)\le c_\alpha h^2\) with \(c_\alpha C_B<1\), and let \(\delta(h)\to0\). Set \(\eta_h^\infty:=\mathfrak c_h+\delta+h^{\beta_c}\). Equations \eqref{eq:Gamma-h-delta} and \eqref{eq:linfty_rate_sim} give \(e_h^\delta\le C\eta_h^\infty\). Choose \(W_*\Subset W_0\) with \(\overline{W_a\cup W_b\cup W_c}\subset W_*\), set \(d_*:=\operatorname{dist}(W_*,\mathbb R^d\setminus W_0)>0\), and choose
\( \sigma_h:=\min\bigl\{(\eta_h^\infty)^{1/2},d_*/4\bigr\}\). Then \(\sigma_h\to0\), and, for all sufficiently small \(h\),
\[
\frac{e_h^\delta}{\sigma_h}\le C(\eta_h^\infty)^{1/2}\longrightarrow0.
\]

Let \(\rho\in C_c^\infty(B(0,1))\) be nonnegative with \(\int_{\mathbb R^d}\rho\,dx=1\), and set \(\rho_\sigma(x)=\sigma^{-d}\rho(x/\sigma)\). For
\(x\in W_*\), define
\[
    \widehat\Gamma_h^\delta(x) := \int_{W_0}
    \rho_{\sigma_h}(x-y)\Gamma_h^\delta(y)\,dy .
\]
Note that the restriction \(\sigma_h<\operatorname{dist}(W_*,\mathbb R^d\setminus W_0)\)
ensures that only values of \(\Gamma_h^\delta\) inside \(W_0\) are used.

\begin{lemma}[Consistency of the discrete exterior Liouville multiplier]
\label{lem:discrete-multiplier-consistency}
Let \(W_*\) and \(\sigma_h\) be chosen as above, and assume Assumption~\ref{ass:liouville-regularity}. Let \(\widehat\Gamma_h^\delta\) be the local mollification of \(\Gamma_h^\delta\) on \(W_*\) with a nonnegative standard mollifier of radius \(\sigma_h\), and define \(a_{h,\sigma}^\delta:=(\widehat\Gamma_h^\delta)^{-1/2}\).  Then, for every \(U\in\{W_a,W_b,W_c\}\), there holds
\begin{equation}\label{eq:eps-a-def}
    \varepsilon_U^a(h,\delta)
    :=
    \sup_{0\ne\phi_h\in V_h^U}
    \frac{
        \|I_h^U(a_{h,\sigma}^\delta\phi_h)-a_\gamma\phi_h\|_{H^s(\mathbb R^d)}
    }{
        \|\phi_h\|_{H^s(\mathbb R^d)}
    }
    \longrightarrow0 .
\end{equation}
More precisely,
\begin{equation}\label{eq:eps-a-bound}
    \varepsilon_U^a(h,\delta)
    \le
    C\left(
        e_h^\delta+\omega_\gamma(\sigma_h)
        +\frac{e_h^\delta}{\sigma_h}
        +\omega_{\nabla\gamma}(\sigma_h)
        +h^{1-s}
    \right),
\end{equation}
where \(\omega_\gamma\) and \(\omega_{\nabla\gamma}\) denote the moduli of
continuity of \(\gamma\) and \(\nabla\gamma\) on \(\overline{W_0}\). The
constant \(C\) is independent of \(h\), \(\delta\), and \(U\).
\end{lemma}

\begin{proof}
For \(x\in W_*\), the choice of \(\sigma_h\) gives
\(B(x,\sigma_h)\subset W_0\). Hence the mollification uses only exterior
values where \(\Gamma_h^\delta\) is known. Since the mollifier is nonnegative
and \(\Gamma_h^\delta\in[\gamma_0/2,2\gamma_1]\), we have \( \gamma_0/2\le \widehat\Gamma_h^\delta\le 2\gamma_1 \quad\text{on }W_* \).
Moreover, \( \|\widehat\Gamma_h^\delta-\gamma\|_{L^\infty(W_*)} \le e_h^\delta+\omega_\gamma(\sigma_h)\). Differentiating the mollification and adding and subtracting \(\nabla(\rho_{\sigma_h}*\gamma)=\rho_{\sigma_h}*\nabla\gamma\), we obtain
\[
\|\nabla\widehat\Gamma_h^\delta-\nabla\gamma\|_{L^\infty(W_*)}
\le
\|\nabla\rho_{\sigma_h}\|_{L^1}
\|\Gamma_h^\delta-\gamma\|_{L^\infty(W_0)}
+
\omega_{\nabla\gamma}(\sigma_h) \le C\frac{e_h^\delta}{\sigma_h}
+
\omega_{\nabla\gamma}(\sigma_h),
\]
since \(\|\nabla\rho_{\sigma_h}\|_{L^1}=C\sigma_h^{-1}\). Thus \(\widehat\Gamma_h^\delta\to\gamma\) in \(W^{1,\infty}(W_*)\). Since
\(t\mapsto t^{-1/2}\) is smooth on \([\gamma_0/2,2\gamma_1]\),
\[
    \eta_a(h,\delta)
    := \|a_{h,\sigma}^\delta-a_\gamma\|_{W^{1,\infty}(W_*)} \le
    C\left( e_h^\delta+\omega_\gamma(\sigma_h) +\frac{e_h^\delta}{\sigma_h}
        +\omega_{\nabla\gamma}(\sigma_h)
    \right).
\]
Fix $U\in\{W_a,W_b,W_c\}$, $\phi_h\in V_h^U$, and $\omega\in W^{1,\infty}(W_*)$. For $U=W_a,W_b$, nodal reproduction permits subtracting $\omega(x_K)\phi_h$ on each element $K$, where $x_K\in K$. The bound $\|\omega-\omega(x_K)\|_{L^\infty(K)}\le h_K\|\nabla\omega\|_{L^\infty(K)}$, local norm equivalence, and polynomial inverse estimates give the following product estimates; see \cite[Thm.~1]{BrennerSuperapproximation2020} for smooth cutoffs:
\[
\begin{aligned}
\|I_h^U(\omega\phi_h)-\omega\phi_h\|_{L^2(\mathbb R^d)}
&\le
Ch\|\nabla\omega\|_{L^\infty(W_*)}
\|\phi_h\|_{L^2(\mathbb R^d)},\\
\|I_h^U(\omega\phi_h)-\omega\phi_h\|_{H^1(\mathbb R^d)}
&\le
C\|\nabla\omega\|_{L^\infty(W_*)}
\|\phi_h\|_{L^2(\mathbb R^d)}.
\end{aligned}
\]
For \(U=W_c\), the same estimates hold by the explicit construction of the local interpolant \(I_h^c\) on the enriched space; see \eqref{eq:tail-enriched-superapprox}. Interpolation between \(L^2\) and \(H^1\) therefore yields, uniformly for \(U\in\{W_a,W_b,W_c\}\),
\[
\|I_h^U(\omega\phi_h)-\omega\phi_h\|_{H^s(\mathbb R^d)}
\le
C h^{1-s}
\|\omega\|_{W^{1,\infty}(W_*)}
\|\phi_h\|_{H^s(\mathbb R^d)} .
\]

Taking \(\omega=a_{h,\sigma}^\delta\), and using the uniform bound \(\|a_{h,\sigma}^\delta\|_{W^{1,\infty}(W_*)}\le C\), gives
\[
\|I_h^U(a_{h,\sigma}^\delta\phi_h) -a_{h,\sigma}^\delta\phi_h\|_{H^s(\mathbb R^d)}
\le Ch^{1-s}\|\phi_h\|_{H^s(\mathbb R^d)} .
\]
Since \(U\Subset W_*\), choose \(\zeta_U\in C_c^\infty(W_*)\) with \(\zeta_U=1\) on \(U\). Multiplication by a \(W^{1,\infty}(\mathbb R^d)\)-function is bounded on \(H^s(\mathbb R^d)\) for \(0<s<\min\{1,d/2\}\). Applying this to \(\zeta_U(a_{h,\sigma}^\delta-a_\gamma)\) gives
\[
\|(a_{h,\sigma}^\delta-a_\gamma)\phi_h\|_{H^s(\mathbb R^d)}
\le
C\|a_{h,\sigma}^\delta-a_\gamma\|_{W^{1,\infty}(W_*)}
\|\phi_h\|_{H^s(\mathbb R^d)} \le C \eta_a(h,\delta)\|\phi_h\|_{H^s(\mathbb R^d)}.
\]
Combining the last two bounds proves \eqref{eq:eps-a-bound}. The choice of \(\sigma_h\) above gives \(e_h^\delta\to0\), \(\sigma_h\to0\), and \(e_h^\delta/\sigma_h\to0\). Moreover, Assumption~\ref{ass:liouville-regularity}(iii) implies that \(\nabla\gamma\) is uniformly continuous on \(\overline{W_0}\). Hence the right-hand side of \eqref{eq:eps-a-bound} tends to zero, proving \eqref{eq:eps-a-def}.
\end{proof}

We henceforth set $a_h^\delta:=a_{h,\sigma}^\delta=(\widehat\Gamma_h^\delta)^{-1/2}$ on $W_*$. The exterior value $m_h^W=\sqrt{\Gamma_h^\delta}-1$ and the projected tail datum $\tau_h^c$ in \eqref{eq:h-c-definition} continue to use the unsmoothed clipped conductivity $\Gamma_h^\delta$.

\subsection{Discrete reconstruction of the Schr\"odinger potential}
\label{subsec:discrete-Qgamma}

Using identity \eqref{eq:separated-identity} and the reconstructed multiplier \(a_h^\delta\), we form finite-dimensional approximations of the Schr\"odinger DN data for \(Q_\gamma\) and establish their consistency. We then recover the Schr\"odinger states by Galerkin--Tikhonov regularization of the exterior observation problem \cite[Thm.~2]{Ghosh2020uniqueness}, compute their regularized weak fractional Laplacian, and fit \(Q_\gamma\) by a stabilized residual method \cite[Sec. 3]{DwivediRailoRuppFCP}. The single-measurement logarithmic potential estimate of \cite[Thm.~1]{Ruland2021single} is not applied: it requires uniformly \(C^{0,s}\) potentials supported in a common \(\Omega'\Subset\Omega\) and a continuous single-source Schr\"odinger DN discrepancy, whereas \(Q_\gamma\) is only assumed to lie in \(L^\infty(\Omega)\) and the present analysis controls finite-dimensional observations and residuals.

Choose exterior Schr\"odinger sources \(\{F_{j,h}^a\}_{j=1}^{J_a}\subset V_h^a\), where \(J_a=J_a(h)\) may depend on \(h\). The observation space is \(V_h^b\), with nodal basis
\(\{\chi_{\ell,h}^b\}_{\ell=1}^{N_b}\). Since \(V_h^a\subset \widetilde H^s(W_a)\), \(V_h^b\subset \widetilde H^s(W_b)\), and \(\operatorname{dist}(W_a,W_b)>0\), the source and observation supports are separated.

Because \(V_h^a,V_h^b\subset V_h^W\), the measured bilinear form
\(\mathcal D_h^\delta\) from \eqref{eq:dn-noise-discrete} can be evaluated at
the transformed exterior data
\(I_h^a(a_h^\delta F_{j,h}^a)\) and
\(I_h^b(a_h^\delta\chi_{\ell,h}^b)\). Define
\[
    \mathsf S_h^\delta
    :=
    \bigl(\mathsf S_{\ell j,h}^{\delta}\bigr)_{
    \ell=1,\ldots,N_b;\,j=1,\ldots,J_a}
    \in \mathbb R^{N_b\times J_a}
\]
by
\begin{equation}\label{eq:sep_re_data}
    \mathsf S_{\ell j,h}^{\delta}
    :=
    \mathcal D_h^\delta\!\left(
        I_h^a(a_h^\delta F_{j,h}^a),
        I_h^b(a_h^\delta\chi_{\ell,h}^b)
    \right).
\end{equation}
The exact reduced separated data are
\begin{equation}\label{eq:separated-exact-data}
    \mathsf S_{\ell j,h}^{e}
    :=
    \big\langle
        \Lambda_{Q_\gamma}^{\Omega}F_{j,h}^a,
        \chi_{\ell,h}^b
    \big\rangle .
\end{equation}
By \eqref{eq:separated-identity}, this is equivalently \(
    \mathsf S_{\ell j,h}^{e} = \big\langle  \Lambda_\gamma(a_\gamma F_{j,h}^a),
        a_\gamma\chi_{\ell,h}^b\big\rangle \). By Lemma~\ref{lem:discrete-multiplier-consistency} and the equivalence of the
\(H^s(\mathbb R^d)\)- and \(W^s(\mathbb R^d)\)-norms, after enlarging the
constant in \(\varepsilon_U^a(h,\delta)\), we have
\begin{equation}\label{eq:eps-a-Ws-version}
    \|I_h^U(a_h^\delta\phi_h)-a_\gamma\phi_h\|_{W^s(\mathbb R^d)}
    \le
    \varepsilon_U^a(h,\delta)\|\phi_h\|_{W^s(\mathbb R^d)},
    \qquad
    \phi_h\in V_h^U,\quad U\in\{W_a,W_b\}.
\end{equation}
Set
\[
    M_a:=
    \sup_{0\ne F\in\widetilde H^s(W_a)}
    \frac{\|a_\gamma F\|_{W^s(\mathbb R^d)}}{\|F\|_{W^s(\mathbb R^d)}},
    \qquad
    M_b:=
    \sup_{0\ne G\in\widetilde H^s(W_b)}
    \frac{\|a_\gamma G\|_{W^s(\mathbb R^d)}}{\|G\|_{W^s(\mathbb R^d)}} .
\]
By Assumption~\ref{ass:liouville-regularity}(iii), \(a_\gamma\in\mathcal M_s(W_a)\cap\mathcal M_s(W_b)\); hence \(M_a,M_b<\infty\), after using the equivalence of the \(H^s\)- and \(W^s\)-norms. Since \(\|[f]\|_X\le \|f\|_{H^s(\mathbb R^d)}\le C\|f\|_{W^s(\mathbb R^d)}\), there is a constant \(C_\Lambda>0\), independent of \(h\) and \(\delta\), such that
\begin{equation}\label{eq:DN-Ws-continuity}
    \big|\langle \Lambda_\gamma [f],[g] \rangle \big|
    \le  C_\Lambda
    \|f\|_{W^s(\mathbb R^d)}
    \|g\|_{W^s(\mathbb R^d)}
    \qquad
    \forall f,g\in H^s(\mathbb R^d).
\end{equation}

\begin{lemma}[Consistency of the reduced separated data]
\label{lem:reduced-separated-data-consistency}
Let \(a_h^\delta\) denote the smoothed exterior multiplier constructed in
Lemma~\ref{lem:discrete-multiplier-consistency}, with \(\sigma_h\) chosen as
there. Assume the finite-dimensional DN noise model \eqref{eq:dn-noise-discrete}.
Then, for every \(F_h\in V_h^a\) and \(G_h\in V_h^b\), the following holds:
\begin{equation}\label{eq:reduced-data-functional-consistency}
\begin{aligned}
\left|\mathcal D_h^\delta\!\left(I_h^a(a_h^\delta F_h), I_h^b(a_h^\delta G_h) \right)- \big\langle \Lambda_{Q_\gamma}^{\Omega}F_h,G_h\big\rangle
\right|\le \eta_{ab}(h,\delta) \|F_h\|_{W^s(\mathbb R^d)} \|G_h\|_{W^s(\mathbb R^d)},
\end{aligned}
\end{equation}
where, writing \(
    \varepsilon_a:=\varepsilon_{W_a}^a(h,\delta)\) and \(\varepsilon_b:=\varepsilon_{W_b}^a(h,\delta)\), we take
\begin{equation}\label{eq:eta-ab-explicit}
    \eta_{ab}(h,\delta) := \delta(M_a+\varepsilon_a)(M_b+\varepsilon_b)
    + C_\Lambda\Big(\varepsilon_a(M_b+\varepsilon_b)+M_a\varepsilon_b
    \Big),
\end{equation}
and the constant \(C_\Lambda\) is from \eqref{eq:DN-Ws-continuity}. In particular, if  \(\delta\to0\), \(\varepsilon_{W_a}^a(h,\delta)\to0\), and \(\varepsilon_{W_b}^a(h,\delta)\to0\), then \(\eta_{ab}(h,\delta)\to0\). Moreover, for all sufficiently small \(h\) and \(\delta\), we have \( \eta_{ab}(h,\delta) \le C\bigl( \delta+\varepsilon_a +\varepsilon_b \bigr)\).
Consequently, for the matrix entries defined in
\eqref{eq:sep_re_data} and \eqref{eq:separated-exact-data}, there holds
\[
\left|
\mathsf S_{\ell j,h}^{\delta}-\mathsf S_{\ell j,h}^{e}
\right|
\le
\eta_{ab}(h,\delta)
\|F_{j,h}^a\|_{W^s(\mathbb R^d)}
\|\chi_{\ell,h}^b\|_{W^s(\mathbb R^d)} .
\]
\end{lemma}

\begin{proof}
Fix \(F_h\in V_h^a\) and \(G_h\in V_h^b\), and define \((f_h^\delta,g_h^\delta) : = (I_h^a(a_h^\delta F_h),I_h^b(a_h^\delta G_h))\). Then \(f_h^\delta,g_h^\delta\in V_h^W\), so \(\mathcal D_h^\delta(f_h^\delta,g_h^\delta)\) is well defined. Also
\(a_\gamma F_h\in\widetilde H^s(W_a)\) and \(a_\gamma G_h\in\widetilde H^s(W_b)\). Since the
supports are separated, the separated Liouville identity
\eqref{eq:separated-identity} holds. Hence
\[
\mathcal D_h^\delta(f_h^\delta,g_h^\delta)- \big\langle \Lambda_{Q_\gamma}^{\Omega}F_h,G_h\big\rangle = 
\mathcal D_h^\delta(f_h^\delta,g_h^\delta)
- \big\langle \Lambda_\gamma f_h^\delta,g_h^\delta\big \rangle 
+ \big\langle \Lambda_\gamma f_h^\delta,g_h^\delta\big \rangle 
- \big\langle \Lambda_\gamma a_\gamma F_h,a_\gamma G_h\big \rangle .
\]
The first two terms on the right hand side are controlled directly by \eqref{eq:dn-noise-discrete}:
\[
\left|
\mathcal D_h^\delta(f_h^\delta,g_h^\delta) -
\big\langle \Lambda_\gamma f_h^\delta,g_h^\delta\big \rangle 
\right|
\le \delta \|f_h^\delta\|_{W^s(\mathbb R^d)} \|g_h^\delta\|_{W^s(\mathbb R^d)}.
\]
By \eqref{eq:eps-a-Ws-version} and the definitions of \(M_a\) and \(M_b\), \( \|f_h^\delta\|_{W^s(\mathbb R^d)} \le (M_a+\varepsilon_a)\|F_h\|_{W^s(\mathbb R^d)}\) and
\(\|g_h^\delta\|_{W^s(\mathbb R^d)}  \le (M_b+\varepsilon_b)\|G_h\|_{W^s(\mathbb R^d)}\). Thus, we have
\[
    \left|
\mathcal D_h^\delta(f_h^\delta,g_h^\delta) -
\big\langle \Lambda_\gamma f_h^\delta,g_h^\delta\big \rangle 
\right| \le \delta(M_a+\varepsilon_a)(M_b+\varepsilon_b)
    \|F_h\|_{W^s(\mathbb R^d)}\|G_h\|_{W^s(\mathbb R^d)}.
\]
For the last two terms, we use bilinearity and \eqref{eq:DN-Ws-continuity}, which gives
\[
\begin{aligned}
&\left|\big\langle\Lambda_\gamma f_h^\delta,g_h^\delta\big\rangle-\big\langle\Lambda_\gamma a_\gamma F_h,a_\gamma G_h\big\rangle\right|\\
&\quad\le C_\Lambda\|f_h^\delta-a_\gamma F_h\|_{W^s(\mathbb R^d)}\|g_h^\delta\|_{W^s(\mathbb R^d)}+C_\Lambda\|a_\gamma F_h\|_{W^s(\mathbb R^d)}\|g_h^\delta-a_\gamma G_h\|_{W^s(\mathbb R^d)}.
\end{aligned}
\]
Again by \eqref{eq:eps-a-Ws-version}, \(
\|f_h^\delta-a_\gamma F_h\|_{W^s(\mathbb R^d)}\le \varepsilon_a\|F_h\|_{W^s(\mathbb R^d)}\) and 
\(\|g_h^\delta-a_\gamma G_h\|_{W^s(\mathbb R^d)}\le \varepsilon_b\|G_h\|_{W^s(\mathbb R^d)}\), and by the definition of \(M_a\), \( \|a_\gamma F_h\|_{W^s(\mathbb R^d)}\le M_a\|F_h\|_{W^s(\mathbb R^d)}\).
Together with the bound for \(\|g_h^\delta\|_{W^s(\mathbb R^d)}\), this yields
\[
\left| \big\langle \Lambda_\gamma f_h^\delta,g_h^\delta\big\rangle
- \big\langle \Lambda_\gamma a_\gamma F_h,a_\gamma G_h\big\rangle
\right| \le C_\Lambda\Big(\varepsilon_a(M_b+\varepsilon_b)+M_a\varepsilon_b
\Big) \|F_h\|_{W^s(\mathbb R^d)}\|G_h\|_{W^s(\mathbb R^d)}.
\]
Combining the two bounds proves \eqref{eq:reduced-data-functional-consistency}
with \(\eta_{ab}\) given by \eqref{eq:eta-ab-explicit}. The convergence and the
simplified bound follow immediately from
\(\delta\to0\), \(\varepsilon_a\to0\), and \(\varepsilon_b\to0\). Finally, taking
\(F_h=F_{j,h}^a\) and \(G_h=\chi_{\ell,h}^b\) gives the entrywise matrix
estimate.
\end{proof}

Since \(\operatorname{dist}(W_b,\Omega)>0\), define the continuous exterior observation operator by
\begin{equation}\label{eq:Lb-continuous}
    \mathcal L_b:L^2(\Omega)\to L^2(W_b),
    \qquad (\mathcal L_b z)(x) := -C_{d,s}\int_\Omega \frac{z(y)}{|x-y|^{d+2s}}\,dy , \quad x\in W_b .
\end{equation}
The separated kernel is smooth; in particular, \(\mathcal L_b:L^2(\Omega)\to H^m(W_b)\) is bounded for every integer \(m\ge0\). Moreover, if \(z\in\widetilde H^s(\Omega)\), then \(\mathcal L_bz=((-\Delta)^sz)|_{W_b}\) in the distributional sense. The map \(\mathcal L_b:\widetilde H^s(\Omega)\to L^2(W_b)\) is compact and injective by exterior unique continuation; see \cite[Lem.~2.2 and Thm.~2]{Ghosh2020uniqueness} and \cite[Lem.~2.3]{DwivediRailoRuppFCP}. 

For each discrete exterior source \(F_{j,h}^a\), let \(v_{j,h}^{e}\in H^s(\mathbb R^d)\) denote the exact Liouville-transformed Schr\"odinger state with exterior value \(F_{j,h}^a\), and write
\begin{equation*}
v_{j,h}^{e} = F_{j,h}^a+z_{j,h}^{e}, \qquad z_{j,h}^{e}\in\widetilde H^s(\Omega). 
\end{equation*}
Thus
\[
    A_R(v_{j,h}^{e},\phi) + \int_\Omega Q_\gamma v_{j,h}^{e}\phi\,dx  = 0 \qquad
    \forall\phi\in\widetilde H^s(\Omega).
\]
For a linear functional $A_h:V_h^b\to\mathbb R$, define
\begin{equation}\label{eq:discrete-negative-norm}
    \|A_h\|_{-s,b,h}
    :=
    \sup_{0\ne G_h\in V_h^b}
    \frac{|A_h(G_h)|}
         {\|G_h\|_{W^s(\mathbb R^d)}}.
\end{equation}
 Since \(V_h^b\) is finite-dimensional and is equipped with the \(W^s(\mathbb R^d)\)-inner product, every functional \(A_h\in (V_h^b)^*\) has a unique Riesz representative \(R_{b,h}A_h\in V_h^b\) defined by
\[
    (R_{b,h}A_h,G_h)_{W^s(\mathbb R^d)} = A_h(G_h) \qquad
    \forall G_h\in V_h^b.
\]
We use the induced Hilbert structure on \((V_h^b)^*\), \(
    (A_h,B_h)_{-s,b,h} := (R_{b,h}A_h,R_{b,h}B_h)_{W^s(\mathbb R^d)}\).
The corresponding norm is exactly the dual norm in
\eqref{eq:discrete-negative-norm}. For each exterior source $F_{j,h}^a$, define $\mu_{j,h}^\delta,\mu_{j,h}^e\in(V_h^b)^*$ by
\begin{align*} 
\langle\mu_{j,h}^\delta,G_h\rangle &:= \mathcal D_h^\delta\!\left(I_h^a(a_h^\delta F_{j,h}^a),I_h^b(a_h^\delta G_h)\right)-A_R(F_{j,h}^a,G_h), \\ \langle\mu_{j,h}^e,G_h\rangle &:= \big\langle\Lambda_{Q_\gamma}^{\Omega}F_{j,h}^a,G_h\big\rangle-A_R(F_{j,h}^a,G_h) = A_R(z_{j,h}^e,G_h), 
\end{align*}
for every $G_h\in V_h^b$. Lemma~\ref{lem:reduced-separated-data-consistency} directly gives
\begin{equation*}
\|\mu_{j,h}^\delta-\mu_{j,h}^e\|_{-s,b,h} \le \eta_{ab}(h,\delta)\|F_{j,h}^a\|_{W^s(\mathbb R^d)} =: \eta_{\mu,j}(h,\delta). \end{equation*}

Define \( \mathcal L_{b,h}:V_h^\Omega\to(V_h^b)^*\) by
\begin{equation*}
\langle\mathcal L_{b,h}z_h,G_h\rangle := A_R(z_h,G_h), \qquad z_h\in V_h^\Omega,\quad G_h\in V_h^b. 
\end{equation*}
For  \(z_h\in V_h^\Omega\) and  \(G_h\in V_h^b\), we have
\begin{equation}\label{eq:Lbh-smooth-kernel}
    A_R(z_h,G_h)
    = -C_{d,s} \int_{W_b}\int_\Omega \frac{z_h(y)G_h(x)}
    {|x-y|^{d+2s}}\,dy\,dx = (\mathcal L_bz_h,G_h)_{L^2(W_b)} .
\end{equation}
We regularize the compact observation operator by quadratic Tikhonov minimization; see \cite[Sec.~5]{BenningBurger2018}. The penalty is centered at the zero-potential correction, which is generally nonzero for exterior forcing. For each \(j=1,\ldots,J_a\), define \(z_{0,j,h}\in V_h^\Omega\) by
\begin{equation}\label{eq:z0-reference-state}
  A_R(F_{j,h}^a+z_{0,j,h},\varphi_h)=0
  \qquad
  \forall \varphi_h\in V_h^\Omega .
\end{equation}
For \(\alpha_z>0\), define the generalized Tikhonov state reconstruction by
\begin{equation}\label{eq:z-state-recovery}
  \zeta_{j,h}^{\alpha_z,\delta} :=
  \operatorname*{argmin}_{\zeta_h\in V_h^\Omega} \left\{
  \frac12 \|\mathcal L_{b,h}(z_{0,j,h}+\zeta_h)-\mu_{j,h}^\delta\|_{-s,b,h}^2 +
  \frac{\alpha_z}{2}\|\zeta_h\|_{W^s(\mathbb R^d)}^2\right\}.
\end{equation}
Set \(z_{j,h}^{\alpha_z,\delta}:=z_{0,j,h}+\zeta_{j,h}^{\alpha_z,\delta}\)  and  \(v_{j,h}^{\alpha_z,\delta}:=F_{j,h}^a+z_{j,h}^{\alpha_z,\delta}\). When no ambiguity is possible, we write \( z_{j,h}:=z_{j,h}^{\alpha_z,\delta}\) and \(v_{j,h}:=v_{j,h}^{\alpha_z,\delta}\).

Apply Lemma~\ref{lem:abstract-shifted-tikhonov} to $z_h=z_{0,j,h}+\zeta_h$ with $\mathsf X=E=V_h^\Omega$ in the $W^s(\mathbb R^d)$-norm, $\mathsf Y=(V_h^b)^*$ in the norm $\|\cdot\|_{-s,b,h}$, $T=\mathcal L_{b,h}$, $b=\mu_{j,h}^\delta$, $x_c=z_{0,j,h}$, and $\alpha=\alpha_z$. The unique minimizer of \eqref{eq:z-state-recovery} satisfies
\begin{equation*} \begin{aligned}
\bigl(\mathcal L_{b,h}(z_{0,j,h}+\zeta_{j,h}^{\alpha_z,\delta}),\mathcal L_{b,h}\theta_h\bigr)_{-s,b,h}+\alpha_z\bigl(\zeta_{j,h}^{\alpha_z,\delta},\theta_h\bigr)_{W^s(\mathbb R^d)}=\bigl(\mu_{j,h}^\delta,\mathcal L_{b,h}\theta_h\bigr)_{-s,b,h},\quad\theta_h\in V_h^\Omega. 
\end{aligned} \end{equation*}
The normal-equation matrix is symmetric positive definite. Let $\zeta_{j,h}^{\alpha_z,e}$ be the minimizer for $\mu_{j,h}^e$ and set $z_{j,h}^{\alpha_z,e}=z_{0,j,h}+\zeta_{j,h}^{\alpha_z,e}$. The common center and \eqref{eq:abstract-tikhonov-data-stability} give
\begin{equation}\label{eq:z-data-stability}
\begin{aligned}
\|\zeta_{j,h}^{\alpha_z,\delta}-\zeta_{j,h}^{\alpha_z,e}\|_{W^s(\mathbb R^d)}
 =\|z_{j,h}^{\alpha_z,\delta}-z_{j,h}^{\alpha_z,e}\|_{W^s(\mathbb R^d)}\le\frac{\|\mu_{j,h}^\delta-\mu_{j,h}^e\|_{-s,b,h}}{2\sqrt{\alpha_z}}
\le\frac{\eta_{\mu,j}(h,\delta)}{2\sqrt{\alpha_z}}.
\end{aligned}
\end{equation}

Let \( \eta_\mu(h,\delta) := \left( \frac1{J_a}\sum_{j=1}^{J_a}
 \eta_{\mu,j}(h,\delta)^2\right)^{1/2}\),
and define
\begin{equation*} 
\rho_{z,h}(\alpha_z) := \left(\frac1{J_a}\sum_{j=1}^{J_a}\|z_{j,h}^{\alpha_z,e}-z_{j,h}^{e}\|_{W^s(\mathbb R^d)}^2\right)^{1/2}\quad\text{and}\quad \mathcal E_{v,h}:=\left(\frac1{J_a}\sum_{j=1}^{J_a}\|v_{j,h}-v_{j,h}^{e}\|_{W^s(\mathbb R^d)}^2\right)^{\frac12}. 
\end{equation*}
Since the exact and reconstructed states have the same exterior value, \eqref{eq:z-data-stability} and the triangle inequality give
\begin{equation}\label{eq:state-error-dual}
 \mathcal E_{v,h} \le
 \frac{\eta_\mu(h,\delta)}{2\sqrt{\alpha_z}}
 +\rho_{z,h}(\alpha_z).
\end{equation}
Thus it is sufficient to choose
\[
 \alpha_z\to0,
 \qquad
 \frac{\eta_\mu(h,\delta)^2}{\alpha_z}\to0,
 \qquad
 \rho_{z,h}(\alpha_z)\to0.
\]

We next compute a regularized $L^2(\Omega)$-representation of the
interior fractional Laplacian. Let $\alpha_w>0$. For each reconstructed
state $v_{j,h}$, define $w_{j,h}\in V_h^\Omega$ by
\begin{equation}\label{eq:regularized-weak-laplacian}
 (w_{j,h},\phi_h)_{L^2(\Omega)}
 +
 \alpha_w
 (w_{j,h},\phi_h)_{W^s(\mathbb R^d)}
 =
 A_R(v_{j,h},\phi_h)
 \qquad
 \forall\phi_h\in V_h^\Omega .
\end{equation}
The Hessian in \eqref{eq:regularized-weak-laplacian} is the one in Lemma~\ref{lem:abstract-shifted-tikhonov} with $\mathsf X=E=V_h^\Omega$ in the $W^s(\mathbb R^d)$-norm, $\mathsf Y=L^2(\Omega)$, $T$ the embedding, $b=0$, $x_c=0$, and $\alpha=\alpha_w$; its load is $\ell_{j,h}(\phi_h)=A_R(v_{j,h},\phi_h)$. Hence the system is symmetric positive definite and has a unique solution. The $W^s$-penalty regularizes the unregularized Riesz representation of \cite{DwivediRailoRuppFCP}, recovered when $\alpha_w=0$. Define the exact-data solution $w_{j,h}^{e,\alpha_w}\in V_h^\Omega$ by
\begin{equation}\label{eq:regularized-weak-laplacian-exact}
\begin{aligned}(w_{j,h}^{e,\alpha_w},\phi_h)_{L^2(\Omega)}
 + \alpha_w (w_{j,h}^{e,\alpha_w},\phi_h)_{W^s(\mathbb R^d)}
 = A_R(v_{j,h}^e,\phi_h) = (w_{j,h}^e,\phi_h)_{L^2(\Omega)}
 \qquad
 \forall\phi_h\in V_h^\Omega ,
\end{aligned}
\end{equation}
where \( w_{j,h}^e=-Q_\gamma v_{j,h}^e\). Set
\begin{equation*} 
\rho_{w,h}(\alpha_w) := \left(\frac1{J_a}\sum_{j=1}^{J_a}\|w_{j,h}^{e,\alpha_w}-w_{j,h}^e\|_{L^2(\Omega)}^2\right)^{\frac12}.
\end{equation*}
The consistency terms \(\rho_{z,h}(\alpha_z)\) and \(\rho_{w,h}(\alpha_w)\) collect exact-data discretization and regularization errors. Suppose that \(J_a\) is fixed, \(F_{j,h}^a\to F_j^a\) in \(W^s(\mathbb R^d)\), the families $\{V_h^\Omega\}_h$ and $\{V_h^b\}_h$ satisfy $\inf_{p_h\in V_h^\Omega}\lVert u-p_h\rVert_{W^s(\mathbb R^d)}\to0$ for every $u\in\widetilde H^s(\Omega)$ and $\inf_{G_h\in V_h^b}\lVert G-G_h\rVert_{W^s(\mathbb R^d)}\to0$ for every $G\in\widetilde H^s(W_b)$, respectively, and
\[
\alpha_z\to0,\qquad \frac{1}{J_a\alpha_z}\sum_{j=1}^{J_a}\inf_{p_h\in V_h^\Omega}\|z_{j,h}^e-p_h\|_{W^s(\mathbb R^d)}^2\to0.
\]
For a best approximant \(p_{j,h}\in V_h^\Omega\), exact-data minimality and \eqref{eq:Lbh-smooth-kernel} give
\[
\|\mathcal L_{b,h}z_{j,h}^{\alpha_z,e}-\mu_{j,h}^e\|_{-s,b,h}^2+\alpha_z\|z_{j,h}^{\alpha_z,e}-z_{0,j,h}\|_{W^s(\mathbb R^d)}^2 \le C\|p_{j,h}-z_{j,h}^e\|_{W^s(\mathbb R^d)}^2+\alpha_z\|p_{j,h}-z_{0,j,h}\|_{W^s(\mathbb R^d)}^2.
\]
Let $z_j^e$ be the exact correction for $F_j^a$, and let $z_{0,j}\in\widetilde H^s(\Omega)$ solve $A_R(F_j^a+z_{0,j},\phi)=0$ for all $\phi\in\widetilde H^s(\Omega)$. The coupling makes the exact-data residual vanish and bounds the reconstructed shifts. Forward continuity gives $z_{j,h}^e\to z_j^e$ in $W^s(\mathbb R^d)$; C\'ea's lemma for \eqref{eq:z0-reference-state} and $F_{j,h}^a\to F_j^a$ give $z_{0,j,h}\to z_{0,j}$ in the same space. Density of $V_h^b$ and injectivity of $\mathcal L_b$ identify every weak limit of $z_{j,h}^{\alpha_z,e}$ with $z_j^e$. Minimality and weak lower semicontinuity give convergence of the shifted norms, hence $z_{j,h}^{\alpha_z,e}-z_{j,h}^e\to0$ in $W^s(\mathbb R^d)$ and $\rho_{z,h}(\alpha_z)\to0$.

\begin{proposition}[Logarithmic control of the state bias]\label{prop:state-logarithmic-bias}
Let $J_a$ be fixed, let $F_{j,h}^a\to F_j^a$ in $W^s(\mathbb R^d)$, and use the full fitted observation spaces $V_h^b$ on shape-regular, quasi-uniform meshes. Define
\[
a_{z,h}:=\left(\frac1{J_a}\sum_{j=1}^{J_a}\inf_{p_h\in V_h^\Omega}\|z_{j,h}^e-p_h\|_{W^s(\mathbb R^d)}^2\right)^{1/2}.
\]
If $\alpha_z\to0$ and $a_{z,h}^2/\alpha_z\to0$, then, for sufficiently small $h$,
\begin{equation}\label{eq:state-logarithmic-bias}
\rho_{z,h}(\alpha_z)\le C\left[\frac{a_{z,h}}{\sqrt{\alpha_z}}+\sqrt{a_{z,h}}+\left\{\log\left(2+\frac1{a_{z,h}+\sqrt{\alpha_z}+h^s}\right)\right\}^{-\sigma_z/2}\right],
\end{equation}
where $C,\sigma_z>0$ are independent of $h,\alpha_z$ along the coupling. Equation~\eqref{eq:state-error-dual} adds the noisy-data contribution.
\end{proposition}

\begin{proof}
Fix $j$ and set $z^e=z_{j,h}^e$, $z^\alpha=z_{j,h}^{\alpha_z,e}$, $z_0=z_{0,j,h}$, and $e=z^\alpha-z^e$. For a best approximant $p_h\in V_h^\Omega$, put $a_j=\|p_h-z^e\|_{W^s(\R^{d})}$. Minimality and $a_j^2/\alpha_z\to0$ bound $\|e\|_{W^s(\R^{d})}$ uniformly and give
\[
r_j:=\sup_{0\ne G_h\in V_h^b}\frac{|A_R(e,G_h)|}{\|G_h\|_{W^s(\R^{d})}}\le C(a_j+\sqrt{\alpha_z}).
\]
Choose a ball $W_b'\Subset W_b$. For $\phi\in\widetilde H^s(W_b')$, insert $\mathcal I_h^b\phi$ in the dual pairing; \eqref{eq:local-averaged-interpolation} and the separated kernel \eqref{eq:Lb-continuous} imply
\[
\|\mathcal L_be\|_{H^{-s}(W_b')}\le C(r_j+h^s\|e\|_{L^2(\R^{d})})\le C(a_j+\sqrt{\alpha_z}+h^s).
\]
Since $e$ is supported in $\overline\Omega$, \cite[Prop.~6.1]{Ghosh2020uniqueness}, with $s'=0$, yields
\[
\|e\|_{L^2(\R^{d})}\le C\left[\log\left(2+(a_j+\sqrt{\alpha_z}+h^s)^{-1}\right)\right]^{-\sigma_z}.
\]
Set $q=z^e-z_0\in\widetilde H^s(\Omega)$ and $\theta_h=z^\alpha-p_h\in V_h^\Omega$. Subtracting \eqref{eq:z0-reference-state} from the exact state equation gives $A_R(q,\theta_h)=-(Q_\gamma v_{j,h}^e,\theta_h)_{L^2(\Omega)}$. Because $q$ and $\theta_h$ vanish outside $\Omega$, \eqref{eq:Ws-norm-convention} gives
\[
 (q,\theta_h)_{W^s(\mathbb R^d)} =(q-Q_\gamma v_{j,h}^e,\theta_h)_{L^2(\Omega)}.
\]
Expanding the penalties and using $\|q-Q_\gamma v_{j,h}^e\|_{L^2(\Omega)}\le C$ yield
\[
\|e\|_{W^s(\R^{d})}^2\le C\frac{a_j^2}{\alpha_z}+a_j^2-2(q,\theta_h)_{W^s(\R^{d})}\le C\left(\frac{a_j^2}{\alpha_z}+a_j+\|e\|_{L^2(\R^{d})}\right),
\]
where $\|\theta_h\|_{L^2(\R^{d})}\le\|e\|_{L^2(\R^{d})}+a_j$. Taking square roots and averaging, with $a_j\le\sqrt{J_a}\,a_{z,h}$, proves \eqref{eq:state-logarithmic-bias}.
\end{proof}

Similarly, the variational characterization of \(w_{j,h}^{e,\alpha_w}\) gives, for every \(\xi_{j,h}\in V_h^\Omega\),
\[
    \|w_{j,h}^{e,\alpha_w}-w_{j,h}^e\|_{L^2(\Omega)}^2 \le \|\xi_{j,h}-w_{j,h}^e\|_{L^2(\Omega)}^2 + \alpha_w\|\xi_{j,h}\|_{W^s(\mathbb R^d)}^2.
\]
If $J_a$ is fixed, $F_{j,h}^a\to F_j^a$ in $W^s(\mathbb R^d)$, and $\alpha_w\to0$, then \eqref{eq:VOmega-approximation} implies $\rho_{w,h}(\alpha_w)\to0$. Indeed, forward continuity gives $w_{j,h}^e\to w_j^e:=-Q_\gamma v_j^e$ in $L^2(\Omega)$. For fixed $\xi_j\in C_c^\infty(\Omega)$, choose $\xi_{j,h}\in V_h^\Omega$ converging to $\xi_j$ in $W^s(\mathbb R^d)$. Taking the limsup in the comparison estimate and then approximating each $w_j^e$ by $\xi_j$ in $L^2(\Omega)$ proves convergence.

\begin{lemma}[Consistency of the regularized weak fractional Laplacian]
\label{lem:regularized-weak-laplacian}
There exists a constant $C>0$, independent of $h,\delta,\alpha_z,\alpha_w$, and $J_a$, such that
\begin{equation}\label{eq:laplacian-consistency}
\left( \frac1{J_a} \sum_{j=1}^{J_a} \|w_{j,h}-w_{j,h}^e\|_{L^2(\Omega)}^2
\right)^{1/2} \le \rho_{w,h}(\alpha_w)+ C\alpha_w^{-1/2}\mathcal E_{v,h}.
\end{equation}
\end{lemma}

\begin{proof}
Set $\theta_{j,h}=w_{j,h}-w_{j,h}^{e,\alpha_w}$. Equations \eqref{eq:regularized-weak-laplacian} and \eqref{eq:regularized-weak-laplacian-exact} give
\[
(\theta_{j,h},\phi_h)_{L^2(\Omega)}+\alpha_w(\theta_{j,h},\phi_h)_{W^s(\mathbb R^d)}=A_R(v_{j,h}-v_{j,h}^e,\phi_h),\qquad\phi_h\in V_h^\Omega.
\]
For the embedding $T:V_h^\Omega\hookrightarrow L^2(\Omega)$, continuity of $A_R$ and \eqref{eq:abstract-tikhonov-load-stability} imply $\|\theta_{j,h}\|_{L^2(\Omega)}\le C\alpha_w^{-1/2}\|v_{j,h}-v_{j,h}^e\|_{W^s(\mathbb R^d)}$. The triangle inequality and root-mean-square over $j$ prove \eqref{eq:laplacian-consistency}.
\end{proof}

Let \(\mathcal T_h^\Omega := \{K\in\mathcal T_h:K\subset\Omega\}\). Since the mesh is fitted to \(\Omega\), these elements form a partition of \(\Omega\). We use the elementwise constant coefficient space
\begin{equation}\label{eq:PhOmega-definition}
    P_h^\Omega := \left\{ p_h\in L^2(\Omega): p_h|_K\in\mathbb P_0(K) \quad\forall K\in\mathcal T_h^\Omega \right\}.
\end{equation}

For \(\alpha_Q>0\), define the discrete reconstruction of \(Q_\gamma\) by
\begin{equation}\label{eq:Q-reconstruction}
    Q_h^\Omega := \operatorname*{argmin}_{p_h\in P_h^\Omega}
    \left\{ \frac{1}{2J_a} \sum_{j=1}^{J_a}  \|w_{j,h}+p_hv_{j,h}\|_{L^2(\Omega)}^2
        + \frac{\alpha_Q}{2} \|p_h\|_{L^2(\Omega)}^2
    \right\}.
\end{equation}
The factor \(J_a^{-1}\) keeps the scale of the data term independent of the number of sources. If \(J_a\) is fixed, it may equivalently be absorbed into the definition of \(\alpha_Q\). The functional \eqref{eq:Q-reconstruction} is a stabilized version of the formal identity \( Q_\gamma=-\frac{(-\Delta)^sv}{v}\). It avoids pointwise division and permits several states to contribute to the coefficient recovery; compare \cite[Eq.~(3.43) and Thm.~3.12]{DwivediRailoRuppFCP}.

Apply Lemma~\ref{lem:abstract-shifted-tikhonov} with $\mathsf X=E=P_h^\Omega$ in the $L^2(\Omega)$-norm, $\mathsf Y=(L^2(\Omega))^{J_a}$ with $\|(y_j)_j\|_{\mathsf Y}^2=J_a^{-1}\sum_j\|y_j\|_{L^2(\Omega)}^2$, $T p_h=(p_hv_{j,h})_j$, $b=(-w_{j,h})_j$, $x_c=0$, and $\alpha=\alpha_Q$. The unique minimizer in \eqref{eq:Q-reconstruction} satisfies
\begin{equation}\label{eq:Q-normal-equation}
\frac1{J_a}\sum_{j=1}^{J_a}\int_\Omega(w_{j,h}+Q_h^\Omega v_{j,h})v_{j,h}\psi_h\,dx+\alpha_Q\int_\Omega Q_h^\Omega\psi_h\,dx=0,\qquad\psi_h\in P_h^\Omega.
\end{equation}
The matrix of \eqref{eq:Q-normal-equation} is symmetric positive definite in every basis of $P_h^\Omega$.

We now turn to convergence and separate the errors arising from state recovery, weak-Laplacian regularization, coefficient approximation, and coefficient visibility. For the coefficient error analysis, let \(\Pi_h^Q:L^2(\Omega)\to P_h^\Omega\) and \(\Pi_h^\Omega:L^2(\Omega)\to V_h^\Omega\) denote the \(L^2(\Omega)\)-orthogonal projections, and set \(Q_{I,h}:=\Pi_h^Q Q_\gamma\). Since \(P_h^\Omega\) consists of elementwise constants,  \(\|Q_{I,h}\|_{L^\infty(\Omega)}\le\|Q_\gamma\|_{L^\infty(\Omega)}\) and \(\|Q_{I,h}\|_{L^2(\Omega)}  \le\|Q_\gamma\|_{L^2(\Omega)}\). Define
\begin{align*}
M_{v,h} := \left(\frac1{J_a}\sum_{j=1}^{J_a}\|v_{j,h}^{e}\|_{W^s(\mathbb R^d)}^2\right)^{1/2}, \qquad \rho_{Q,h}^{(p)} := \|Q_{I,h}-Q_\gamma\|_{L^p(\Omega)}, \quad p\in\{2,d/s\}. 
\end{align*}
All these quantities may depend on \(h,\delta,\alpha_z\), and the source family; this dependence is suppressed to keep the notation readable.

The stability of the unclipped reconstruction \(Q_h^\Omega\) depends on how strongly the states illuminate each coefficient cell. Since \(P_h^\Omega\) is piecewise constant, define
\begin{equation}\label{eq:Q-visibility}
c_{Q,h}:=\inf_{0\ne p_h\in P_h^\Omega}\frac{J_a^{-1}\sum_{j=1}^{J_a}\|p_hv_{j,h}\|_{L^2(\Omega)}^2}{\|p_h\|_{L^2(\Omega)}^2}=\min_{K\in\mathcal T_h^\Omega}\frac{1}{J_a|K|}\sum_{j=1}^{J_a}\int_K|v_{j,h}|^2\,dx.
\end{equation}
Thus, for every \(p_h\in P_h^\Omega\),
\[
J_a^{-1}\sum_{j=1}^{J_a}\|p_hv_{j,h}\|_{L^2(\Omega)}^2+\alpha_Q\|p_h\|_{L^2(\Omega)}^2\ge(c_{Q,h}+\alpha_Q)\|p_h\|_{L^2(\Omega)}^2.
\]
An exact state generated by nonzero exterior data cannot vanish on an open coefficient cell: its equation would also give \((-\Delta)^sv=0\) there, contradicting fractional unique continuation \cite[Thm.~1.2]{Ghosh2020calderon}. This qualitative fact gives no mesh-uniform lower bound for the reconstructed states, so Theorem~\ref{thm:Q-whole-domain-error-regularized} retains \(c_{Q,h}\) explicitly.

\begin{theorem}[Whole-domain error estimate for $Q_\gamma$]
\label{thm:Q-whole-domain-error-regularized}
Under Assumption~\ref{ass:liouville-regularity}, let \(Q_h^\Omega\in P_h^\Omega\) be the minimizer of \eqref{eq:Q-reconstruction}, and let
$w_{j,h}$ be defined by
\eqref{eq:regularized-weak-laplacian}. Then
\begin{align}\label{eq:Q-whole-domain-error}
\|Q_h^\Omega-Q_\gamma\|_{L^2(\Omega)} \le{}& \rho_{Q,h}^{(2)} +
\frac{C}{\sqrt{c_{Q,h}+\alpha_Q}}
\Big[ \rho_{w,h}(\alpha_w) +\alpha_w^{-1/2}\mathcal E_{v,h}
 +\|Q_\gamma\|_{L^\infty(\Omega)}\mathcal E_{v,h}
 +M_{v,h}\rho_{Q,h}^{(d/s)}\Big]\nonumber
\\
&+ \frac{\alpha_Q}{c_{Q,h}+\alpha_Q} \|Q_\gamma\|_{L^2(\Omega)}.
\end{align}
The constant $C$ is independent of $h,\delta,\alpha_z,\alpha_w,\alpha_Q$, and $J_a$. Suppose, in addition, that \(\sup_h M_{v,h}<\infty\). Then \(Q_h^\Omega\to Q_\gamma\) in \(L^2(\Omega)\), provided
\begin{equation}\label{eq:Q-convergence}
 \begin{aligned}
\alpha_Q\to0,\quad
     \alpha_w\to0,\quad
     \rho_{Q,h}^{(2)}\to0,\quad \frac{
        \rho_{w,h}(\alpha_w)
        +\alpha_w^{-1/2}\mathcal E_{v,h}
        +\rho_{Q,h}^{(d/s)}
    }{\sqrt{c_{Q,h}+\alpha_Q}}
    \to0,\quad \frac{\alpha_Q}{c_{Q,h}+\alpha_Q}\to0.
\end{aligned}
\end{equation}
\end{theorem}

\begin{proof}
Set $e_h=Q_h^\Omega-Q_{I,h}$, $r_{j,h}=w_{j,h}+Q_{I,h}v_{j,h}$ and $R_h=(J_a^{-1}\sum_j\|r_{j,h}\|_{L^2(\Omega)}^2)^{1/2}$. For the multiplication operator in \eqref{eq:Q-reconstruction}, \eqref{eq:Q-visibility} gives $\sigma^2=c_{Q,h}$. Estimate \eqref{eq:abstract-tikhonov-comparator} with comparison function $Q_{I,h}$ therefore yields
\begin{equation}\label{eq:coefficient-energy-bound}
\|e_h\|_{L^2(\Omega)}\le\frac{R_h}{\sqrt{c_{Q,h}+\alpha_Q}}+\frac{\alpha_Q}{c_{Q,h}+\alpha_Q}\|Q_{I,h}\|_{L^2(\Omega)}.
\end{equation}
The exact identity $w_{j,h}^e+Q_\gamma v_{j,h}^e=0$ yields
\begin{equation}\label{eq:decom_rh}
r_{j,h}=(w_{j,h}-w_{j,h}^e)+Q_{I,h}(v_{j,h}-v_{j,h}^e)+(Q_{I,h}-Q_\gamma)v_{j,h}^e.
\end{equation}
Lemma~\ref{lem:regularized-weak-laplacian}, $\|Q_{I,h}\|_{L^\infty(\Omega)}\le\|Q_\gamma\|_{L^\infty(\Omega)}$, and H\"older--Sobolev with $1/2=s/d+1/p_s^*$ give
\[
R_h\le\rho_{w,h}(\alpha_w)+C\alpha_w^{-1/2}\mathcal E_{v,h}+\|Q_\gamma\|_{L^\infty(\Omega)}\mathcal E_{v,h}+CM_{v,h}\rho_{Q,h}^{(d/s)};
\]
see \cite[Thm.~6.5]{Di2012Hitchhiker}. Substitution into \eqref{eq:coefficient-energy-bound}, followed by $\|Q_{I,h}\|_{L^2(\Omega)}\le\|Q_\gamma\|_{L^2(\Omega)}$ and the triangle inequality, proves \eqref{eq:Q-whole-domain-error}. For $\alpha_w\le1$, $\mathcal E_{v,h}\le\alpha_w^{-1/2}\mathcal E_{v,h}$, so \eqref{eq:Q-convergence} and $\sup_hM_{v,h}<\infty$ give convergence.
\end{proof}

The convergence conditions couple the preceding reconstruction errors and the regularization parameters to the visibility constant \(c_{Q,h}\), which may decrease as \(h\to0\). When \(c_{Q,h}>0\), sufficient conditions, in addition to
\(\alpha_Q\to0\), \(\alpha_w\to0\), and \(\rho_{Q,h}^{(2)}\to0\), are
\[
    \frac{\alpha_Q}{c_{Q,h}}\to0,
    \qquad \frac{
        \rho_{w,h}(\alpha_w)
        +\alpha_w^{-1/2} \mathcal E_{v,h}
        +\rho_{Q,h}^{(d/s)}
    }{\sqrt{c_{Q,h}}}\to0.
\]
Since \(Q_\gamma\in L^\infty(\Omega)\), the cell-average projection satisfies \(\rho_{Q,h}^{(2)}+\rho_{Q,h}^{(d/s)}\to0\). If \(Q_\gamma\in W^{m,d/s}(\Omega)\), \(0<m\le1\), then
\[
    \rho_{Q,h}^{(d/s)}
    \le Ch^m |Q_\gamma|_{W^{m,d/s}(\Omega)}.
\]
The bound \(\sup_h M_{v,h}<\infty\) follows from forward stability when the exterior sources are uniformly bounded in \(W^s(\mathbb R^d)\). The required decay of the errors relative to the visibility remains an additional condition.

\subsection{Recovery of the exterior Liouville tail}
\label{subsec:exterior-tail}
Subsection~\ref{subsec:discrete-Qgamma} provides the approximation
$Q_h^\Omega$ of the interior Liouville potential $Q_\gamma=q_\gamma|_\Omega$. We now recover the remaining part $q_\gamma|_{W_c}$.

In \eqref{eq:overlap-identity}, the first pairing is measured and the second is approximated by solving the Schr\"odinger equation with the reconstructed interior potential. Their difference provides moments of \(q_\gamma\) on \(W_c\).

Let \( \mathcal T_h^c := \{K\in\mathcal T_h:K\subset W_c\}\), and assume that these cells form a partition of $W_c$. As in \eqref{eq:PhOmega-definition}, define
\begin{equation*} 
P_h^c := \left\{p_h\in L^2(W_c):p_h|_K\in\mathbb P_0(K)\quad\text{for every }K\in\mathcal T_h^c\right\}. 
\end{equation*}
For $1\le p<\infty$, the cell-average projection of $q_\gamma\vert _{W_c}\in L^\infty(W_c)$ onto $P_h^c$ converges in $L^p(W_c)$ as the fitted mesh size tends to zero; no positive algebraic rate is asserted under this assumption.

Fix a number $Q_*>\|Q_\gamma\|_{L^\infty(\Omega)}$ and define
\begin{equation}\label{eq:tail-clipped-potential}
 \widehat Q_h^\Omega := \min\{Q_*,\max\{-Q_*,Q_h^\Omega\}\}.
\end{equation}
The clipping map is one-Lipschitz and fixes $Q_\gamma$. Hence $\widehat Q_h^\Omega\to Q_\gamma$ in $L^2(\Omega)$ whenever $Q_h^\Omega\to Q_\gamma$ in $L^2(\Omega)$. Set \(
\varepsilon_{Q,h} := \|\widehat Q_h^\Omega-Q_\gamma\|_{L^{\frac{d}{2s}}(\Omega)}\). Note that \(\|\widehat Q_h^\Omega- Q_\gamma\|_{L^\infty(\Omega)}\le 2Q_*\). If \(d/(2s)\le2\), the finite measure of \(\Omega\) gives \( \varepsilon_{Q,h} \le C\|\widehat Q_h^\Omega-Q_\gamma\|_{L^2(\Omega)}.
\) If \(d/(2s)>2\), interpolation between \(L^2(\Omega)\) and
\(L^\infty(\Omega)\) gives \(\varepsilon_{Q,h} \le C \|\widehat Q_h^\Omega-Q_\gamma\|_{L^2(\Omega)}^{\frac{4s}{d}}\).  Consequently,
\begin{equation}\label{eq:tail-potential-error-from-L2}
 \varepsilon_{Q,h}
 \le
 C\|Q_h^\Omega-Q_\gamma\|_{L^2(\Omega)}^{\theta_Q}, \qquad
 \theta_Q:=\min\left\{1,\frac{4s}{d}\right\}.
\end{equation}
Thus, under the convergence conditions of Theorem~\ref{thm:Q-whole-domain-error-regularized}, we have \(\varepsilon_{Q,h}\to0\).

\begin{proposition}[Clipped-potential convergence for fixed sources]\label{prop:Q-clipped-fixed-sources}
Under Assumption~\ref{ass:liouville-regularity}, let $J_a=J$ be fixed and let $v_j$ be the exact Schr\"odinger states for fixed sources $F_j\in\widetilde H^s(W_a)$. Assume that one $F_j$ is nonnegative and nonzero, and let $Q_*>\|Q_\gamma\|_{L^\infty(\Omega)}$. With $Q_{I,h}=\Pi_h^Q Q_\gamma$, define
\[
\begin{aligned}
V&:=\left(J^{-1}\sum_j|v_j|^2\right)^{1/2},\qquad 
d_{v,h}:=\left(J^{-1}\sum_j\|v_{j,h}-v_j\|_{L^2(\Omega)}^2\right)^{1/2},\\
R_h&:=\left(J^{-1}\sum_j\|w_{j,h}+Q_{I,h}v_{j,h}\|_{L^2(\Omega)}^2\right)^{1/2},
A_h:=\left(J^{-1}\sum_j\|(Q_{I,h}-Q_\gamma)v_j\|_{L^2(\Omega)}^2\right)^{1/2},
\end{aligned}
\]
\[
B_h:=2R_h+\sqrt{\alpha_Q}\|Q_\gamma\|_{L^2(\Omega)}+2Q_*d_{v,h}+A_h.
\]
The clipped minimizer \eqref{eq:tail-clipped-potential} satisfies
\begin{equation}\label{eq:Q-clipped-weighted-levelset}
\|\widehat Q_h^\Omega-Q_\gamma\|_{L^2(\Omega)}\le B_h/t+2Q_*|\{x\in\Omega:V(x)\le t\}|^{1/2},\qquad t>0.
\end{equation}
For $B_h\le1$,
\begin{equation}\label{eq:Q-clipped-fixed-source-rate}
\begin{aligned}
\|\widehat Q_h^\Omega-Q_\gamma\|_{L^2(\Omega)}\le C B_h^{1/(1+4s)},\qquad
\varepsilon_{Q,h}\le C B_h^{\theta_Q/(1+4s)},\qquad\theta_Q=\min\{1,4s/d\}.
\end{aligned}
\end{equation}
The constant $C$ is independent of $h,\delta$ and the regularization parameters, but may depend on the fixed sources, $Q_*$, the ellipticity bounds and the domains. In particular, $R_h+d_{v,h}+A_h+\sqrt{\alpha_Q}\to0$ implies strong $L^2(\Omega)$ convergence of $\widehat Q_h^\Omega$ without a condition involving $c_{Q,h}$.
\end{proposition}
\begin{proof}
Write $Q=Q_\gamma$, $\widehat Q_h=\widehat Q_h^\Omega$ and $\|\cdot\|_2=\|\cdot\|_{L^2(\Omega)}$. Minimality in \eqref{eq:Q-reconstruction} with competitor $Q_{I,h}$ gives
\[
\left(J^{-1}\sum_j\|(Q_h^\Omega-Q_{I,h})v_{j,h}\|_2^2\right)^{1/2}\le2R_h+\sqrt{\alpha_Q}\|Q\|_2.
\]
Since $|Q_{I,h}|\le\|Q\|_\infty<Q_*$, clipping contracts the pointwise distance to $Q_{I,h}$, so the inequality holds with $Q_h^\Omega$ replaced by $\widehat Q_h$. The bound $|\widehat Q_h-Q_{I,h}|\le2Q_*$ then yields
\[
\|(\widehat Q_h-Q)V\|_2
=\left(J^{-1}\sum_j\|(\widehat Q_h-Q)v_j\|_2^2\right)^{1/2}\le B_h.
\]
Splitting $\Omega$ into $\{V>t\}$ and $\{V\le t\}$ proves \eqref{eq:Q-clipped-weighted-levelset}.

For a fixed nonnegative nonzero source $F$, the conductivity datum $a_\gamma F$ is nonnegative. Testing the conductivity equation with the negative part gives $u\ge0$, hence $v=\sqrt\gamma\,u\ge0$. Consequently, $w=v-F\in\widetilde H^s(\Omega)$ is nonnegative and satisfies
\[
(( -\Delta)^s+Q_*)w=h_F+(Q_*-Q)v\ge c_F>0\quad\text{in }\Omega,\qquad
h_F(x)=C_{d,s}\int_{W_a}\frac{F(y)}{|x-y|^{d+2s}}\,dy.
\]
The uniform lower bound follows from the boundedness of $\Omega$ and $W_a$ and from $\int_{W_a}F>0$. Choose $\phi\in C_c^\infty(B_1)$ with $\phi\ge0$ and $\phi\ge1$ on $B_{1/2}$. For $B_r(x)\Subset\Omega$, the function $b_{x,r}(y)=r^{2s}\phi((y-x)/r)$ satisfies
\[
\sup_{\mathbb R^d}(( -\Delta)^s+Q_*)b_{x,r}\le\|(-\Delta)^s\phi\|_\infty+Q_*\operatorname{diam}(\Omega)^{2s}\|\phi\|_\infty=:M.
\]
Since \(b_{x,r}\ge r^{2s}\) on \(B_{r/2}(x)\), weak comparison gives \(v=w\ge(c_F/M)r^{2s}\) almost everywhere on that ball. Choose \(r=\operatorname{dist}(x,\partial\Omega)/2\); the balls \(B_{\operatorname{dist}(x,\partial\Omega)/4}(x)\) admit a countable subcover of \(\Omega\), and \(d(y):=\operatorname{dist}(y,\partial\Omega)\asymp \operatorname{dist}(x,\partial\Omega)\) within each ball, so \(v(y)\ge c\,d(y)^{2s}\) almost everywhere in \(\Omega\). Since \(V\ge J^{-1/2}v\), the Lipschitz boundary-collar estimate yields \(|\{V\le t\}|\le Ct^{1/(2s)}\) for small \(t\). Choosing $t=B_h^{4s/(1+4s)}$ in \eqref{eq:Q-clipped-weighted-levelset} proves the $L^2$-bound in \eqref{eq:Q-clipped-fixed-source-rate}; let $t\downarrow0$ when $B_h=0$. Enlarge $C$ to cover $B_h\le1$. Boundedness of $\widehat Q_h$ and $Q$ in $L^\infty(\Omega)$ gives the $L^{d/(2s)}$-bound by the finite measure of $\Omega$ and interpolation.
\end{proof}
If $F_{j,h}\to F_j$ in $W^s(\mathbb R^d)$ for every $j$, forward stability gives
\[
\Delta_h:=\left(J^{-1}\sum_j\|v_{j,h}^e-v_j\|_{L^2(\Omega)}^2\right)^{1/2}\to0,\qquad d_{v,h}\le\mathcal E_{v,h}+\Delta_h.
\]
Sobolev embedding, \eqref{eq:decom_rh} and Lemma~\ref{lem:regularized-weak-laplacian} give
\[
\begin{aligned}
A_h&\le C\rho_{Q,h}^{(d/s)}\left(J^{-1}\sum_j\|v_j\|_{W^s(\mathbb R^d)}^2\right)^{1/2},\\
R_h&\le\rho_{w,h}(\alpha_w)+C\alpha_w^{-1/2}\mathcal E_{v,h}+\|Q_\gamma\|_\infty\mathcal E_{v,h}+CM_{v,h}\rho_{Q,h}^{(d/s)}.
\end{aligned}
\]
Thus strong convergence of the sources, $\alpha_Q\to0$, $\alpha_w\to0$, $\rho_{Q,h}^{(d/s)}\to0$, and $\rho_{w,h}(\alpha_w)+\alpha_w^{-1/2}\mathcal E_{v,h}\to0$ imply $B_h\to0$.

Choose nonzero exterior sources $\{F_{r,h}^c\}_{r=1}^{J_c}\subset V_h^c$. For each source, find $z_{r,h}^c\in V_h^\Omega$ such that
\begin{equation}\label{eq:tail-forward-state}
 A_R(F_{r,h}^c+z_{r,h}^c,\phi_h) + \int_\Omega \widehat Q_h^\Omega
 (F_{r,h}^c+z_{r,h}^c)\phi_h\,dx =0 \qquad\forall\phi_h\in V_h^\Omega,
\end{equation}
and set $v_{r,h}^c:=F_{r,h}^c+z_{r,h}^c$. For each $r=1,\ldots,J_c$, define $T_{r,h}^\delta\in(V_h^c)^*$ by
\begin{equation}\label{eq:tail-data-functional}
 T_{r,h}^\delta(G_h) := \mathcal D_h^\delta\!\left( I_h^c(a_h^\delta F_{r,h}^c),
 I_h^c(a_h^\delta G_h)\right) - A_R(v_{r,h}^c,G_h), \qquad G_h\in V_h^c.
\end{equation}

Let $v_{r,h}^e$ denote the exact Schr\"odinger state for $Q_\gamma$ with exterior value $F_{r,h}^c$. In the exact-data case, \eqref{eq:overlap-identity} gives
\begin{equation}\label{eq:tail-exact-moment}
\begin{aligned}
 \big\langle \Lambda_\gamma(a_\gamma F_{r,h}^c),a_\gamma G_h
 \big\rangle - A_R(v_{r,h}^e,G_h) = (q_\gamma F_{r,h}^c,G_h)_{L^2(W_c)}
 = \bigl(\Pi_h^c(q_\gamma F_{r,h}^c),G_h \bigr)_{L^2(W_c)}.
\end{aligned}
\end{equation}
The last equality uses the projection $\Pi_h^c$ defined in \eqref{eq:Pi-h-c}. It shows that the discrete data determine the $V_h^c$-moments of $q_\gamma F_{r,h}^c$. Therefore, for a trial coefficient $p_h\in P_h^c$, define the moments by
\begin{equation}\label{eq:tail-predicted-functional}
 T_{r,h}(p_h)(G_h) := \bigl( \Pi_h^c(p_hF_{r,h}^c),G_h\bigr)_{L^2(W_c)} =
 (p_hF_{r,h}^c,G_h)_{L^2(W_c)}, \qquad G_h\in V_h^c.
\end{equation}
We use the discrete dual norm from \eqref{eq:discrete-negative-norm}, with $V_h^b$ replaced by $V_h^c$, and denote it by $\|\cdot\|_{-s,c,h}$. 

For $\alpha_t>0$, define
\begin{equation*} 
J_{t,h}^\delta(p_h) := \frac1{2J_c}\sum_{r=1}^{J_c}\|T_{r,h}(p_h)-T_{r,h}^\delta\|_{-s,c,h}^2+\frac{\alpha_t}{2}\|p_h\|_{L^2(W_c)}^2, \qquad p_h\in P_h^c. 
\end{equation*}
The exterior Liouville tail is reconstructed by
\begin{equation}\label{eq:tail-reconstruction}
 q_h^c \in \operatorname*{argmin}_{p_h\in P_h^c} J_{t,h}^\delta(p_h).
\end{equation}

Equip $P_h^c$ with the $L^2(W_c)$-inner product and set
\[
\begin{gathered}
Y_h^c:=((V_h^c)^*)^{J_c},\qquad \|(L_r)_r\|_{Y_h^c}^2:=\frac1{J_c}\sum_{r=1}^{J_c}\|L_r\|_{-s,c,h}^2, \qquad
\mathbb K_hp_h:=(T_{r,h}(p_h))_r,\qquad \mathsf t_h^\delta:=(T_{r,h}^\delta)_r.
\end{gathered}
\]
Then $J_{t,h}^\delta(p_h)=\tfrac12\|\mathbb K_hp_h-\mathsf t_h^\delta\|_{Y_h^c}^2+\tfrac{\alpha_t}{2}\|p_h\|_{L^2(W_c)}^2$. Lemma~\ref{lem:abstract-shifted-tikhonov}, with trial space $P_h^c$, operator $\mathbb K_h$, data $\mathsf t_h^\delta$, center zero, and parameter $\alpha_t$, gives a unique minimizer satisfying
\begin{equation}\label{eq:tail-normal-equation}
\begin{aligned}
\frac1{J_c}\sum_{r=1}^{J_c}\bigl(T_{r,h}(q_h^c),T_{r,h}(\psi_h)\bigr)_{-s,c,h}+\alpha_t(q_h^c,\psi_h)_{L^2(W_c)}=\frac1{J_c}\sum_{r=1}^{J_c}\bigl(T_{r,h}^\delta,T_{r,h}(\psi_h)\bigr)_{-s,c,h},\qquad \psi_h\in P_h^c.
\end{aligned}
\end{equation}
Using the basis \(\{\chi_{\ell,h}^c\}_{\ell=1}^{N_c}\) of \(V_h^c\) and the coefficient basis \(\psi_{m,h}^c:=\mathbf1_{K_m}\), \(m=1,\ldots,N_t\), define
\begin{align} 
(S_c)_{\ell k} &:= (\chi_{k,h}^c,\chi_{\ell,h}^c)_{W^s(\mathbb R^d)}, \qquad &(B_r)_{\ell m} &:= (\psi_{m,h}^cF_{r,h}^c,\chi_{\ell,h}^c)_{L^2(W_c)}, \label{eq:tail-multiplication-matrix} \\ (\mathbf t_r^\delta)_\ell &:= T_{r,h}^\delta(\chi_{\ell,h}^c), \qquad &(M_t)_{mn} &:= (\psi_{n,h}^c,\psi_{m,h}^c)_{L^2(W_c)}, \notag 
\end{align}
where \(r=1,\dots,J_c\), \(k = 1,\dots,N_c\) and \(n=1,\ldots,N_t\). The matrices $S_c$ and $M_t$ are symmetric positive definite. Moreover, \( M_t=\operatorname{diag}(|K_1|,\ldots,|K_{N_t}|)\). If \(
 p_h=\sum_{m=1}^{N_t}p_m\psi_{m,h}^c\) and \(\mathbf p=(p_1,\ldots,p_{N_t})^\top\),
then
\begin{equation*} 
\|T_{r,h}(p_h)-T_{r,h}^\delta\|_{-s,c,h}^2 = (B_r\mathbf p-\mathbf t_r^\delta)^\top S_c^{-1}(B_r\mathbf p-\mathbf t_r^\delta).
\end{equation*}
Indeed, the coefficient vector of the $W^s$-Riesz representative of
the functional $T_{r,h}(p_h)-T_{r,h}^\delta$ is the solution of
\[
 S_c\mathbf y_r=B_r\mathbf p-\mathbf t_r^\delta.
\]
Write \( q_h^c=\sum_{m=1}^{N_t}q_m^c\psi_{m,h}^c\) and \(\mathbf q^c=(q_1^c,\ldots,q_{N_t}^c)^\top\). Then \eqref{eq:tail-normal-equation} is equivalent to
\begin{equation}\label{eq:tail-matrix-system}
 G_t\mathbf q^c=\mathbf b_t,
\end{equation}
where \( G_t := \frac1{J_c}\sum_{r=1}^{J_c}B_r^\top S_c^{-1}B_r
 +\alpha_tM_t\) and  \(\mathbf b_t := \frac1{J_c}\sum_{r=1}^{J_c} B_r^\top S_c^{-1}\mathbf t_r^\delta\). For every nonzero $\mathbf p\in\mathbb R^{N_t}$, we have
\[
 \mathbf p^\top G_t\mathbf p = \frac1{J_c}\sum_{r=1}^{J_c}
 (B_r\mathbf p)^\top S_c^{-1}(B_r\mathbf p)
 +\alpha_t\mathbf p^\top M_t\mathbf p
 >0.
\]
Thus $G_t$ is symmetric positive definite. 

Prior to the error analysis, we distinguish the computed reduced data from their exact counterparts. For each source $F_{r,h}^c$, define $T_{r,h}^e\in(V_h^c)^*$ by
\begin{equation}\label{eq:tail-exact-data}
 T_{r,h}^e(G_h) := \bigl\langle \Lambda_\gamma(a_\gamma F_{r,h}^c),a_\gamma G_h
 \bigr\rangle  - A_R(v_{r,h}^e,G_h), \qquad G_h\in V_h^c .
\end{equation}
By the exact overlapping identity \eqref{eq:tail-exact-moment}, we get
\begin{equation}\label{eq:tail-exact-data-moment}
 T_{r,h}^e(G_h) = (q_\gamma F_{r,h}^c,G_h)_{L^2(W_c)} \qquad\forall G_h\in V_h^c.
\end{equation}
Assumption~\ref{ass:liouville-regularity} gives $q_\gamma|_{W_c}\in L^\infty(W_c)$, so the right-hand side is well-defined and $T_{r,h}^e$ belongs to $(V_h^c)^*$. Using the same regularization parameter $\alpha_t$, define the
exact-data functional
\begin{equation*} 
J_{t,h}^e(p_h) := \frac1{2J_c}\sum_{r=1}^{J_c}\|T_{r,h}(p_h)-T_{r,h}^e\|_{-s,c,h}^2+\frac{\alpha_t}{2}\|p_h\|_{L^2(W_c)}^2, \qquad p_h\in P_h^c, \end{equation*}
and let
\begin{equation*} 
q_h^{c,e} := \operatorname*{argmin}_{p_h\in P_h^c}J_{t,h}^e(p_h). 
\end{equation*}
Write $\mathsf t_h^e:=(T_{r,h}^e)_r$. Lemma~\ref{lem:abstract-shifted-tikhonov} also gives a unique exact-data minimizer $q_h^{c,e}$; dependence on $\alpha_t$ is suppressed. Its a priori and data-stability estimates \eqref{eq:abstract-tikhonov-apriori} and \eqref{eq:abstract-tikhonov-data-stability} give
\begin{equation*} \sqrt{\alpha_t}\|q_h^c\|_{L^2(W_c)}\le\left(\frac1{J_c}\sum_{r=1}^{J_c}\|T_{r,h}^\delta\|_{-s,c,h}^2\right)^{1/2},
\end{equation*}
and
\begin{equation*} 
\|q_h^c-q_h^{c,e}\|_{L^2(W_c)}\le\frac1{2\sqrt{\alpha_t}}\left(\frac1{J_c}\sum_{r=1}^{J_c}\|T_{r,h}^\delta-T_{r,h}^e\|_{-s,c,h}^2\right)^{1/2}. \end{equation*}

We first control the error introduced by using the reconstructed
interior potential in the Schr\"odinger forward problems. For
$Q\in L^{\frac{d}{2s}}(\Omega)$, write
\[
 \mathsf A_Q(z,\phi) := A_R(z,\phi)+\int_\Omega Qz\phi\,dx,
 \qquad z,\phi\in\widetilde H^s(\Omega).
\]
 Since $\frac{2s}{d}+\frac{2}{p_s^*}=1$, H\"older's inequality and then fractional Sobolev embedding give
\begin{equation}\label{eq:tail-potential-form-bound}
 \left|\int_\Omega Qz\phi\,dx\right|
 \le C\|Q\|_{L^{\frac{d}{2s}}(\Omega)}
 \|z\|_{L^{p_s^*}(\mathbb R^d)}
 \|\phi\|_{L^{p_s^*}(\mathbb R^d)}\le 
 C\|Q\|_{L^{\frac{d}{2s}}(\Omega)}
 \|z\|_{W^s(\mathbb R^d)}
 \|\phi\|_{W^s(\mathbb R^d)}.
\end{equation}
Also note that, by the fitted-mesh construction, the family $\{V_h^\Omega\}_{h>0}$ has the approximation property
\begin{equation}\label{eq:VOmega-approximation}
 \lim_{h\to0} \inf_{\zeta_h\in V_h^\Omega} \|z-\zeta_h\|_{W^s(\mathbb R^d)}
 =0 \qquad \forall z\in\widetilde H^s(\Omega).
\end{equation}
Indeed, this follows from the density of $C_c^\infty(\Omega)$ in $\widetilde H^s(\Omega)$, fitted-mesh interpolation of compactly supported smooth functions, and the continuous embedding $H^1(\mathbb R^d)\hookrightarrow H^s(\mathbb R^d)$ for $0<s<1$; see \cite[Lem.~3.6]{DwivediRailoRuppFCP}.

\begin{proposition}[Consistency of the tail forward states]
\label{prop:tail-forward-consistency}
Under Assumption~\ref{ass:liouville-regularity} and \eqref{eq:VOmega-approximation}, suppose that $\varepsilon_{Q,h}\to0$, as supplied by Theorem~\ref{thm:Q-whole-domain-error-regularized} or Proposition~\ref{prop:Q-clipped-fixed-sources} under their respective convergence conditions. Then there exist $h_0,\kappa_c>0$, independent of $h$, such that
\begin{equation}\label{eq:tail-forward-infsup}
 \inf_{0\ne z_h\in V_h^\Omega} \sup_{0\ne\phi_h\in V_h^\Omega}
 \frac{|\mathsf A_{\widehat Q_h^\Omega}(z_h,\phi_h)|}
 {\|z_h\|_{W^s(\mathbb R^d)}\|\phi_h\|_{W^s(\mathbb R^d)}}
 \ge \kappa_c, \qquad 0<h<h_0.
\end{equation}
In particular, \eqref{eq:tail-forward-state} has a unique solution for \(0<h<h_0\). Let $v_{r,h}^e=F_{r,h}^c+z_{r,h}^e$ be the exact state for $Q_\gamma$ and define
\begin{align*} 
E_{v,h}^c &:= \left(\frac1{J_c}\sum_{r=1}^{J_c}\|v_{r,h}^c-v_{r,h}^e\|_{W^s(\mathbb R^d)}^2\right)^{1/2}, \\ \rho_{z,h}^c &:= \left(\frac1{J_c}\sum_{r=1}^{J_c}\inf_{\zeta_h\in V_h^\Omega}\|z_{r,h}^e-\zeta_h\|_{W^s(\mathbb R^d)}^2\right)^{1/2}, \qquad M_{z,h}^c := \left(\frac1{J_c}\sum_{r=1}^{J_c}\|z_{r,h}^e\|_{W^s(\mathbb R^d)}^2\right)^{1/2}. 
\end{align*}
Then
\begin{equation}\label{eq:tail-state-consistency}
 E_{v,h}^c
 \le C\bigl(\rho_{z,h}^c+\varepsilon_{Q,h}M_{z,h}^c\bigr),
\end{equation}
where $C$ is independent of $h$, $J_c$, and the source index.
\end{proposition}

\begin{proof}
We first prove \eqref{eq:tail-forward-infsup}. Suppose that it fails.
Then there are $h_n\to0$ and $z_n\in V_{h_n}^\Omega$ such that
$\|z_n\|_{W^s(\mathbb R^d)}=1$ and
\[
 \sup_{0\ne\phi_h\in V_{h_n}^\Omega}
 \frac{|\mathsf A_{\widehat Q_{h_n}^\Omega}(z_n,\phi_h)|}
 {\|\phi_h\|_{W^s(\mathbb R^d)}}
 \longrightarrow0.
\]
After passing to a subsequence, $z_n\rightharpoonup z$ in
$\widetilde H^s(\Omega)$ and $z_n\to z$ in $L^2(\Omega)$. Given
$\phi\in\widetilde H^s(\Omega)$, choose $\phi_n\in V_{h_n}^\Omega$ with $\phi_n\to\phi$ in $W^s(\mathbb R^d)$ by \eqref{eq:VOmega-approximation}. Using \eqref{eq:tail-potential-form-bound} and $\varepsilon_{Q,h_n}\to0$, we may pass to the limit and obtain $\mathsf A_{Q_\gamma}(z,\phi)=0$. The exact homogeneous Schr\"odinger problem has only the zero solution by the Liouville reduction and the well-posed conductivity problem. Hence $z=0$. It follows that $z_n\to0$ in $L^2(\Omega)$. Choosing $\phi_h=z_n$ in the residual estimate gives $\mathsf A_{\widehat Q_{h_n}^\Omega}(z_n,z_n)\to0$. Since
$\|\widehat Q_{h_n}^\Omega\|_{L^\infty(\Omega)}\le Q_*$, we obtain $A_R(z_n,z_n)\to0$. Together with \eqref{eq:Ws-norm-convention} and $z_n\to0$ in $L^2(\Omega)$, this contradicts $\|z_n\|_{W^s(\mathbb R^d)}=1$. Thus
\eqref{eq:tail-forward-infsup} holds.

Fix $r\in\{1,\ldots,J_c\}$ and let $\zeta_h\in V_h^\Omega$. Applying \eqref{eq:tail-forward-infsup} to $z_{r,h}^c-\zeta_h$ gives
\[
 \kappa_c\|z_{r,h}^c-\zeta_h\|_{W^s(\R^{d})}
 \le \sup_{0\ne\phi_h\in V_h^\Omega}
 \frac{ |\mathsf A_{\widehat Q_h^\Omega}(z_{r,h}^c-\zeta_h,\phi_h)|}{\|\phi_h\|_{W^s(\R^{d})}}.
\]
Since $F_{r,h}^c=0$ in $\Omega$, the exact and discrete corrections
satisfy
\[
 \mathsf A_{Q_\gamma}(z_{r,h}^e,\phi_h) = \mathsf A_{\widehat Q_h^\Omega}(z_{r,h}^c,\phi_h) = -A_R(F_{r,h}^c,\phi_h).
\]
Consequently,
\[
 \mathsf A_{\widehat Q_h^\Omega} (z_{r,h}^c-\zeta_h,\phi_h)={} \mathsf A_{\widehat Q_h^\Omega}
 (z_{r,h}^e-\zeta_h,\phi_h)+ \int_\Omega (Q_\gamma-\widehat Q_h^\Omega)
 z_{r,h}^e\phi_h\,dx .
\]
The uniform continuity of $\mathsf A_{\widehat Q_h^\Omega}$ and
\eqref{eq:tail-potential-form-bound} imply
\[
 \|z_{r,h}^c-\zeta_h\|_{W^s(\R^{d})} \le C\left( \|z_{r,h}^e-\zeta_h\|_{W^s(\R^{d})}+ \varepsilon_{Q,h}\|z_{r,h}^e\|_{W^s(\R^{d})}
 \right).
\]
Using the triangle inequality and minimizing over
$\zeta_h\in V_h^\Omega$, we obtain
\[
 \|z_{r,h}^c-z_{r,h}^e\|_{W^s(\R^{d})}
 \le
 C\left(
 \inf_{\zeta_h\in V_h^\Omega}
 \|z_{r,h}^e-\zeta_h\|_{W^s(\R^{d})}
 +
 \varepsilon_{Q,h}\|z_{r,h}^e\|_{W^s(\R^{d})}
 \right).
\]
Since the exact and computed states have the same exterior value,
$v_{r,h}^c-v_{r,h}^e=z_{r,h}^c-z_{r,h}^e$. Taking the
root-mean-square over $r=1,\ldots,J_c$ and using Minkowski's
inequality proves \eqref{eq:tail-state-consistency}.
\end{proof}

For later use, set
\begin{equation*} 
M_{F,h}^c := \left(\frac1{J_c}\sum_{r=1}^{J_c}\|F_{r,h}^c\|_{W^s(\mathbb R^d)}^2\right)^{1/2}. 
\end{equation*}
The well-posedness of the exact Schr\"odinger forward problem gives \(
 \|z_{r,h}^e\|_{W^s(\mathbb R^d)} \le C\|F_{r,h}^c\|_{W^s(\mathbb R^d)}\) for \(r=1,\ldots,J_c\).
Indeed, the exact correction satisfies
\[
\mathsf A_{Q_\gamma}(z_{r,h}^e,\phi) = -A_R(F_{r,h}^c,\phi)
 \qquad
 \forall\phi\in\widetilde H^s(\Omega),
\]
and the conclusion follows from the bounded inverse of the exact Schr\"odinger operator and the continuity of $A_R$. Consequently, \(M_{z,h}^c\le C M_{F,h}^c\). In the convergence statements below we use the source consistency quantities \(M_{F,h}^c\) and \(\rho_{z,h}^c\). For fixed smooth tail sources, one obtains the usual consistency condition as follows: If $J_c$ is fixed and $F_{r,h}^c\to F_r^c$ in $W^s(\mathbb R^d)$ for every $r$, continuity of the exact forward map and \eqref{eq:VOmega-approximation} give $M_{F,h}^c=\mathcal O(1)$ and $\rho_{z,h}^c\to0$.

We now estimate the total perturbation in the reduced tail data. Set
\[
 M_a^c := \sup_{0\ne G\in\widetilde H^s(W_c)} \frac{\|a_\gamma G\|_{W^s(\mathbb R^d)}} {\|G\|_{W^s(\mathbb R^d)}},
 \qquad \varepsilon_c^a:=\varepsilon_{W_c}^a(h,\delta),
\]
where \(\varepsilon_{W_c}^a\) is defined in \eqref{eq:eps-a-def}. By Assumption~\ref{ass:liouville-regularity}(iii), $M_a^c<\infty$. Define
\begin{equation*} 
\eta_{cc}(h,\delta) := \delta(M_a^c+\varepsilon_c^a)^2+C_\Lambda\varepsilon_c^a(2M_a^c+\varepsilon_c^a), 
\end{equation*}
where $C_\Lambda$ is the constant in \eqref{eq:DN-Ws-continuity}.

\begin{lemma}[Consistency of the reduced tail data]
\label{lem:tail-data-consistency}
For every $r=1,\ldots,J_c$,
\begin{equation}\label{eq:tail-data-consistency-sourcewise}
 \|T_{r,h}^\delta-T_{r,h}^e\|_{-s,c,h}
 \le
 \eta_{cc}(h,\delta)
 \|F_{r,h}^c\|_{W^s(\mathbb R^d)}
 +C\|v_{r,h}^c-v_{r,h}^e\|_{W^s(\mathbb R^d)}.
\end{equation}
Consequently, with
\begin{equation}\label{eq:eta-tail-data}
 \eta_{t,h}
 :=
 \left(
 \frac1{J_c}\sum_{r=1}^{J_c}
 \|T_{r,h}^\delta-T_{r,h}^e\|_{-s,c,h}^2
 \right)^{1/2},
\end{equation}
there holds
\begin{equation}\label{eq:eta-tail-data-bound}
 \eta_{t,h}
 \le
 \eta_{cc}(h,\delta)M_{F,h}^c+CE_{v,h}^c
 \le
 C\left[
 \eta_{cc}(h,\delta)M_{F,h}^c
 +\rho_{z,h}^c
 +\varepsilon_{Q,h}M_{z,h}^c
 \right].
\end{equation}
\end{lemma}

\begin{proof}
Fix $G_h\in V_h^c$. Repeating the proof of
Lemma~\ref{lem:reduced-separated-data-consistency}, now with both
transformed factors supported in $W_c$, gives
\[
\begin{aligned}
 &\left| \mathcal D_h^\delta\!\left( I_h^c(a_h^\delta F_{r,h}^c), I_h^c(a_h^\delta G_h) \right) - \big\langle
 \Lambda_\gamma(a_\gamma F_{r,h}^c),a_\gamma G_h \big\rangle
 \right| \le \eta_{cc}(h,\delta) \|F_{r,h}^c\|_{W^s(\mathbb R^d)} \|G_h\|_{W^s(\mathbb R^d)}.
\end{aligned}
\]
Subtracting \eqref{eq:tail-data-functional} and
\eqref{eq:tail-exact-data}, and using the continuity of $A_R$, yields
\[
 |(T_{r,h}^\delta-T_{r,h}^e)(G_h)| \le \left[ \eta_{cc}(h,\delta)\|F_{r,h}^c\|_{W^s(\R^{d})} +C\|v_{r,h}^c-v_{r,h}^e\|_{W^s(\R^{d})} \right]\|G_h\|_{W^s(\R^{d})}.
\]
Taking the supremum over $G_h\ne0$ proves \eqref{eq:tail-data-consistency-sourcewise}. The first estimate in \eqref{eq:eta-tail-data-bound} follows from the root-mean-square triangle inequality, and the second follows from
Proposition~\ref{prop:tail-forward-consistency}.
\end{proof}

Let \(q_\gamma^c:=q_\gamma|_{W_c}\), let \(\Pi_{0,h}^c:L^2(W_c)\to P_h^c\) be the elementwise \(L^2(W_c)\)-projection, and set
\begin{equation*}
 q_{I,h}^c:=\Pi_{0,h}^c q_\gamma^c,
 \qquad
 q_{I,h}^c|_K=\frac1{|K|}\int_Kq_\gamma^c(x)\,dx.
\end{equation*}
For $p\in\{2,\frac{d}{2s}\}$, define \( \rho_{q,h}^{(p)} :=
\|q_{I,h}^c-q_\gamma^c\|_{L^p(W_c)}\). The cell-average projection is contractive in $L^p(W_c)$ for every $1\le p\le\infty$. Since $q_\gamma^c\in L^\infty(W_c)$ and $\frac{d}{2s}<\infty$, it follows that $\rho_{q,h}^{(2)}+\rho_{q,h}^{(\frac{d}{2s})}\to0$.

\begin{lemma}[Consistency of the projected tail]
\label{lem:tail-projection-consistency}
For every $r=1,\ldots,J_c$, there holds
\begin{equation*} 
\|T_{r,h}(q_{I,h}^c)-T_{r,h}^e\|_{-s,c,h} \le C\rho_{q,h}^{(\frac{d}{2s})}\|F_{r,h}^c\|_{W^s(\mathbb R^d)}. 
\end{equation*}
\end{lemma}

\begin{proof}
For $G_h\in V_h^c$, the identities
\eqref{eq:tail-predicted-functional} and
\eqref{eq:tail-exact-data-moment} give
\[
 (T_{r,h}(q_{I,h}^c)-T_{r,h}^e)(G_h)
 =
 \int_{W_c}(q_{I,h}^c-q_\gamma^c)F_{r,h}^cG_h\,dx.
\]
Since $\frac{2s}{d}+\frac{2}{p_s^*}=1$, H\"older's inequality and
$H^s(\mathbb R^d)\hookrightarrow L^{p_s^*}(\mathbb R^d)$ imply
\[
 |(T_{r,h}(q_{I,h}^c)-T_{r,h}^e)(G_h)|
 \le
 C\rho_{q,h}^{(\frac{d}{2s})}
 \|F_{r,h}^c\|_{W^s(\mathbb R^d)}
 \|G_h\|_{W^s(\mathbb R^d)}.
\]
Taking the supremum over $G_h\ne0$ proves the claim.
\end{proof}

Lemma~\ref{lem:tail-projection-consistency} gives
\begin{equation*} 
\|\mathsf t_h^e-\mathbb K_hq_{I,h}^c\|_{Y_h^c} \le CM_{F,h}^c\rho_{q,h}^{(\frac{d}{2s})}. 
\end{equation*}
Together with \eqref{eq:eta-tail-data}, this yields
\begin{equation}\label{eq:tail-total-perturbation}
\|\mathsf t_h^\delta-\mathbb K_hq_{I,h}^c\|_{Y_h^c} \le \xi_{t,h}, \qquad \xi_{t,h} := \eta_{t,h}+CM_{F,h}^c\rho_{q,h}^{(\frac{d}{2s})}.
\end{equation}

The term $\alpha_t\|p_h\|_{L^2(W_c)}^2$ makes \eqref{eq:tail-matrix-system} symmetric positive definite for every $\alpha_t>0$. To quantify the unregularized moment map on $P_h^c$, define
\begin{equation}\label{eq:tail-smallest-singular-value}
    \sigma_{t,h} := \inf_{0\ne p_h\in P_h^c} \frac{\|\mathbb K_h p_h\|_{Y_h^c}}{\|p_h\|_{L^2(W_c)}} .
\end{equation}
Thus the chosen tail source family \(\mathcal F_h^c:=\{F_{r,h}^c\}_{r=1}^{J_c}\subset V_h^c\) is identifying on \(P_h^c\) if and only if \(\sigma_{t,h}>0\), equivalently
\(\mathbb K_h:P_h^c\to Y_h^c\) is injective.

In the matrix notation of \eqref{eq:tail-multiplication-matrix}, this singular
value \(\sigma_{t,h}\) is the smallest generalized singular value of the stacked tail moment matrix. More precisely, if \(M_t\) is the \(L^2(W_c)\)-mass matrix on \(P_h^c\),
then
\begin{equation}\label{eq:tail-singular-value-matrix}
    \sigma_{t,h}^2 = \lambda_{\min} \left[M_t^{-1/2}\left(\frac1{J_c}
    \sum_{r=1}^{J_c}B_r^\top S_c^{-1}B_r\right)M_t^{-1/2}
    \right].
\end{equation}
Here \(\lambda_{\min}(H)\) denotes the smallest eigenvalue of the symmetric
matrix \(H\).  The following elementary construction gives an explicit identifying family.
\begin{lemma}[Nonvanishing bubble moments and tail identifiability]
\label{lem:bubble-tail-identifiability}
Let \(\mathcal T_h^c=\{K_1,\ldots,K_{N_t}\}\), and let \(P_h^c\) be the elementwise constant space on \(\mathcal T_h^c\). Assume that the tail source/test space \(V_h^c\) contains cell bubbles \(b_{m,h}\in H_0^1(K_m)\) with \(b_{m,h}>0\) in the interior of \(K_m\), \(m=1,\ldots,N_t \). Suppose that the chosen tail source family \( \mathcal F_h^c=\{F_{r,h}^c\}_{r=1}^{J_c}\subset V_h^c\) contains at least one source, say \(F_{1,h}^c\), such that
\begin{equation}\label{eq:positive-tail-source-condition}
    \int_{K_m} F_{1,h}^c\, b_{m,h}\,dx \neq 0,
    \qquad m=1,\ldots,N_t .
\end{equation}
Then \(\mathbb K_h:P_h^c\to Y_h^c\) is injective. Equivalently,
\(\sigma_{t,h}>0\).
\end{lemma}

\begin{proof}
Let \(p_h\in P_h^c\) and suppose that \(\mathbb K_hp_h=0\). Since \(p_h\) is
elementwise constant, write \(p_h|_{K_m}=p_m\), \(m=1,\ldots,N_t \). Because \(\mathbb K_hp_h=0\), the first component \(T_{1,h}(p_h)\) is the zero functional on \(V_h^c\). Testing this functional with \(G_h=b_{m,h}\in V_h^c\) and using \eqref{eq:tail-predicted-functional}, we get
\[
    0 =  T_{1,h}(p_h)(b_{m,h}) = \int_{W_c}p_hF_{1,h}^c b_{m,h}\,dx
    =  p_m\int_{K_m}F_{1,h}^c b_{m,h}\,dx .
\]
By \eqref{eq:positive-tail-source-condition}, the last integral is nonzero.
Hence \(p_m=0\). Since this holds for every \(m=1,\ldots,N_t\), we obtain
\(p_h=0\). Thus \(\mathbb K_h\) is injective.
\end{proof}
In particular, condition \eqref{eq:positive-tail-source-condition} is satisfied
by the explicit choice \(F_{1,h}^c := \frac{b_h^+}{\|b_h^+\|_{W^s(\mathbb R^d)}} \) where \( b_h^+ := \sum_{m=1}^{N_t} b_{m,h}\).
\begin{remark}[Computable source-selection criterion]
\label{rem:computable-source-selection}
Lemma~\ref{lem:bubble-tail-identifiability} gives a sufficient identifying family. The numerical experiments instead use finite sine-source families and check \(\operatorname{rank}(\mathbb K_h)=\dim P_h^c\), equivalently \(\sigma_{t,h}>0\), through \eqref{eq:tail-singular-value-matrix}. This verifies identifiability for each tested discretization.
\end{remark}

\begin{theorem}[Error estimate for the exterior Liouville tail]
\label{thm:tail-reconstruction-error}
Assume that the chosen tail source family satisfies \(\sigma_{t,h}>0\), where \(\sigma_{t,h}\) is defined by \eqref{eq:tail-smallest-singular-value}. Then, for every \(\alpha_t>0\), the reconstruction \eqref{eq:tail-reconstruction} satisfies
\begin{equation}\label{eq:tail-reconstruction-error}
    \|q_h^c-q_\gamma^c\|_{L^2(W_c)}
    \le
    \rho_{q,h}^{(2)}
    +
    \frac{\xi_{t,h}}{2\sqrt{\alpha_t}}
    +
    \frac{\alpha_t}{\sigma_{t,h}^2+\alpha_t}
    \|q_{I,h}^c\|_{L^2(W_c)} .
\end{equation}
Consequently, \(q_h^c\to q_\gamma^c\) in \(L^2(W_c)\) if
\begin{equation}\label{eq:tail-convergence-parameter-rule}
    \rho_{q,h}^{(2)}\to0,\qquad
    \alpha_t\to0,\qquad
    \frac{\xi_{t,h}^2}{\alpha_t}\to0,\qquad
    \frac{\alpha_t}{\sigma_{t,h}^2}\to0 .
\end{equation}
\end{theorem}

\begin{proof}
Apply \eqref{eq:abstract-tikhonov-comparator} with comparator $q_{I,h}^c$, center zero, and smallest singular value $\sigma_{t,h}$. Using \eqref{eq:tail-total-perturbation} yields
\[
\|q_h^c-q_{I,h}^c\|_{L^2(W_c)}\le\frac{\xi_{t,h}}{2\sqrt{\alpha_t}}+\frac{\alpha_t}{\sigma_{t,h}^2+\alpha_t}\|q_{I,h}^c\|_{L^2(W_c)}.
\]
The triangle inequality with $\|q_{I,h}^c-q_\gamma^c\|_{L^2(W_c)}=\rho_{q,h}^{(2)}$ proves \eqref{eq:tail-reconstruction-error}. Conditions \eqref{eq:tail-convergence-parameter-rule} and $\|q_{I,h}^c\|_{L^2(W_c)}\le\|q_\gamma^c\|_{L^2(W_c)}$ imply convergence.
\end{proof}

Combining \eqref{eq:eta-tail-data-bound} and
\eqref{eq:tail-total-perturbation} gives
\begin{equation*} 
\xi_{t,h} \le C\left[\eta_{cc}(h,\delta)M_{F,h}^c+\rho_{z,h}^c+\varepsilon_{Q,h}M_{z,h}^c+M_{F,h}^c\rho_{q,h}^{(\frac d{2s})}\right]. 
\end{equation*}
If \(q_\gamma^c\in W^{m,p}(W_c)\), \(0<m\le1\), and \(p\in\{2,d/(2s)\}\), then \(\rho_{q,h}^{(p)} \le Ch^m |q_\gamma^c|_{W^{m,p}(W_c)}\). The two parameter requirements \( \frac{\xi_{t,h}^2}{\alpha_t}\to0\) and  \(\frac{\alpha_t}{\sigma_{t,h}^2}\to0\) are compatible whenever \(\xi_{t,h}\to0\) and \(\frac{\xi_{t,h}}{\sigma_{t,h}}\to0\). For \(\xi_{t,h}>0\), one may choose \(\alpha_t = \xi_{t,h}\min\{\sigma_{t,h},1\}\). Then \(\alpha_t\to0\), and
\[
\frac{\xi_{t,h}^2}{\alpha_t} + \frac{\alpha_t}{\sigma_{t,h}^2}
\le 2\max\left\{ \xi_{t,h}, \frac{\xi_{t,h}}{\sigma_{t,h}}
\right\}
\longrightarrow0.
\]
The bubble criterion gives $\sigma_{t,h}>0$ for each fixed $h$. Consequently $\lVert p_h\rVert_{L^2(W_c)}\le\sigma_{t,h}^{-1}\lVert \mathbb K_hp_h\rVert_{Y_h^c}$ for every $p_h\in P_h^c$. The condition number additionally depends on the largest singular value of $\mathbb K_h$.

\begin{corollary}[Tail convergence for fixed limiting sources]\label{cor:tail-fixed-source-convergence}
Let $J_c$ be fixed, let $F_{r,h}^c\to F_r^c$ in $L^2(W_c)$, and assume
\[
\sum_{r=1}^{J_c}|F_r^c(x)|^2>0\quad\text{for almost every }x\in W_c,\qquad
\rho_{q,h}^{(2)}\to0,\quad\alpha_t\to0,\quad\xi_{t,h}^2/\alpha_t\to0.
\]
For the full fitted observation spaces $V_h^c$, the minimizers in \eqref{eq:tail-reconstruction} satisfy $q_h^c\to q_\gamma^c$ strongly in $L^2(W_c)$, without a coupling involving $\sigma_{t,h}$.
\end{corollary}
\begin{proof}
Minimality and \eqref{eq:tail-total-perturbation} give
\[
\|\mathbb K_hq_h^c-\mathsf t_h^\delta\|_{Y_h^c}^2+\alpha_t\|q_h^c\|_{L^2(W_c)}^2
\le\xi_{t,h}^2+\alpha_t\|q_{I,h}^c\|_{L^2(W_c)}^2.
\]
Since $q_{I,h}^c\to q_\gamma^c$ in $L^2(W_c)$ and $\|\mathsf t_h^\delta-\mathsf t_h^e\|_{Y_h^c}=\eta_{t,h}\le\xi_{t,h}$, the sequence $q_h^c$ is bounded, $\|\mathbb K_hq_h^c-\mathsf t_h^e\|_{Y_h^c}\to0$, and
\[
\limsup_{h\to0}\|q_h^c\|_{L^2(W_c)}^2\le\|q_\gamma^c\|_{L^2(W_c)}^2.
\]
Let $q_h^c\rightharpoonup\overline q$ along a subsequence. For $G\in C_c^\infty(W_c)$, nodal interpolation supplies $G_h\in V_{h,0}^c\subset V_h^c$ converging to $G$ in $W^s(\mathbb R^d)$ and $L^\infty(W_c)$. Hence $F_{r,h}^cG_h\to F_r^cG$ in $L^2(W_c)$, and the vanishing exact-data residual and \eqref{eq:tail-exact-data-moment} imply
\[
\int_{W_c}(\overline q-q_\gamma^c)F_r^cG\,dx=0,\qquad r=1,\ldots,J_c.
\]
The limiting source coverage gives $\overline q=q_\gamma^c$ almost everywhere. Weak lower semicontinuity and the norm bound yield convergence of the norms, hence strong convergence of the entire sequence.
\end{proof}
The limiting source coverage is an additional assumption; the finite-mesh bubble criterion in Lemma~\ref{lem:bubble-tail-identifiability} does not imply it. Estimate \eqref{eq:tail-reconstruction-error} retains the quantitative dependence on $\sigma_{t,h}$.

Finally, the recovered tail gives the second exterior Cauchy datum for \(m_\gamma=\sqrt{\gamma}-1\). Define
\begin{equation}\label{eq:h-c-definition}
 \tau_\gamma :=(-\Delta)^s m_\gamma|_{W_c} =-\sqrt{\gamma}\,q_\gamma^c \in L^2(W_c), \qquad \tau_h^c :=-\Pi_h^c\!\left(\sqrt{\Gamma_h^\delta}\,q_h^c\right) \in V_h^c. 
\end{equation}
Let \( \rho_{\tau,h} := \|\Pi_h^c\tau_\gamma-\tau_\gamma\|_{L^2(W_c)}\).
Since $\Gamma_h^\delta\in[\gamma_0/2,2\gamma_1]$ and the square-root
map is Lipschitz on this interval, the $L^2$-contractivity of
$\Pi_h^c$ gives
\begin{equation}\label{eq:tail-error}
 \|\tau_h^c-\tau_\gamma\|_{L^2(W_c)}
 \le
 \rho_{\tau,h}
 +C\|q_h^c-q_\gamma^c\|_{L^2(W_c)}
 +Ce_h^\delta\|q_\gamma^c\|_{L^2(W_c)}.
\end{equation}
Therefore $\tau_h^c\to\tau_\gamma$ in $L^2(W_c)$ whenever $q_h^c\to q_\gamma^c$ in $L^2(W_c)$, $e_h^\delta\to0$, and $\rho_{\tau,h}\to0$.

\subsection{Regularized Cauchy--Liouville continuation}
\label{subsec:regularized-cauchy}

The final continuation on \(\Omega_R\) uses \(\widehat Q_h^\Omega\), \(m_h^W=\sqrt{\Gamma_h^\delta}-1\), and \(\tau_h^c\), defined in \eqref{eq:tail-clipped-potential}, \eqref{eq:ah-mh}, and \eqref{eq:h-c-definition}, respectively. Their errors are controlled by \(\varepsilon_{Q,h}\), \(e_h^\delta:=\|\Gamma_h^\delta-\gamma\|_{L^\infty(W_0)}\), and \eqref{eq:tail-error}. Dependence on the noise and regularization parameters of the reconstructed inputs is suppressed in the notation.

By \eqref{eq:fse2}, the continuous continuation problem is to determine \(m\in\widetilde H^s(\Omega_R)\) such that
\begin{equation}
\label{eq:continuous-cauchy-liouville-problem}
\left\{
\begin{aligned}
(-\Delta)^s m+Q_\gamma(1+m)&=0 &&\text{in }\mathcal D'(\Omega),\\
m&=m_\gamma|_{W_c} &&\text{almost everywhere in }W_c,\\ (-\Delta)^s m&=\tau_\gamma
&&\text{in }\mathcal D'(W_c).
\end{aligned}
\right.
\end{equation}
Assumption~\ref{ass:liouville-regularity} guarantees that \(m_\gamma\) is a solution of \eqref{eq:continuous-cauchy-liouville-problem}. The Cauchy data in \eqref{eq:continuous-cauchy-liouville-problem} determine the continuation uniquely. Indeed, if \(m_1,m_2\in\widetilde H^s(\Omega_R)\) have the same Cauchy data on \(W_c\), then \(e:=m_1-m_2\in H^s(\mathbb R^d)\) satisfies
\[ e=0 \qquad \text{ and } \qquad (-\Delta)^s e=0 \quad\text{in }W_c
\]
in the distributional sense. Since $W_c$ is a nonempty open set, the unique continuation property implies $e\equiv0$ in $\mathbb R^d$; see \cite[Thm.~1.2]{Ghosh2020calderon}. Proposition~\ref{prop:discrete-cauchy-logarithmic} below quantifies this continuation using the discrete observations. We approximate \eqref{eq:continuous-cauchy-liouville-problem} by a finite-dimensional Tikhonov problem; see \cite[Ch.~4]{Engl1996Regularization}.

The nodal subspace \(V_{h,0}^c\subset V_h^c\) gives the approximation properties
\begin{equation}
\label{eq:cauchy-space-density}
\inf_{\eta_h\in V_h^c}
\|\eta-\eta_h\|_{W^s(\mathbb R^d)}
\longrightarrow0
\qquad
\forall\,\eta\in C_c^\infty(W_c),
\end{equation}
and
\begin{equation}
\label{eq:cauchy-projection-consistency}
\|\Pi_h^c w-w\|_{L^2(W_c)}
\longrightarrow0
\qquad
\forall\,w\in L^2(W_c).
\end{equation}
Indeed, let \(\eta\in C_c^\infty(W_c)\). Since \(\operatorname{supp}\eta\) has positive distance from \(\mathbb R^d\setminus W_c\), the global Lagrange interpolant \(I_h\eta\) is supported in \(W_c\) for all sufficiently small \(h\). Hence \(I_h\eta\in V_{h,0}^c\subset V_h^c \). Standard Lagrange interpolation gives \(I_h\eta\to \eta\) in \(H^1(\mathbb R^d)\), and since \(H^1(\mathbb R^d)\hookrightarrow H^s(\mathbb R^d)\) for \(0<s<1\), equivalently in the \(W^s(\mathbb R^d)\)-norm, this proves \eqref{eq:cauchy-space-density}; see \cite[Ch.~3--4]{brenner2008mathematical}. To prove \eqref{eq:cauchy-projection-consistency}, let \(w\in L^2(W_c)\) and choose \(\eta\in C_c^\infty(W_c)\) with \(\|w-\eta\|_{L^2(W_c)}\) arbitrarily small. By \eqref{eq:cauchy-space-density}, there exists \(\eta_h\in V_h^c\) with \(\|\eta-\eta_h\|_{L^2(W_c)}\to0\). Since \(\Pi_h^c\) is the \(L^2(W_c)\)-orthogonal projection onto \(V_h^c\),
\[
    \|\Pi_h^c w-w\|_{L^2(W_c)}
    \le
    \|w-\eta_h\|_{L^2(W_c)}
    \le
    \|w-\eta\|_{L^2(W_c)}
    +
    \|\eta-\eta_h\|_{L^2(W_c)} .
\]
Letting first \(h\to0\) and then using the density of
\(C_c^\infty(W_c)\) in \(L^2(W_c)\) proves \eqref{eq:cauchy-projection-consistency}.

Define the projection error of the exterior value datum by
\begin{equation*}
\rho_{m,h} := \|\Pi_h^c(m_\gamma|_{W_c})-m_\gamma|_{W_c}\|_{L^2(W_c)}. \end{equation*}
Since both \(\Gamma_h^\delta\) and \(\gamma\) take values in the compact interval \([\gamma_0/2,2\gamma_1]\), the square-root map is Lipschitz on this interval. Hence the \(L^2(W_c)\)-contractivity of \(\Pi_h^c\), followed by the embedding \(L^\infty(W_c)\hookrightarrow L^2(W_c)\), gives
\begin{equation*} \begin{aligned} 
\|\Pi_h^c(m_h^W|_{W_c})-m_\gamma|_{W_c}\|_{L^2(W_c)} \le \rho_{m,h}+\|\Pi_h^c(m_h^W-m_\gamma|_{W_c})\|_{L^2(W_c)} \le \rho_{m,h}+C|W_c|^{1/2}e_h^\delta \longrightarrow0.
\end{aligned} \end{equation*}

Since both equations involving \((-\Delta)^s\) are weak equations, we measure their discrete residuals in the corresponding dual norms. For \(L_h\in(V_h^\Omega)^*\) and \(T_h\in(V_h^c)^*\), define \(\|L_h\|_{-s,\Omega,h}\) and \(\| T_h\|_{-s,c,h}\) in analogy with \eqref{eq:discrete-negative-norm}.

For \(Q\in L^\infty(\Omega)\) and \(\mathsf m_h\in V_h^R\), define the interior Liouville residual functional
\begin{equation}
\label{eq:R-omega-h-functional}
\mathcal R_{\Omega,h}(\mathsf m_h;Q)\in(V_h^\Omega)^*,
\qquad
\big\langle
\mathcal R_{\Omega,h}(\mathsf m_h;Q),\phi_h
\big\rangle
:=
A_R(\mathsf m_h,\phi_h)
+
\int_\Omega Q(1+\mathsf m_h)\phi_h\,dx
\end{equation}
for every \(\phi_h\in V_h^\Omega\). This functional is well defined. Indeed, the boundedness of \(A_R\), Hölder's inequality with exponents \((\infty,2,2)\), and the Cauchy--Schwarz inequality on the bounded set \(\Omega\) give
\begin{equation*} 
\big|\langle\mathcal R_{\Omega,h}(\mathsf m_h;Q),\phi_h\rangle\big| \le C(1+\|Q\|_{L^\infty(\Omega)})(1+\|\mathsf m_h\|_{W^s(\mathbb R^d)})\|\phi_h\|_{W^s(\mathbb R^d)}. 
\end{equation*}

The discrepancy in the second Cauchy datum is represented by the affine functional
\begin{equation}
\label{eq:C-c-h-functional}
\mathcal C_{c,h}^\delta(\mathsf m_h)\in(V_h^c)^*,
\qquad
\big\langle
\mathcal C_{c,h}^\delta(\mathsf m_h),\eta_h
\big\rangle
:=
A_R(\mathsf m_h,\eta_h)
-
(\tau_h^c,\eta_h)_{L^2(W_c)}
\end{equation}
for every \(\eta_h\in V_h^c\). By the boundedness of \(A_R\) and the Cauchy--Schwarz inequality,
\begin{equation*} 
\big|\langle\mathcal C_{c,h}^\delta(\mathsf m_h),\eta_h\rangle\big| \le C\left(\|\mathsf m_h\|_{W^s(\mathbb R^d)}+\|\tau_h^c\|_{L^2(W_c)}\right)\|\eta_h\|_{W^s(\mathbb R^d)}. \end{equation*}

Let \( \alpha_v>0\), \(\alpha_\tau>0\), and \(\alpha_m>0\). Here \(\alpha_v\) and \(\alpha_\tau\) weight the two exterior Cauchy residuals, while \(\alpha_m\) is the coercive Tikhonov parameter. For \(\mathsf m_h\in V_h^R\), define
\begin{align}
\nonumber J_{m,h}^\delta(\mathsf m_h)&:={} \frac12\big\|\mathcal R_{\Omega,h} (\mathsf m_h;\widehat Q_h^\Omega) \big\|_{-s,\Omega,h}^2 + \frac{\alpha_v}{2} \big\| \Pi_h^c(\mathsf m_h|_{W_c})-\Pi_h^c(m_h^W|_{W_c})
\big\|_{L^2(W_c)}^2\\& \qquad + \frac{\alpha_\tau}{2} \big\| \mathcal C_{c,h}^\delta(\mathsf m_h)
\big\|_{-s,c,h}^2 + \frac{\alpha_m}{2} \|\mathsf m_h\|_{W^s(\mathbb R^d)}^2 .
\label{eq:m-cauchy-functional}
\end{align}
The discrete Cauchy--Liouville reconstruction is then defined by
\begin{equation}
\label{eq:m-cauchy-minimization}
m_h \in \operatorname*{argmin}_{\mathsf m_h\in V_h^R} J_{m,h}^\delta(\mathsf m_h).
\end{equation}

Let $(\cdot,\cdot)_{-s,\Omega,h}$ and $(\cdot,\cdot)_{-s,c,h}$ be the Hilbert inner products associated with the respective dual norms. Equip $V_h^R$ with the $W^s(\mathbb R^d)$-inner product and introduce the product Hilbert space
\[
\begin{gathered}
Y_{m,h}:=(V_h^\Omega)^*\times V_h^c\times(V_h^c)^*,\qquad
\|(r,d,t)\|_{Y_{m,h}}^2:=\|r\|_{-s,\Omega,h}^2+\|d\|_{L^2(W_c)}^2+\|t\|_{-s,c,h}^2.
\end{gathered}
\]
For $Q\in L^\infty(\Omega)$ and $\tau\in L^2(W_c)$, define
\[
\begin{gathered}
\mathcal L_{Q,h}v:=\mathcal R_{\Omega,h}(v;Q)-\mathcal R_{\Omega,h}(0;Q), \qquad
\langle\mathcal A_{c,h}v,\eta_h\rangle:=A_R(v,\eta_h),\qquad
\ell_{\tau,h}(\eta_h):=(\tau,\eta_h)_{L^2(W_c)}.
\end{gathered}
\]
Set
\begin{equation}\label{eq:m-cauchy-stacked-form}
\begin{aligned}
\mathbb L_{Q,h}v&:=\bigl(\mathcal L_{Q,h}v,\sqrt{\alpha_v}\Pi_h^c(v|_{W_c}),\sqrt{\alpha_\tau}\mathcal A_{c,h}v\bigr),\\
\mathsf b_{Q,h}^\delta&:=\bigl(-\mathcal R_{\Omega,h}(0;Q),\sqrt{\alpha_v}\Pi_h^c(m_h^W|_{W_c}),\sqrt{\alpha_\tau}\ell_{\tau_h^c,h}\bigr).
\end{aligned}
\end{equation}
Both the operator and the first data component depend on $Q$. For $\widehat Q=\widehat Q_h^\Omega$, the functional \eqref{eq:m-cauchy-functional} equals $\tfrac12\|\mathbb L_{\widehat Q,h}m-\mathsf b_{\widehat Q,h}^\delta\|_{Y_{m,h}}^2+\tfrac{\alpha_m}{2}\|m\|_{W^s(\R^{d})}^2$. Lemma~\ref{lem:abstract-shifted-tikhonov}, with center zero and parameter $\alpha_m$, therefore gives a unique minimizer characterized by
\begin{equation}\label{eq:m-cauchy-normal-equation}
(\mathbb L_{\widehat Q,h}m_h-\mathsf b_{\widehat Q,h}^\delta,\mathbb L_{\widehat Q,h}\mathsf v_h)_{Y_{m,h}}+\alpha_m(m_h,\mathsf v_h)_{W^s(\mathbb R^d)}=0,\qquad \mathsf v_h\in V_h^R.
\end{equation}
The matrix of \eqref{eq:m-cauchy-normal-equation} is symmetric positive definite in any basis of $V_h^R$. Estimate \eqref{eq:abstract-tikhonov-apriori} gives
\begin{equation*} \begin{aligned} 
\sqrt{\alpha_m}\|m_h\|_{W^s(\mathbb R^d)} &\le\|\mathsf b_{\widehat Q,h}^\delta\|_{Y_{m,h}}\\ &=\Bigl(\|\mathcal R_{\Omega,h}(0;\widehat Q)\|_{-s,\Omega,h}^2+\alpha_v\|\Pi_h^c(m_h^W|_{W_c})\|_{L^2(W_c)}^2+\alpha_\tau\|\mathcal C_{c,h}^\delta(0)\|_{-s,c,h}^2\Bigr)^{1/2}. 
\end{aligned} \end{equation*}
The same strict-convexity argument gives a unique minimizer on any prescribed reduced affine trial space, with the full lifted function retained in the residuals and energy penalty. Finite-dimensional solvability does not require density; the consistency and logarithmic results retain their stated approximation and observation-space hypotheses.

For stability, compare the computed Cauchy data with their exact projections onto the same observation space. Define
\begin{equation*} 
\eta_{m,h}^c := \|\Pi_h^c(m_h^W|_{W_c})-\Pi_h^c m_\gamma|_{W_c}\|_{L^2(W_c)}, \qquad \eta_{\tau,h}^c := \|\tau_h^c-\Pi_h^c\tau_\gamma\|_{L^2(W_c)}.
\end{equation*}
 Since \(W_c\subset W_0\), the definition \(e_h^\delta=\|\Gamma_h^\delta-\gamma\|_{L^\infty(W_0)}\), the \(L^2(W_c)\)-contractivity of \(\Pi_h^c\), and the Lipschitz continuity of \(t\mapsto\sqrt t\) on \([\gamma_0/2,2\gamma_1]\) give
\begin{equation}
\label{eq:cauchy-value-propagated-error}
\begin{aligned}
\eta_{m,h}^c =
\|\Pi_h^c((m_h^W-m_\gamma)|_{W_c})\|_{L^2(W_c)} \le C|W_c|^{1/2} \|\Gamma_h^\delta-\gamma\|_{L^\infty(W_c)} \le
C|W_c|^{1/2}e_h^\delta .
\end{aligned}
\end{equation}
For the fractional-Laplacian datum, \eqref{eq:h-c-definition} and
\(\tau_\gamma=-\sqrt\gamma\,q_\gamma^c\) imply \(
    \tau_h^c-\Pi_h^c\tau_\gamma  = -\Pi_h^c\!\left(\sqrt{\Gamma_h^\delta}\,q_h^c\right)
    + \Pi_h^c\!\left(\sqrt\gamma\,q_\gamma^c\right)
\). Therefore,
\begin{equation*} \begin{aligned} 
\eta_{\tau,h}^c \le \left\|\sqrt{\Gamma_h^\delta}\,q_h^c-\sqrt\gamma\,q_\gamma^c\right\|_{L^2(W_c)} \le C\|q_h^c-q_\gamma^c\|_{L^2(W_c)} + Ce_h^\delta\|q_\gamma^c\|_{L^2(W_c)}. \end{aligned}
\end{equation*}

Let \(\mathcal C_{c,h}^e\) denote the functional in \eqref{eq:C-c-h-functional} with \(\tau_h^c\) replaced by \(\Pi_h^c\tau_\gamma\). Using the same regularization parameters as in \eqref{eq:m-cauchy-functional}, define
\begin{align*} 
J_{m,h}^e(\mathsf m_h) :={}& \frac12 \|\mathcal R_{\Omega,h}(\mathsf m_h;Q_\gamma)\|_{-s,\Omega,h}^2 + \frac{\alpha_v}{2}\|\Pi_h^c(\mathsf m_h|_{W_c})-\Pi_h^c m_\gamma|_{W_c}\|_{L^2(W_c)}^2 \\ &+ \frac{\alpha_\tau}{2}\|\mathcal C_{c,h}^e(\mathsf m_h)\|_{-s,c,h}^2 + \frac{\alpha_m}{2}\|\mathsf m_h\|_{W^s(\mathbb R^d)}^2, \qquad \mathsf m_h\in V_h^R,
\end{align*}
and let
\begin{equation}
\label{eq:m-cauchy-exact-minimizer}
m_h^e
:=
\operatorname*{argmin}_{\mathsf m_h\in V_h^R}
J_{m,h}^e(\mathsf m_h).
\end{equation}
Define $\mathsf b_{Q,h}^e$ by \eqref{eq:m-cauchy-stacked-form} with $\Pi_h^c(m_h^W|_{W_c})$ and $\tau_h^c$ replaced by $\Pi_h^c(m_\gamma|_{W_c})$ and $\Pi_h^c\tau_\gamma$, respectively. Lemma~\ref{lem:abstract-shifted-tikhonov}, with operator $\mathbb L_{Q_\gamma,h}$ and data $\mathsf b_{Q_\gamma,h}^e$, gives existence and uniqueness of $m_h^e$.

\begin{lemma}[Perturbation of the interior Liouville residual]
\label{lem:m-cauchy-potential-perturbation}
Let \(Q_1,Q_2\in L^\infty(\Omega)\). There exists a constant \(C>0\), independent of \(h\), such that for every \(\mathsf m_h,\mathsf v_h\in V_h^R\),
\begin{align}
& \|\mathcal R_{\Omega,h}(\mathsf m_h;Q_1)
- \mathcal R_{\Omega,h}(\mathsf m_h;Q_2)\|_{-s,\Omega,h}
\le C\|Q_1-Q_2\|_{L^{\frac d{2s}}(\Omega)} \bigl(1+\|\mathsf m_h\|_{W^s(\mathbb R^d)}\bigr),
\label{eq:m-cauchy-affine-potential-perturbation}
\\
& \left\| \left[\mathcal R_{\Omega,h}(\mathsf v_h;Q_1) - \mathcal R_{\Omega,h}(0;Q_1) \right]
- \left[ \mathcal R_{\Omega,h}(\mathsf v_h;Q_2) - \mathcal R_{\Omega,h}(0;Q_2) \right]
\right\|_{-s,\Omega,h}\nonumber
\\& \hspace{8cm}\le C\|Q_1-Q_2\|_{L^{\frac d{2s}}(\Omega)}
\|\mathsf v_h\|_{W^s(\mathbb R^d)}.
\label{eq:m-cauchy-linear-potential-perturbation}
\end{align}
Moreover,
\begin{align}
\|\mathcal R_{\Omega,h}(\mathsf m_h;Q)\|_{-s,\Omega,h}
&\le
C\bigl(1+\|Q\|_{L^{\frac d{2s}}(\Omega)}\bigr)
\bigl(1+\|\mathsf m_h\|_{W^s(\mathbb R^d)}\bigr),
\label{eq:m-cauchy-residual-uniform-bound}\\
\|\mathcal R_{\Omega,h}(\mathsf v_h;Q)
-
\mathcal R_{\Omega,h}(0;Q)\|_{-s,\Omega,h}
&\le
C\bigl(1+\|Q\|_{L^{\frac d{2s}}(\Omega)}\bigr)
\|\mathsf v_h\|_{W^s(\mathbb R^d)}.
\label{eq:m-cauchy-linear-residual-bound}
\end{align}
\end{lemma}

\begin{proof}
For \(q\in L^{d/(2s)}(\Omega)\), \(z\in L^{p_s^*}(\Omega)\), and \(\phi_h\in V_h^\Omega\), H\"older's inequality and \(H^s(\mathbb R^d)\hookrightarrow L^{p_s^*}(\mathbb R^d)\) give
\[
 \left|\int_\Omega qz\phi_h\,dx\right|
 \le C\|q\|_{L^{d/(2s)}(\Omega)}\|z\|_{L^{p_s^*}(\Omega)}\|\phi_h\|_{W^s(\mathbb R^d)},
 \qquad \frac{2s}{d}+\frac2{p_s^*}=1.
\]
By \eqref{eq:R-omega-h-functional}, take \(q=Q_1-Q_2\) and \(z=1+\mathsf m_h\) or \(z=\mathsf v_h\), and then take the supremum over \(\phi_h\), to obtain \eqref{eq:m-cauchy-affine-potential-perturbation} and \eqref{eq:m-cauchy-linear-potential-perturbation}. The same estimate with \(q=Q\), together with continuity of \(A_R\), proves \eqref{eq:m-cauchy-residual-uniform-bound} and \eqref{eq:m-cauchy-linear-residual-bound}.
\end{proof}

\begin{proposition}[Stability of the regularized continuation]
\label{prop:m-cauchy-stability}
Let \(m_h\) and \(m_h^e\) be defined by
\eqref{eq:m-cauchy-minimization} and
\eqref{eq:m-cauchy-exact-minimizer}, respectively. Then
\begin{align}
\|m_h-m_h^e\|_{W^s(\mathbb R^d)}
\le{}&
\frac{1}{2\sqrt{\alpha_m}}
\left[ \alpha_v(\eta_{m,h}^c)^2
+ \alpha_\tau(\eta_{\tau,h}^c)^2
\right]^{1/2}+
\frac{C\varepsilon_{Q,h}}{\sqrt{\alpha_m}} \left(
1+\|m_h^e\|_{W^s(\mathbb R^d)}
\right)
\nonumber\\
&+
\frac{C\varepsilon_{Q,h}}{\alpha_m}
\|\mathcal R_{\Omega,h}(m_h^e;Q_\gamma)\|_{-s,\Omega,h},
\label{eq:m-cauchy-stability}
\end{align}
where \(C>0\) is independent of \(h\), \(\delta\), and the
regularization parameters. The constant may depend on \(d,s,\Omega,Q_*\), and \(\|Q_\gamma\|_{L^\infty(\Omega)}\). Furthermore, there is a constant \(C_\gamma>0\) such that
\begin{equation}
\label{eq:m-cauchy-exact-apriori}
\sqrt{\alpha_m}\,
\|m_h^e\|_{W^s(\mathbb R^d)}
+
\|\mathcal R_{\Omega,h}(m_h^e;Q_\gamma)\|_{-s,\Omega,h}
\le
C_\gamma .
\end{equation}
Here \(C_\gamma\) is independent of \(h\), \(\delta\), and \(\alpha_m\);
it may depend on the fixed weights \(\alpha_v,\alpha_\tau\), on
\(Q_\gamma\), \(m_\gamma|_{W_c}\), and \(\tau_\gamma\). Consequently, for \(0<\alpha_m\le1\),
\begin{equation}
\label{eq:m-cauchy-stability-simplified}
\|m_h-m_h^e\|_{W^s(\mathbb R^d)}
\le
\frac{1}{2\sqrt{\alpha_m}}
\left[
\alpha_v(\eta_{m,h}^c)^2
+
\alpha_\tau(\eta_{\tau,h}^c)^2
\right]^{1/2}
+
\frac{C_\gamma\varepsilon_{Q,h}}{\alpha_m}.
\end{equation}
\end{proposition}

\begin{proof}
Write $Q=Q_\gamma$ and $\widehat Q=\widehat Q_h^\Omega$. Let $\overline m_h$ be the minimizer with operator $\mathbb L_{\widehat Q,h}$ and exact projected data $\mathsf b_{\widehat Q,h}^e$. Since $\|\ell_{\tau_h^c-\Pi_h^c\tau_\gamma,h}\|_{-s,c,h}\le\eta_{\tau,h}^c$, estimate \eqref{eq:abstract-tikhonov-data-stability} yields
\begin{equation*} 
\|m_h-\overline m_h\|_{W^s(\mathbb R^d)}\le\frac1{2\sqrt{\alpha_m}}\left[\alpha_v(\eta_{m,h}^c)^2+\alpha_\tau(\eta_{\tau,h}^c)^2\right]^{1/2}. 
\end{equation*}
For $e_2:=\overline m_h-m_h^e$, subtracting the normal equations gives
\begin{align*} 
\|\mathbb L_{\widehat Q,h}e_2\|_{Y_{m,h}}^2+\alpha_m\|e_2\|_{W^s(\R^{d})}^2&=\bigl(\mathcal R_{\Omega,h}(m_h^e;Q)-\mathcal R_{\Omega,h}(m_h^e;\widehat Q),\mathcal L_{\widehat Q,h}e_2\bigr)_{-s,\Omega,h}\\ &\qquad+\bigl(\mathcal R_{\Omega,h}(m_h^e;Q),(\mathcal L_{Q,h}-\mathcal L_{\widehat Q,h})e_2\bigr)_{-s,\Omega,h}. 
\end{align*}
In \eqref{eq:abstract-tikhonov-operator-perturbation}, the effective data change $(\mathsf b_{\widehat Q,h}^e-\mathsf b_{Q,h}^e)-(\mathbb L_{\widehat Q,h}-\mathbb L_{Q,h})m_h^e$ has only the first component $\mathcal R_{\Omega,h}(m_h^e;Q)-\mathcal R_{\Omega,h}(m_h^e;\widehat Q)$. The adjoint perturbation likewise pairs only with the interior residual. Lemma~\ref{lem:m-cauchy-potential-perturbation} therefore gives
\begin{align*} 
\|\overline m_h-m_h^e\|_{W^s(\mathbb R^d)} \le\frac{C\varepsilon_{Q,h}}{\sqrt{\alpha_m}}(1+\|m_h^e\|_{W^s(\mathbb R^d)})+\frac{C\varepsilon_{Q,h}}{\alpha_m}\|\mathcal R_{\Omega,h}(m_h^e;Q_\gamma)\|_{-s,\Omega,h}. \end{align*}
The triangle inequality proves \eqref{eq:m-cauchy-stability}. The comparison estimate \eqref{eq:abstract-tikhonov-comparison}, with comparator zero, bounds the exact-data residual norm and $\sqrt{\alpha_m}\|m_h^e\|_{W^s(\R^{d})}$ by $C\|\mathsf b_{Q,h}^e\|_{Y_{m,h}}$. Contractivity of $\Pi_h^c$, $\|\ell_{\Pi_h^c\tau_\gamma,h}\|_{-s,c,h}\le\|\tau_\gamma\|_{L^2(W_c)}$, and \eqref{eq:m-cauchy-residual-uniform-bound} bound $\|\mathsf b_{Q,h}^e\|_{Y_{m,h}}$ independently of $h$, proving \eqref{eq:m-cauchy-exact-apriori}. Substitution into \eqref{eq:m-cauchy-stability}, for $0<\alpha_m\le1$, gives \eqref{eq:m-cauchy-stability-simplified}.
\end{proof}

Proposition~\ref{prop:m-cauchy-stability} controls the perturbation
caused by the three reconstructed inputs. The remaining error is the
exact-input discretization and regularization error. More precisely,
\begin{equation*} \|m_h-m_\gamma\|_{W^s(\mathbb R^d)} \le \|m_h-m_h^e\|_{W^s(\mathbb R^d)} + \|m_h^e-m_\gamma\|_{W^s(\mathbb R^d)}. \end{equation*}
The first term is bounded by
\eqref{eq:m-cauchy-stability}. The second term will be estimated by
comparing \(m_h^e\) with an appropriate finite element approximation
of \(m_\gamma\), separating the interior Liouville residual, the two
projected Cauchy residuals, the finite element approximation error,
and the regularization bias.

For the convergence analysis below, the residual weights \(\alpha_v>0\) and \(\alpha_\tau>0\) are kept fixed, independently of \(h\) and \(\delta\). Only the coercive Tikhonov parameter \(\alpha_m=\alpha_m(h)\) is allowed to tend to zero. We also use the density property
\begin{equation}\label{eq:VR-density}
\lim_{h\to0} \inf_{v_h\in V_h^R} \|v-v_h\|_{W^s(\mathbb R^d)} =0
\qquad \forall v\in\widetilde H^s(\Omega_R).
\end{equation}
Indeed, \(C_c^\infty(\Omega_R)\) is dense in \(\widetilde H^s(\Omega_R)\). A compactly supported smooth function can be approximated by fitted-mesh Lagrange interpolation in \(H^1(\mathbb R^d)\), and hence also in \(H^s(\mathbb R^d)\).

Let us take \(m_{\gamma,h}\) to be a near-best approximation of \(m_\gamma\) in \(V_h^R\). The condition \(\|m_{\gamma,h}-m_\gamma\|_{W^s(\mathbb R^d)}^2=o(\alpha_m)\)
is therefore a standard coupling between the finite element approximation error and the final Tikhonov parameter.

\begin{proposition}
\label{prop:m-cauchy-exact-consistency}
Assume Assumption~\ref{ass:liouville-regularity},
\eqref{eq:cauchy-space-density}, \eqref{eq:cauchy-projection-consistency}, and
\eqref{eq:VR-density}. Let \(m_h^e\) be defined by
\eqref{eq:m-cauchy-exact-minimizer}. Then, for every
\(m_{\gamma,h}\in V_h^R\), there holds
\begin{align}
&\|\mathcal R_{\Omega,h}(m_h^e;Q_\gamma)\|_{-s,\Omega,h}^2
+
\alpha_v
\|\Pi_h^c((m_h^e-m_\gamma)|_{W_c})\|_{L^2(W_c)}^2
+
\alpha_\tau
\|\mathcal C_{c,h}^e(m_h^e)\|_{-s,c,h}^2
+
\alpha_m\|m_h^e\|_{W^s(\mathbb R^d)}^2
\nonumber\\
&\qquad\le
C\|m_{\gamma,h}-m_\gamma\|_{W^s(\mathbb R^d)}^2
+
\alpha_m\|m_{\gamma,h}\|_{W^s(\mathbb R^d)}^2 .
\label{eq:m-cauchy-exact-comparison-bound}
\end{align}
If, in addition, \(\alpha_m(h)\to0\) and there exists
\(m_{\gamma,h}\in V_h^R\) such that
\begin{equation}
\label{eq:m-gamma-comparison-rule}
\|m_{\gamma,h}-m_\gamma\|_{W^s(\mathbb R^d)}
\longrightarrow0,
\qquad
\frac{\|m_{\gamma,h}-m_\gamma\|_{W^s(\mathbb R^d)}^2}{\alpha_m}
\longrightarrow0,
\end{equation}
then \( m_h^e\longrightarrow m_\gamma\) strongly in \(W^s(\mathbb R^d)\).
\end{proposition}

\begin{proof}
Set \(e_{I,h}=m_{\gamma,h}-m_\gamma\). Since \(m_\gamma\) satisfies \eqref{eq:continuous-cauchy-liouville-problem}, continuity of \(A_R\), H\"older--Sobolev estimates, and contractivity of \(\Pi_h^c\) give
\[
 \|\mathcal R_{\Omega,h}(m_{\gamma,h};Q_\gamma)\|_{-s,\Omega,h}
 +\|\Pi_h^c(e_{I,h}|_{W_c})\|_{L^2(W_c)}
 +\|\mathcal C_{c,h}^e(m_{\gamma,h})\|_{-s,c,h}
 \le C\|e_{I,h}\|_{W^s(\mathbb R^d)}.
\]
Minimality, \(J_{m,h}^e(m_h^e)\le J_{m,h}^e(m_{\gamma,h})\), proves \eqref{eq:m-cauchy-exact-comparison-bound}. Under \eqref{eq:m-gamma-comparison-rule} and \(\alpha_m\to0\), this comparison makes all three residuals vanish and gives
\[
 \|m_h^e\|_{W^s(\R^{d})}^2\le C\frac{\|e_{I,h}\|_{W^s(\R^{d})}^2}{\alpha_m}+\|m_{\gamma,h}\|_{W^s(\R^{d})}^2,
 \qquad
 \limsup_{h\to0}\|m_h^e\|_{W^s(\R^{d})}^2\le\|m_\gamma\|_{W^s(\R^{d})}^2.
\]
Here \(\alpha_v,\alpha_\tau\) are fixed positive weights. Let a subsequence converge weakly to \(\overline m\) in \(\widetilde H^s(\Omega_R)\). For \(\eta\in C_c^\infty(W_c)\), choose \(\eta_h\in V_h^c\) converging to \(\eta\) in \(W^s\) by \eqref{eq:cauchy-space-density}. Orthogonality of \(\Pi_h^c\), the exact Cauchy identity, and the vanishing residuals imply
\[
 \begin{aligned}
 (m_h^e-m_\gamma,\eta_h)_{L^2(W_c)}
 &=\bigl(\Pi_h^c((m_h^e-m_\gamma)|_{W_c}),\eta_h\bigr)_{L^2(W_c)}\longrightarrow0,\\
 A_R(m_h^e-m_\gamma,\eta_h)
 &=\langle\mathcal C_{c,h}^e(m_h^e),\eta_h\rangle\longrightarrow0.
 \end{aligned}
\]
Passing to the limit yields \(\overline m-m_\gamma=(-\Delta)^s(\overline m-m_\gamma)=0\) in \(W_c\). Unique continuation \eqref{eq:sec2-ucp} gives \(\overline m=m_\gamma\); hence the entire sequence converges weakly to \(m_\gamma\). Weak lower semicontinuity and the displayed limsup bound give convergence of the \(W^s\)-norms, and therefore strong convergence in \(W^s(\mathbb R^d)\).
\end{proof}

\begin{corollary}[Perturbation estimate adapted to the exact-input comparison]
\label{cor:m-cauchy-sharp-stability}
Under the hypotheses of Propositions~\ref{prop:m-cauchy-stability} and~\ref{prop:m-cauchy-exact-consistency}, there holds
\begin{align}
\|m_h-m_h^e\|_{W^s(\mathbb R^d)} \le{}&
\frac{1}{2\sqrt{\alpha_m}} \left[ \alpha_v(\eta_{m,h}^c)^2 + \alpha_\tau(\eta_{\tau,h}^c)^2 \right]^{1/2} + \frac{C\varepsilon_{Q,h}}{\sqrt{\alpha_m}} \left( 1+\|m_{\gamma,h}\|_{W^s(\mathbb R^d)} \right)
\nonumber\\ &+ \frac{C\varepsilon_{Q,h}}{\alpha_m} \|m_{\gamma,h}-m_\gamma\|_{W^s(\mathbb R^d)} \qquad  \forall ~m_{\gamma,h}\in V_h^R.
\label{eq:m-cauchy-sharp-stability}
\end{align}
\end{corollary}

\begin{proof}
Estimate \eqref{eq:m-cauchy-exact-comparison-bound} implies
\[
\|m_h^e\|_{W^s(\R^{d})} \le C\frac{ \|m_{\gamma,h}-m_\gamma\|_{W^s(\R^{d})}}{\sqrt{\alpha_m}}
+ \|m_{\gamma,h}\|_{W^s(\R^{d})}
\]
and
\[
\|\mathcal R_{\Omega,h}(m_h^e;Q_\gamma)\|_{-s,\Omega,h} \le C\|m_{\gamma,h}-m_\gamma\|_{W^s(\R^{d})} + \sqrt{\alpha_m}\, \|m_{\gamma,h}\|_{W^s(\R^{d})}.
\]
Substitution into \eqref{eq:m-cauchy-stability}, followed by the triangle inequality, gives \eqref{eq:m-cauchy-sharp-stability}.
\end{proof}

We define the admissible interior conductivity reconstruction by
\begin{equation}\label{eq:gamma-interior-reconstruction}
\gamma_h^\Omega(x) := \left[ \min\left\{ \sqrt{2\gamma_1}, \max\left\{\sqrt{\gamma_0/2},
1+m_h(x) \right\} \right\} \right]^2, \qquad x\in\Omega.
\end{equation}
Then \(\frac{\gamma_0}{2} \le \gamma_h^\Omega \le 2\gamma_1\) almost everywhere in \(\Omega\).

With the potential error \(\varepsilon_{Q,h}\) and projected Cauchy-data errors \(\eta_{m,h}^c\) and \(\eta_{\tau,h}^c\), the combined input error \eqref{eq:final-combined-input-error} takes the form
\[
E_h^{\rm in}=\varepsilon_{Q,h}+\sqrt{\alpha_v}\,\eta_{m,h}^c+\sqrt{\alpha_\tau}\,\eta_{\tau,h}^c.
\]

\begin{theorem}[Final interior error decomposition and convergence]
\label{thm:final-interior-conductivity-error}
Let \(\gamma_h^\Omega\) be defined by \eqref{eq:gamma-interior-reconstruction}. Under the hypotheses of Proposition~\ref{prop:m-cauchy-exact-consistency} and Corollary~\ref{cor:m-cauchy-sharp-stability}, there holds
\begin{align}
\|\gamma_h^\Omega-\gamma\|_{L^2(\Omega)}
\le C\Big[ \|m_h^e-m_\gamma\|_{L^2(\mathbb R^d)} +\frac{E_h^{\rm in}}{\sqrt{\alpha_m}}+
\frac{\varepsilon_{Q,h}}{\sqrt{\alpha_m}} \|m_{\gamma,h}\|_{W^s(\mathbb R^d)}
+ \frac{\varepsilon_{Q,h}}{\alpha_m} \|m_{\gamma,h}-m_\gamma\|_{W^s(\mathbb R^d)} \Big].
\label{eq:final-interior-conductivity-error}
\end{align}
In particular, suppose that \(\alpha_m\to0\), that \(m_{\gamma,h}\in V_h^R\) satisfies
\begin{equation}\label{eq:m-gamma-comparison-rule-compact}
    \|m_{\gamma,h}-m_\gamma\|_{W^s(\mathbb R^d)}\to0, \qquad \frac{\|m_{\gamma,h}-m_\gamma\|_{W^s(\mathbb R^d)}^2}{\alpha_m} \to0,
\end{equation}
and that \(\frac{E_h^{\rm in}}{\sqrt{\alpha_m}}\to0 \). Then
\[
    m_h\longrightarrow m_\gamma \quad\text{in }W^s(\mathbb R^d), \qquad \gamma_h^\Omega\longrightarrow\gamma \quad\text{in }L^2(\Omega).
\]
\end{theorem}

\begin{proof}
Let \(T(r):=\min\left\{\sqrt{2\gamma_1},    \max\left\{\sqrt{\gamma_0/2},r\right\}\right\}
\). Then \(T\) is one-Lipschitz and \(T(1+m_\gamma)=\sqrt\gamma\). Hence
\[
 |\gamma_h^\Omega-\gamma| = |T(1+m_h)^2-\gamma|  \le  C|m_h-m_\gamma|
    \qquad\text{almost everywhere in }\Omega.
\]
Therefore \(\|\gamma_h^\Omega-\gamma\|_{L^2(\Omega)}\leq C\|m_h-m_\gamma\|_{L^2(\Omega)}\leq C\bigl(\|m_h-m_h^e\|_{W^s(\mathbb R^d)}+\|m_h^e-m_\gamma\|_{L^2(\mathbb R^d)}\bigr)\). Applying \eqref{eq:m-cauchy-sharp-stability} and \eqref{eq:final-combined-input-error} gives \eqref{eq:final-interior-conductivity-error}. Proposition~\ref{prop:m-cauchy-exact-consistency} gives \(m_h^e\to m_\gamma\) in \(W^s(\mathbb R^d)\). Conditions \eqref{eq:m-gamma-comparison-rule-compact} and \(E_h^{\rm in}/\sqrt{\alpha_m}\to0\) make all perturbation terms tend to zero. Thus \(m_h\to m_\gamma\) in \(W^s(\mathbb R^d)\), and the conductivity convergence follows.
\end{proof}

A sufficient concrete form of \(\frac{E_h^{\rm in}}{\sqrt{\alpha_m}}\to0 \) is obtained from \eqref{eq:cauchy-value-propagated-error}, \eqref{eq:tail-reconstruction-error}, \eqref{eq:tail-potential-error-from-L2}, and \eqref{eq:linfty_rate_sim}. Namely, it is enough that
\[
    \frac{1}{\sqrt{\alpha_m}}
    \left[ \mathfrak c_h+\delta+h^{\beta_c} +\rho_{q,h}^{(2)} +\frac{\xi_{t,h}}{2\sqrt{\alpha_t}}+ \frac{\alpha_t}{\sigma_{t,h}^2+\alpha_t} \|q_{I,h}^c\|_{L^2(W_c)}
        +\|Q_h^\Omega-Q_\gamma\|_{L^2(\Omega)}^{\theta_Q} \right] \longrightarrow0.
\]

\begin{proposition}[Logarithmic estimate from discrete Cauchy data]\label{prop:discrete-cauchy-logarithmic}
Use the full fitted space $V_h^c$ and projection \eqref{eq:Pi-h-c} on shape-regular, quasi-uniform meshes. For $e\in\widetilde H^s(\Omega_R)$, $\|e\|_{W^s(\mathbb R^d)}\le M$, $M>0$, set
\[
\begin{aligned}
d_h(e):=\|\Pi_h^c(e|_{W_c})\|_{L^2(W_c)}, \qquad
\mathrm{r}_h(e):=\sup_{0\ne\eta_h\in V_h^c}\frac{|A_R(e,\eta_h)|}{\|\eta_h\|_{W^s(\mathbb R^d)}},\qquad
\vartheta_h(e)&:=\frac{d_h(e)+\mathrm{r}_h(e)}{M}+h^s.
\end{aligned}
\]
Constants $C,\sigma,h_0>0$, depending only on the geometry, $s$, polynomial degree, and mesh regularity, satisfy
\begin{equation}\label{eq:discrete-cauchy-logarithmic}
\|e\|_{L^2(\mathbb R^d)}\le CM[\log(2+\vartheta_h(e)^{-1})]^{-\sigma},\qquad 0<h<h_0.
\end{equation}
\end{proposition}

\begin{proof}
Choose a ball $B_{4r}\Subset W_c$ and smooth cutoffs $0\le\chi,\zeta\le1$ with $\chi=1$ on $B_{3r}$, $\operatorname{supp}\chi\subset B_{4r}$, $\zeta=1$ on $B_r$, and $\operatorname{supp}\zeta\subset B_{2r}$, for small \(r>0\); all balls are concentric. After translation, we take their common center to be the origin. Since \(\operatorname{supp}\chi\cup\operatorname{supp}\zeta\Subset W_c\), fix \(h_0>0\) small enough that, for \(0<h<h_0\), the averaged interpolants \(\mathcal I_h^c(\chi^2e)\) and \(\mathcal I_h^c(\zeta\phi)\) belong to \(V_h^c\) and satisfy \eqref{eq:local-averaged-interpolation}. Write $\vartheta=\vartheta_h(e)$. Orthogonality of $\Pi_h^c$ and \eqref{eq:local-averaged-interpolation} give
\[
\begin{aligned}
\|\chi e\|_{L^2(\R^{d})}^2=(\Pi_h^c e,\mathcal I_h^c(\chi^2e))_{L^2(W_c)}+(e,\chi^2e-\mathcal I_h^c(\chi^2e))_{L^2(W_c)}\le CMd_h(e)+Ch^sM^2\le CM^2\vartheta.
\end{aligned}
\]
For \(\phi\in H^{2s}(\mathbb R^d)\), \(\mathcal I_h^c(\zeta\phi)\in V_h^c\). The zero-extension identity following \eqref{eq:AR-standard-form}, the definition of \(\mathrm r_h(e)\), and \eqref{eq:local-averaged-interpolation} give
\[
\begin{aligned}
|\langle\zeta(-\Delta)^s e,\phi\rangle|
&=|A_R(e,\zeta\phi)|
\le |A_R(e,\mathcal I_h^c(\zeta\phi))|+|A_R(e,\zeta\phi-\mathcal I_h^c(\zeta\phi))|\\
&\le \mathrm{r}_h(e)\|\mathcal I_h^c(\zeta\phi)\|_{W^s(\mathbb R^d)}
+CM\|\zeta\phi-\mathcal I_h^c(\zeta\phi)\|_{H^s(\mathbb R^d)}
\le C(\mathrm{r}_h(e)+Mh^s)\|\phi\|_{H^{2s}(\mathbb R^d)}.
\end{aligned}
\]
Taking the supremum over \(\|\phi\|_{H^{2s}(\mathbb R^d)}=1\) yields \(\|\zeta(-\Delta)^s e\|_{H^{-2s}(\mathbb R^d)}\le C(\mathrm{r}_h(e)+Mh^s)\).
Set $w=(1-\chi)e$. The map $(-\Delta)^s:L^2\to H^{-2s}$ is bounded, and $w$ vanishes on $B_{3r}$. Hence, for $\vartheta\le1$,
\[
\begin{aligned}
\|\zeta(-\Delta)^sw\|_{H^{-2s}}\le C(\mathrm{r}_h(e)+Mh^s+\|\chi e\|_{L^2(\R^{d})})\le CM\vartheta^{1/2},\qquad
\|\zeta(-\Delta)^sw\|_{L^2(\R^{d})}\le CM.
\end{aligned}
\]
The second estimate follows from the separated kernel; interpolation gives $\|(-\Delta)^sw\|_{H^{-s}(B_r)}\le CM\vartheta^{1/4}$. Also, $\|w\|_{H^s(\mathbb R^d)}\le CM$ while $\operatorname{supp}w\Subset B_L$ for a concentric ball $B_L$ containing $\overline{\Omega_R}$. If $d\ge2$, $D:=B_L\setminus\overline{B_{5r/2}}$ is a bounded Lipschitz domain containing $\operatorname{supp}w$ and disjoint from $B_r$. For sufficiently small \(\vartheta\), Proposition~6.1 of \cite{Ghosh2020uniqueness}, with \(s'=0\), yields
\[
\|w\|_{L^2(\R^{d})}\le CM[\log(c\vartheta^{-1/4})]^{-\sigma}\le CM[\log(2+\vartheta^{-1})]^{-\sigma},
\]
since the two logarithms are comparable in this range.
For \(d=1\), the interval \(B_L\setminus\overline{B_{5r/2}}\) has two components. We adapt the continuation argument in the proof of \cite[Prop.~6.1]{Ghosh2020uniqueness} using the local estimates of \cite{RulandSalo2020low}. Choose \(I_-:=(-L-r/2,-2r)\) and \(I_+:=(2r,L+r/2)\), which together contain \(\operatorname{supp}w\), and let \(\widetilde w\) be the Caffarelli--Silvestre extension of the whole function \(w\). By \cite[Lems.~4.1--4.2]{RulandSalo2020low},
\[
\|\widetilde w(\cdot,y)\|_{L^2(\mathbb R)}\le\|w\|_{L^2(\mathbb R)}\le CM,\qquad \|y^{(1-2s)/2}\nabla\widetilde w\|_{L^2(\mathbb R\times(0,\infty))}\le CM.
\]
Its trace vanishes on \(B_r\), and its conormal derivative there has \(H^{-s}(B_r)\) norm at most \(CM\vartheta^{1/4}\). The local boundary-to-bulk estimate \cite[Prop.~5.13]{RulandSalo2020low} therefore bounds the weighted \(L^2\) norm on a fixed interior ball by \(CM\vartheta^c\), for some \(c>0\).
For an auxiliary height \(0<\varrho<1/2\), cover \((I_-\cup I_+)\times(\varrho,1)\) by at most \(C\varrho^{-2}\) balls whose radii are one fifth of their centre heights. Chains from the initial ball first reach a fixed height, move horizontally, and then descend with geometrically decreasing heights. Each chain uses at most \(C(1+|\log\varrho|)\) balls, whose fourfold enlargements remain in the upper half-plane and in a fixed-height strip. Iterating \cite[Prop.~5.4]{RulandSalo2020low}, as in the proof of \cite[Thm.~5.5]{RulandSalo2020low}, bounds the weighted norm on each covering ball by \(CM\vartheta^{c\varrho^\kappa}\), for some \(\kappa>0\). The \(L^2\) contraction above also gives \(\|y^{(1-2s)/2}\widetilde w\|_{L^2(\mathbb R\times(0,\varrho))}\le CM\varrho^{1-s}\). Summing squared norms over the cover and adding the thin-layer bound yields
\[
\sum_{\pm}\|y^{(1-2s)/2}\widetilde w\|_{L^2(I_\pm\times(0,1))}\le CM\bigl(\varrho^{-1}\vartheta^{c\varrho^\kappa}+\varrho^{1-s}\bigr).
\]
Choose \(\varrho=[\log(\vartheta^{-1})]^{-1/(2\kappa)}\) for sufficiently small \(\vartheta\). Since \(\vartheta^{c\varrho^\kappa}=e^{-c\sqrt{\log(\vartheta^{-1})}}\), the preceding estimate and comparison of the logarithms give
\[
\sum_{\pm}\|y^{(1-2s)/2}\widetilde w\|_{L^2(I_\pm\times(0,1))}\le CM[\log(2+\vartheta^{-1})]^{-\mu},\qquad \mu:=\frac{1-s}{2\kappa}>0.
\]
For \(0<t<1/2\), the weighted vertical trace estimate gives
\[
\|w\|_{L^2(I_\pm)}\le C\left(t^sM+t^{-(1-s)}\|y^{(1-2s)/2}\widetilde w\|_{L^2(I_\pm\times(0,2t))}\right).
\]
Taking \(t=[\log(2+\vartheta^{-1})]^{-\mu}\) for sufficiently small \(\vartheta\) and summing over the two intervals yields \(\|w\|_{L^2(\mathbb R)}\le CM[\log(2+\vartheta^{-1})]^{-\mu s}\).
Thus, in both cases, adding $\|\chi e\|_{L^2(\R^{d})}\le CM\vartheta^{1/2}$ proves \eqref{eq:discrete-cauchy-logarithmic} for small $\vartheta$; $\|e\|_{L^2(\R^{d})}\le M$ covers the remaining values.
\end{proof}

\begin{proof}[Proof of Theorem~\ref{cor:final-interior-logarithmic}]
The convergence assertions follow from Theorem~\ref{thm:final-interior-conductivity-error}. To prove the logarithmic estimate, use the comparison function \(m_{\gamma,h}\) and the errors \(a_{m,h}\) and \(E_h^{\rm in}\) from Theorem~\ref{cor:final-interior-logarithmic}.
The exact equation \eqref{eq:fse2} implies, for $\phi_h\in V_h^\Omega$,
\begin{align*}
\langle\mathcal R_{\Omega,h}(m_{\gamma,h};\widehat Q_h^\Omega),\phi_h\rangle
&=A_R(m_{\gamma,h}-m_\gamma,\phi_h)+(\widehat Q_h^\Omega(m_{\gamma,h}-m_\gamma),\phi_h)_{L^2(\Omega)}\\
&\quad+((\widehat Q_h^\Omega-Q_\gamma)(1+m_\gamma),\phi_h)_{L^2(\Omega)}.
\end{align*}
The clipping bound, H\"older--Sobolev inequality, and Cauchy data estimates therefore give
\[
2J_{m,h}^\delta(m_h)\le2J_{m,h}^\delta(m_{\gamma,h})\le C(a_{m,h}^2+(E_h^{\rm in})^2)+\alpha_m\|m_{\gamma,h}\|_{W^s(\R^{d})}^2.
\]
Since $\alpha_v,\alpha_\tau>0$ are fixed, $a_{m,h}/\sqrt{\alpha_m}\to0$, and $E_h^{\rm in}/\sqrt{\alpha_m}\to0$, the comparison inequality bounds $\|m_h-m_\gamma\|_{W^s(\R^{d})}$ uniformly and gives
\[
d_h(m_h-m_\gamma)+\mathrm{r}_h(m_h-m_\gamma)\le C(a_{m,h}+E_h^{\rm in}+\sqrt{\alpha_m}).
\]
Proposition~\ref{prop:discrete-cauchy-logarithmic} and the Lipschitz clipping in \eqref{eq:gamma-interior-reconstruction} prove \eqref{eq:final-interior-logarithmic}.
\end{proof}

\begin{remark}[Dependence on the original measurement error]\label{rem:m-cauchy-logarithmic-bridge}
The logarithmic modulus in \eqref{eq:final-interior-logarithmic} contains every propagated error through $E_h^{\rm in}$. A rate in the original noise requires a bound for the complete quantity $\eta_{{\rm C},h}$. For example, $\eta_{{\rm C},h}\le C[\log(2+\delta_\Lambda^{-1})]^{-\mu}$, $\mu>0$, would give a log--log rate in the continuous DN-map discrepancy $\delta_\Lambda$; such a bound is not proved here. The potential estimate \cite[Thm.~3.16 and Cor.~3.17]{DwivediRailoRuppFCP} assumes common compact support in $\Omega$ and a uniform $C^{0,s}$ bound, neither imposed here. The continuous-data estimates \cite[Thm.~1.2]{RulandSalo2020low} and \cite[Thms.~1.1--1.2]{CoviKujanpaaRailo2027partial} likewise require their additional data and admissibility assumptions.
\end{remark}

\section{Numerical experiments}\label{sec:num}

The experiments illustrate the exterior-to-interior reconstruction of Sections~\ref{sec:ext} and~\ref{sec:interior-cauchy} for manufactured conductivities in one and two dimensions. The exterior measurements implement the determination principle of \cite{covi2022global}; the interior stage uses the fractional Liouville reduction of \cite{Covi2020Inverse} and Galerkin--Tikhonov reconstruction \cite{DwivediRailoRuppFCP,Engl1996Regularization}. For each study we specify the geometry, forward mesh size \(h\), inverse spaces, regularization parameters, and error evaluation. In the experiments, \(\delta\) denotes the imposed relative noise level; the total discrepancy in the analysis additionally includes the forward approximation and assembly errors. Synthetic measurements are generated by the corresponding discrete forward problem. The norm labels in the tables and figures refer to the numerical evaluations specified in each example.

\subsection{Example 5.1. One-dimensional exterior-to-interior reconstruction.}
\label{ex:examp1_1d}
We take \(d=1\), \(s=0.35\), \(\Omega=(-1,1)\), and \(\Omega_R=(-4,4)\). The exterior data and reconstruction sets are
\[
    W=(-3.4,-1.15)\cup(1.15,3.4),\qquad W_0=(-3.05,-1.45)\cup(1.45,3.05).
\]
For the interior stage, we use
\[
\begin{aligned}
    W_a&=(-1.68,-1.50)\cup(1.50,1.68),\\
    W_b&=(-3.00,-1.84)\cup(1.84,3.00),\\
    W_c&=(-1.84,-1.50)\cup(1.50,1.84).
\end{aligned}
\]
Thus \(W_a,W_b,W_c\Subset W_0\), with
\(\operatorname{dist}(W_a,W_b)>0\), while the measurements on \(W_c\) have overlapping input and observation supports.

We use one manufactured conductivity throughout the experiment. Let
\[
\eta(t)=
\begin{cases} \displaystyle \exp\!\left(1-\frac{1}{1-t^2}\right),& |t|<1,\\[1ex]
0,& |t|\geq1,
\end{cases}
\]
and define
\[
\rho(t)=
\begin{cases}
e^{-1/t},&t>0,\\ 0,&t\leq0,
\end{cases}
\qquad
\chi(t)=
\begin{cases}
1,&t\leq0,\\[1mm]
\displaystyle
\frac{\rho(1-t)}{\rho(t)+\rho(1-t)},&0<t<1,\\[2mm]
0,&t\geq1.
\end{cases}
\]
The interior background deviation is
\[
    m_0(x) := 0.06 \left(1-\left(\frac{x}{1.10}\right)^2\right)_{+}^{\,s}
    \chi\!\left(\frac{|x|-1}{0.08}\right),
\]
and the complete manufactured conductivity is
\[
\gamma(x)
:= (1+m_0(x))^2
+0.50\,\eta\!\left(\frac{x+2.30}{0.45}\right)
+0.50\,\eta\!\left(\frac{x-2.30}{0.45}\right).
\]
Consequently,
\[
\gamma(x) = 1
+0.50\,\eta\!\left(\frac{x+2.30}{0.45}\right)
+0.50\,\eta\!\left(\frac{x-2.30}{0.45}\right)
\quad\text{in }W,
\]
whereas
\[
m_\gamma(x) = 0.06 \left(1-\left(\frac{x}{1.10}\right)^2\right)^s,
\qquad \gamma(x)=(1+m_\gamma(x))^2 \quad\text{in }\Omega.
\]
The two exterior bumps test both connected components of \(W_0\), while the interior profile represents one broad symmetric inclusion. Since \(m_0=0\) for \(|x|\geq1.08\), the interior and exterior perturbations are separated by a background collar, and the complete coefficient satisfies the positivity and compact-support requirements of the Liouville reduction.

The forward problem uses continuous \(P_1\) functions on an endpoint-fitted mesh of maximum element size at most \(h\). Both conductivity and reference forms use four-point tensor-product Gaussian quadrature, with coincident quadrature pairs omitted. Exterior samples use the capped nominal radius ratios \(r/h\in\{1.5,2,2.5,3,4,5\}\), with calibration and extrapolation as defined in Appendix~\ref{app:calibrated-dn}. After extrapolation, each sample vector is perturbed in a fixed Gaussian direction normalized in \(\ell^\infty\) to the prescribed relative amplitude \(\delta\). The exterior errors are \(\|\gamma_h^\delta-\gamma\|_{L^2(W_0)}\) and \(\|\gamma_h^\delta-\gamma\|_{L^\infty(W_0)}\), computed with the consistent mass matrix and the nodal maximum, respectively. Interior \(L^2\)-errors are evaluated with four-point Gaussian quadrature on each element, and interior maximum errors are sampled at \(1201\) equally spaced points of \([-1,1]\).

The displayed reconstructions use \(h=0.008\) and the nominal noise levels indicated in Figures~\ref{fig:ex51-complete}(A) and~\ref{fig:ex51-complete}(D). The exterior mesh study uses \(h\in\{0.064,0.032,0.016,0.008,0.004\}\), and the interior mesh study uses \(h\in\{0.032,0.016,0.008,0.004\}\) at \(\delta=10^{-5}\). At \(h=0.008\), both noise studies use \(\delta\in\{10^{-7},10^{-6},\cdots,10^{-2},10^{-1}\}\). Exterior regularization is \(\alpha=0.03h^2\). Each run of the interior study uses the exterior reconstruction computed for the same \(h\) and \(\delta\); interior DN noise is generated separately from the exterior post-extrapolation sample noise. Interior recovery uses \(J_a=J_c=8\) reflection-paired sine sources orthonormalized in the discrete \(W^s\)-inner product. The state correction has eight coarse \(P_1\) degrees of freedom. The interior and tail potentials use even piecewise-constant spaces of dimensions two and nine, and the Cauchy variable uses an even \(P_1\) space of dimension sixteen. The exterior input is reflection averaged. The implemented \(W_c\) space uses nodal \(P_1\) functions without the cell-bubble enrichment of the theoretical sufficient identifiability construction. The remaining regularization parameters are fixed:
\[
\begin{gathered}
\alpha_z=3.30\times10^{-4},\qquad \alpha_w=3\times10^{-4},\qquad \alpha_Q=10^{-8},\qquad \alpha_t=10^{-7},\\
\alpha_v=\alpha_\tau=3,\qquad \alpha_m=10^{-4},\qquad Q_*=1.5.
\end{gathered}
\]

Figure~\ref{fig:ex51-complete} and Table~\ref{tab:ex51-gamma-errors} report the one-dimensional errors. The exterior \(L^2(W_0)\)-error relative to \(I_h^{W_0}\gamma\) decreases from \(3.0722\times10^{-2}\) at \(h=0.064\) to \(3.4030\times10^{-3}\) at \(h=0.004\), alongside the decreasing noiseless calibrated-sample error \(\widehat E_h^0\) of \eqref{eq:def-Eh-cal-delta}. The corresponding error relative to \(\gamma\) is bounded by adding the interpolation error, as in \eqref{eq:recon-error-final}; \(\widehat E_h^0\) does not measure discrepancy from the continuous DN form.  The interior mesh errors are reduced at \(h=0.004\) compared to \(h=0.032\). However, these are  nonmonotone. The interior studies assess accuracy and noise sensitivity for the prescribed inverse spaces and fixed regularization parameters.

\begin{table}[!htbp]
\centering\small
\setlength{\tabcolsep}{3.4pt}\renewcommand{\arraystretch}{1.1}
\begin{tabular}{|c c c||c c c|}
\hline
\multicolumn{6}{|c|}{\textbf{Exterior reconstruction}}\\ \hline
\multicolumn{3}{|c||}{Noise study, \(h=0.008\)} & \multicolumn{3}{c|}{Mesh study, \(\delta=10^{-5}\)}\\ \hline
\(\delta\) & \(\|\gamma_h^\delta-\gamma\|_{L^2(W_0)}\) & \(\|\gamma_h^\delta-\gamma\|_{L^\infty(W_0)}\) & \(h\) & \(\|\gamma_h^\delta-\gamma\|_{L^2(W_0)}\) & \(\|\gamma_h^\delta-\gamma\|_{L^\infty(W_0)}\)\\ \hline
\(10^{-7}\) & \(5.3642\times10^{-3}\) & \(5.3186\times10^{-3}\) & \(0.004\) & \(3.4030\times10^{-3}\) & \(3.3527\times10^{-3}\)\\
\(10^{-6}\) & \(5.3642\times10^{-3}\) & \(5.3193\times10^{-3}\) & \(0.008\) & \(5.3644\times10^{-3}\) & \(5.3254\times10^{-3}\)\\
\(10^{-5}\) & \(5.3644\times10^{-3}\) & \(5.3254\times10^{-3}\) & \(0.016\) & \(8.3565\times10^{-3}\) & \(8.1316\times10^{-3}\)\\
\(10^{-4}\) & \(5.3659\times10^{-3}\) & \(5.3874\times10^{-3}\) & \(0.032\) & \(1.4279\times10^{-2}\) & \(1.7581\times10^{-2}\)\\ \hline\hline
\multicolumn{6}{|c|}{\textbf{Interior reconstruction}}\\ \hline
\multicolumn{3}{|c||}{Noise study, \(h=0.008\)} & \multicolumn{3}{c|}{Mesh study, \(\delta=10^{-5}\)}\\ \hline
\(\delta\) & \(\|\gamma_h^\Omega-\gamma\|_{L^2(\Omega)}\) & \(\|\gamma_h^\Omega-\gamma\|_{L^\infty(\Omega)}\) & \(h\) & \(\|\gamma_h^\Omega-\gamma\|_{L^2(\Omega)}\) & \(\|\gamma_h^\Omega-\gamma\|_{L^\infty(\Omega)}\)\\ \hline
\(10^{-7}\) & \(6.1173\times10^{-3}\) & \(1.1339\times10^{-2}\) & \(0.004\) & \(6.7970\times10^{-3}\) & \(1.2186\times10^{-2}\)\\
\(10^{-6}\) & \(6.1134\times10^{-3}\) & \(1.1335\times10^{-2}\) & \(0.008\) & \(6.0751\times10^{-3}\) & \(1.1301\times10^{-2}\)\\
\(10^{-5}\) & \(6.0751\times10^{-3}\) & \(1.1301\times10^{-2}\) & \(0.016\) & \(5.7008\times10^{-3}\) & \(9.7964\times10^{-3}\)\\
\(10^{-4}\) & \(5.7125\times10^{-3}\) & \(1.0968\times10^{-2}\) & \(0.032\) & \(1.8757\times10^{-2}\) & \(1.7492\times10^{-2}\)\\ \hline
\end{tabular}
\caption{Absolute conductivity errors for Example~\ref{ex:examp1_1d}.}
\label{tab:ex51-gamma-errors}
\end{table}

\begin{figure}[!htbp]
\centering
\begin{subfigure}[t]{0.32\textwidth}
    \centering
    \includegraphics[width=\linewidth]{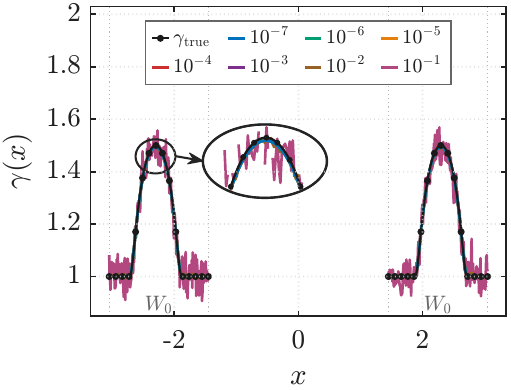}
    \caption{Exterior reconstruction.}
    \label{fig:ex51-exterior}
\end{subfigure}
\hfill
\begin{subfigure}[t]{0.32\textwidth}
    \centering
    \includegraphics[width=\linewidth]{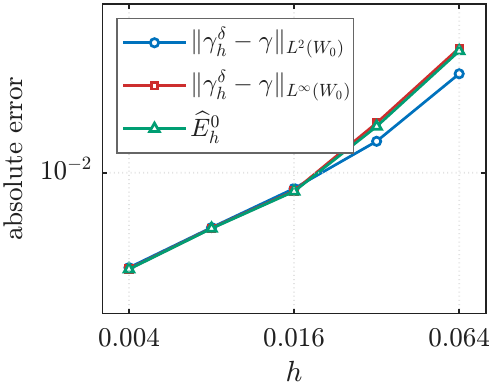}
    \caption{Exterior error versus \(h\).}
    \label{fig:ex51-exterior-error}
\end{subfigure}
\hfill
\begin{subfigure}[t]{0.32\textwidth}
    \centering
    \includegraphics[width=\linewidth]{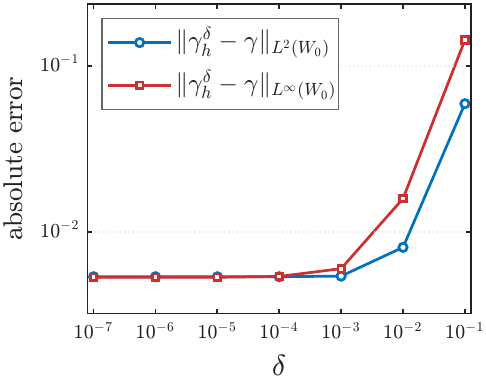}
    \caption{Exterior noise sensitivity.}
    \label{fig:ex51-exterior-noise}
\end{subfigure}

\medskip
\begin{subfigure}[t]{0.32\textwidth}
    \centering
    \includegraphics[width=\linewidth]{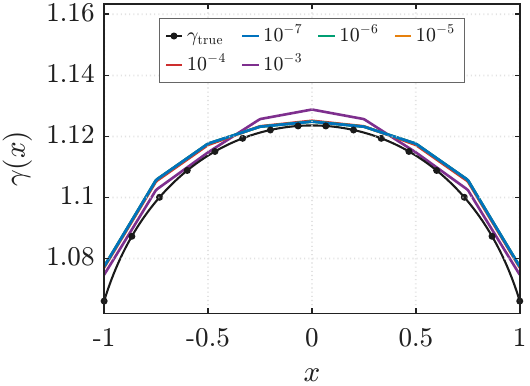}
    \caption{Interior reconstruction.}
    \label{fig:ex51-interior}
\end{subfigure}
\hfill
\begin{subfigure}[t]{0.32\textwidth}
    \centering
    \includegraphics[width=\linewidth]{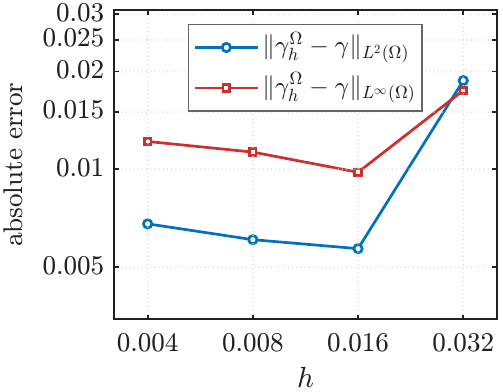}
    \caption{Interior error versus \(h\).}
    \label{fig:ex51-interior-h}
\end{subfigure}
\hfill
\begin{subfigure}[t]{0.32\textwidth}
    \centering
    \includegraphics[width=\linewidth]{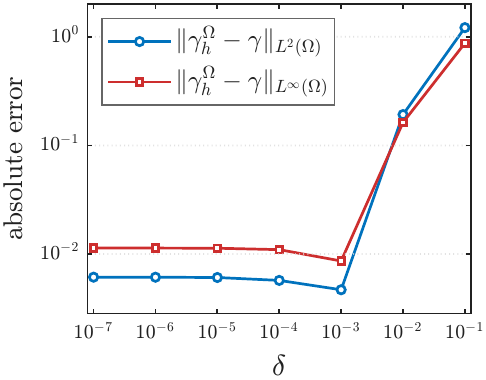}
    \caption{Interior noise sensitivity.}
    \label{fig:ex51-interior-error}
\end{subfigure}
\caption{One-dimensional reconstructions and error studies for Example~\ref{ex:examp1_1d}. The displayed reconstructions use \(h=0.008\) and the noise levels indicated in the legends. Both mesh studies use \(\delta=10^{-5}\). Each interior point receives the exterior result computed at its prescribed \(h\) and \(\delta\).}
\label{fig:ex51-complete}
\end{figure}
\FloatBarrier

\subsection{Example 5.2. Two-dimensional exterior-to-interior reconstruction.}\label{ex:examp2_2d}

We take \(d=2\), \(s=0.60\), \(\Omega=(-1.2,1.2)^2\), and the computational rectangle
\(\Omega_R=(-4,4)\times(-2,2)\). The exterior sets are
\[
\begin{aligned}
W= &(-3.20,-1.35)\times(-1.60,1.60) \cup(1.35,3.20)\times(-1.60,1.60),\\
W_0= &(-3.00,-1.50)\times(-1.25,1.25) \cup(1.50,3.00)\times(-1.25,1.25),\\
W_a=&W_c=(-1.995,-1.510)\times(-1.220,1.220) \cup(1.510,1.995)\times(-1.220,1.220),\\
W_b=&(-2.990,-2.005)\times(-1.220,1.220)\cup(2.005,2.990)\times(-1.220,1.220).
\end{aligned}
\]
Thus \(W_a,W_b,W_c\Subset W_0\) and \(\operatorname{dist}(W_a,W_b)=0.01\). Let \(z=(x,y)\). Using the functions \(\eta\) and \(\chi\) from Example~\ref{ex:examp1_1d}, define
\[
 m_0(z):=(\sqrt{1.50}-1) \left(1-\frac{|z|^2}{1.10^2}\right)_+^s \chi\!\left(\frac{|z|-1}{0.08}\right)
\]
and
\[
 \gamma(z):=(1+m_0(z))^2 +0.35\,\eta\!\left(\frac{|z-(2.25,0)|}{0.40}\right)
 +0.35\,\eta\!\left(\frac{|z+(2.25,0)|}{0.40}\right).
\]
Hence \(1\leq\gamma\leq1.50\) in \(\Omega\), while each exterior bump has peak value \(1.35\).

We use endpoint-inclusive tensor grids and symmetric nodal quadrature with lumped bilinear mass weights. The \(W_a\)-source and \(W_b\)-test vectors have disjoint nodal supports, although the associated \(Q_1\) functions need not have separated supports. The implementation uses the resulting nodal-quadrature analogue of \eqref{eq:separated-identity}, with diagonal kernel pairs omitted. Exterior interactions use midpoint quadrature in \([-10,10]^2\setminus\Omega_R\), supplemented by \((\pi/s)10^{-2s}\). Conductivity and reference forms use the same approximation; their DN matrices are obtained by Schur complementation. The exterior mesh study uses \(h\in\{0.32,0.16,0.08,0.04\}\), \(\delta=10^{-7}\), and \(\alpha=0.02h^2\). Localized sources use the capped nominal radius ratios \(r/h\in\{1.5,2,2.5,3,4,5\}\), calibrated and extrapolated as in Appendix~\ref{app:calibrated-dn}. The reported \(L^2\)-errors use the lumped nodal weights, and the \(L^\infty\)-errors are maxima over the corresponding nodes.

The interior studies use a fixed exterior reconstruction computed separately at \(h=0.16\) and \(\delta=10^{-7}\), transferred to each interior grid before forming the Liouville data. State corrections use a \(7\times7\) tensor-bilinear grid; the interior potential is constant on a \(2\times2\) partition, and the tail potential has one coefficient per component of \(W_c\). The Cauchy space uses eight radial knots with prescribed radial scale \(1.10\) in \(\Omega\), and a \(4\times5\) tensor-bilinear grid on each component of \(W_c\). Radial symmetry and this radial scale are supplied a priori to the inverse Cauchy space. For the final numerical conductivity, we project onto \([0.5,2]\), which contains the exact interior range \([1,1.5]\); this projection cannot increase the pointwise error relative to \(\gamma\). The background deviation is zero outside \(\Omega\cup W_0\).

The interior mesh study uses \(h\in\{0.32,0.16,0.08,0.04\}\) and \(\delta=10^{-7}\), keeping the exterior input and reduced inverse grids fixed. Source counts are selected from the prescribed family by discrete visibility and tail-rank diagnostics, without comparing trial reconstructions to the manufactured solution. These fixed inverse spaces do not test the dense-space limit in Theorem~\ref{thm:final-interior-conductivity-error}. The regularization weights are
\[
\begin{gathered}
\alpha_z(h)=\min\!\left\{5\times10^{-4},\max\!\left\{5\times10^{-5},9\times10^{-5}(h/0.10)^2\right\}\right\},\\
\alpha_w=5\times10^{-4},\quad \alpha_Q=3\times10^{-5},\quad \alpha_t=10^{-7},\quad
\alpha_v=\alpha_\tau=32,\quad \alpha_m=0.2,\quad Q_*=100.
\end{gathered}
\]
The exterior and interior noise studies use \(h=0.04\) and \(h=0.08\), respectively, with nominal levels \(\delta\in\{10^{-7},10^{-5},10^{-4},10^{-2},10^{-1}\}\). 

Figure \ref{fig:ex52-reconstruction} displays the exact conductivity representation. Table~\ref{tab:ex52-gamma-errors} and Figure \ref{fig:ex52-errors} show that refinement from \(h=0.32\) to \(h=0.04\) reduces the computed exterior \(L^2(W_0)\)-error from \(6.2905\times10^{-2}\) to \(8.8340\times10^{-3}\) and the conditional interior \(L^2(\Omega)\)-error from \(2.3709\times10^{-1}\) to \(3.4056\times10^{-2}\). The errors change little over the three smallest noise levels and increase at larger noise, while the exterior \(L^\infty(W_0)\)-errors are nonmonotone under mesh refinement. The interior studies quantify mesh and noise sensitivity for the fixed exterior input and inverse spaces, with regularization chosen as specified above. The reported errors are numerical approximations of the conductivity errors appearing in \eqref{eq:recon-error-final} and \eqref{eq:final-interior-conductivity-error}.

\begin{figure}[!htbp]
\centering
\begin{subfigure}[t]{0.48\textwidth}
    \centering
    \includegraphics[width=\linewidth]{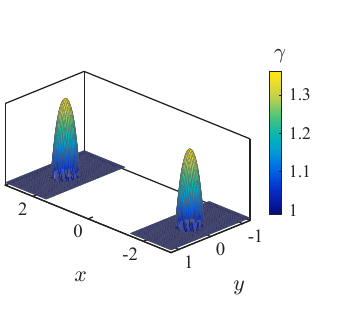}
    \caption{Exact conductivity in \(W_0\).}
    \label{fig:ex52-true-exterior}
\end{subfigure}
\hfill
\begin{subfigure}[t]{0.48\textwidth}
    \centering
    \includegraphics[width=\linewidth]{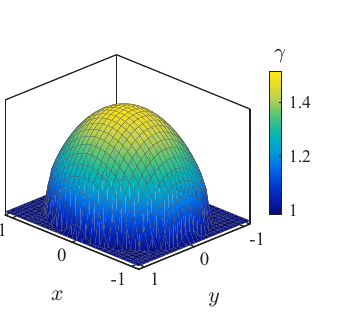}
    \caption{Exact conductivity in \(\Omega\).}
    \label{fig:ex52-true-interior}
\end{subfigure}
\caption{Exact reference conductivities for Example~\ref{ex:examp2_2d}.}
\label{fig:ex52-reconstruction}
\end{figure}

\begin{figure}[!htbp]
\centering
\begin{subfigure}[t]{0.32\textwidth}
    \centering
    \includegraphics[width=\linewidth]{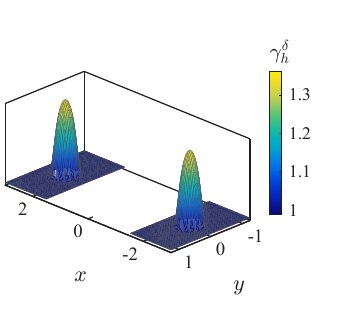}
    \caption{Exterior reconstruction.}
    \label{fig:ex52-reconstructed-exterior}
\end{subfigure}
\hfill
\begin{subfigure}[t]{0.32\textwidth}
    \centering
    \includegraphics[width=\linewidth]{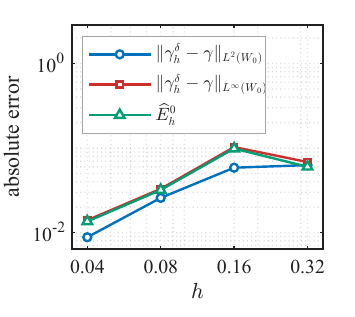}
    \caption{Exterior error versus \(h\).}
    \label{fig:ex52-exterior-h}
\end{subfigure}
\hfill
\begin{subfigure}[t]{0.32\textwidth}
    \centering
    \includegraphics[width=\linewidth]{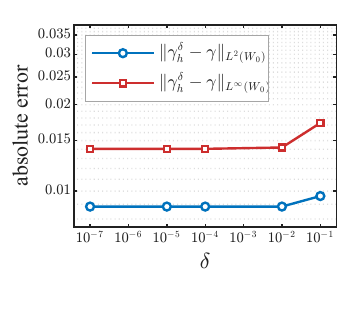}
    \caption{Exterior noise sensitivity.}
    \label{fig:ex52-exterior-noise}
\end{subfigure}

\medskip
\begin{subfigure}[t]{0.32\textwidth}
    \centering
    \includegraphics[width=\linewidth]{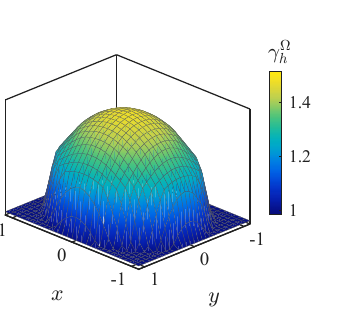}
    \caption{Interior reconstruction.}
    \label{fig:ex52-reconstructed-interior}
\end{subfigure}
\hfill
\begin{subfigure}[t]{0.32\textwidth}
    \centering
    \includegraphics[width=\linewidth]{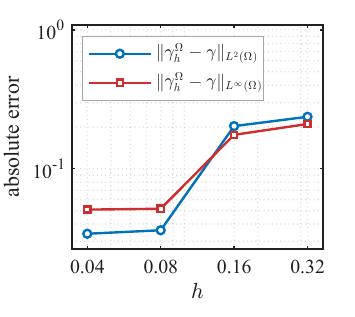}
    \caption{Interior error versus \(h\).}
    \label{fig:ex52-interior-h}
\end{subfigure}
\hfill
\begin{subfigure}[t]{0.32\textwidth}
    \centering
    \includegraphics[width=\linewidth]{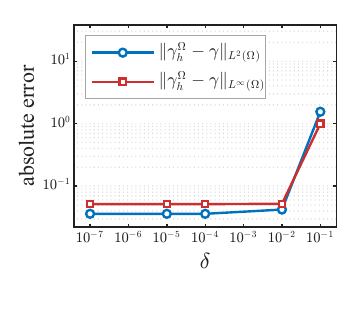}
    \caption{Interior noise sensitivity.}
    \label{fig:ex52-interior-noise}
\end{subfigure}
\caption{Exterior and interior reconstructions and error studies for Example~\ref{ex:examp2_2d}.}
\label{fig:ex52-errors}
\end{figure}
\FloatBarrier

\begin{table}[!htbp]
\centering\small
\setlength{\tabcolsep}{3.4pt}\renewcommand{\arraystretch}{1.1}
\begin{tabular}{|c c c||c c c|}
\hline
\multicolumn{6}{|c|}{\textbf{Exterior reconstruction}}\\ \hline
\multicolumn{3}{|c||}{Noise study, \(h=0.04\)} & \multicolumn{3}{c|}{Mesh study, \(\delta=10^{-7}\)}\\ \hline
\(\delta\) & \(\|\gamma_h^\delta-\gamma\|_{L^2(W_0)}\) & \(\|\gamma_h^\delta-\gamma\|_{L^\infty(W_0)}\) & \(h\) & \(\|\gamma_h^\delta-\gamma\|_{L^2(W_0)}\) & \(\|\gamma_h^\delta-\gamma\|_{L^\infty(W_0)}\)\\ \hline
\(10^{-7}\) & \(8.8340\times10^{-3}\) & \(1.4042\times10^{-2}\) & \(0.04\) & \(8.8340\times10^{-3}\) & \(1.4042\times10^{-2}\)\\
\(10^{-5}\) & \(8.8340\times10^{-3}\) & \(1.4042\times10^{-2}\) & \(0.08\) & \(2.5791\times10^{-2}\) & \(3.2957\times10^{-2}\)\\
\(10^{-4}\) & \(8.8339\times10^{-3}\) & \(1.4044\times10^{-2}\) & \(0.16\) & \(5.8653\times10^{-2}\) & \(1.0276\times10^{-1}\)\\
\(10^{-2}\) & \(8.8353\times10^{-3}\) & \(1.4197\times10^{-2}\) & \(0.32\) & \(6.2905\times10^{-2}\) & \(6.8451\times10^{-2}\)\\ \hline\hline
\multicolumn{6}{|c|}{\textbf{Interior reconstruction}}\\ \hline
\multicolumn{3}{|c||}{Noise study, \(h=0.08\)} & \multicolumn{3}{c|}{Mesh study, \(\delta=10^{-7}\)}\\ \hline
\(\delta\) & \(\|\gamma_h^\Omega-\gamma\|_{L^2(\Omega)}\) & \(\|\gamma_h^\Omega-\gamma\|_{L^\infty(\Omega)}\) & \(h\) & \(\|\gamma_h^\Omega-\gamma\|_{L^2(\Omega)}\) & \(\|\gamma_h^\Omega-\gamma\|_{L^\infty(\Omega)}\)\\ \hline
\(10^{-7}\) & \(3.6094\times10^{-2}\) & \(5.1537\times10^{-2}\) & \(0.04\) & \(3.4056\times10^{-2}\) & \(5.0833\times10^{-2}\)\\
\(10^{-5}\) & \(3.6094\times10^{-2}\) & \(5.1537\times10^{-2}\) & \(0.08\) & \(3.6094\times10^{-2}\) & \(5.1537\times10^{-2}\)\\
\(10^{-4}\) & \(3.6099\times10^{-2}\) & \(5.1541\times10^{-2}\) & \(0.16\) & \(2.0308\times10^{-1}\) & \(1.7574\times10^{-1}\)\\
\(10^{-2}\) & \(4.2220\times10^{-2}\) & \(5.2218\times10^{-2}\) & \(0.32\) & \(2.3709\times10^{-1}\) & \(2.1031\times10^{-1}\)\\ \hline
\end{tabular}
\caption{Absolute conductivity errors for Example~\ref{ex:examp2_2d}.}
\label{tab:ex52-gamma-errors}
\end{table}

\section{Conclusion}
We developed an exterior-to-interior reconstruction for the fractional conductivity equation. The exterior estimates separate concentration, approximation, regularization, and total data discrepancy, including the critical logarithm in the H\"older rate. For full observation spaces, discrete continuation gives logarithmic state-bias and interior conductivity estimates. Fixed-source conditions yield convergence of the clipped potential and exterior tail without couplings involving discrete visibility constants. The complete reconstruction converges under the stated approximation and parameter conditions.

\subsection*{Acknowledgments} 

Dwivedi and Rupp have been supported by the Deutsche Forschungsgemeinschaft (DFG, German Research Foundation) -- 577175348. Railo was supported by the Research Council of Finland through the Flagship of Advanced Mathematics for Sensing, Imaging and Modelling (grant number 359183), the Emil Aaltonen Foundation, and the Jenny and Antti Wihuri Foundation.

The authors used Claude Opus 5 (Anthropic) to assist with English-language drafting, manuscript organization, and literature searches and reference checks. The authors reviewed the mathematical arguments and proofs and take full responsibility for the mathematical content, the accuracy of the references, and the final presentation.

\bibliographystyle{siamdoiurl}
\bibliography{main}

@article{AcostaBorthagarayHeuer2019,
    AUTHOR = {Acosta, Gabriel and Borthagaray, Juan Pablo and Heuer,
              Norbert},
     TITLE = {Finite element approximations of the nonhomogeneous fractional
              {D}irichlet problem},
   JOURNAL = {IMA J. Numer. Anal.},
  FJOURNAL = {IMA Journal of Numerical Analysis},
    VOLUME = {39},
      YEAR = {2019},
    NUMBER = {3},
     PAGES = {1471--1501},
      ISSN = {0272-4979,1464-3642},
   MRCLASS = {65N30},
  MRNUMBER = {4023752},
       DOI = {10.1093/imanum/dry023},
       URL = {https://doi.org/10.1093/imanum/dry023},
}

@article {Li2025numerical,
    AUTHOR = {Li, Xinyan},
     TITLE = {A numerical method for reconstructing the potential in
              fractional {C}alder\'on problem with a single measurement},
   JOURNAL = {Comput. Math. Appl.},
  FJOURNAL = {Computers \& Mathematics with Applications. An International
              Journal},
    VOLUME = {183},
      YEAR = {2025},
     PAGES = {256--270},
      ISSN = {0898-1221,1873-7668},
   MRCLASS = {65M32 (35J10 35J70 35R11 35R30 65F22)},
  MRNUMBER = {4869090},
       DOI = {10.1016/j.camwa.2025.02.018},
       URL = {https://doi.org/10.1016/j.camwa.2025.02.018},
}

@misc{KowNurminenRailo2025Bayesian,
  title={Bayesian inference for the fractional {Calder\'on} problem with a single measurement},
  author={Kow, Pu-Zhao and Nurminen, Janne and Railo, Jesse},
  year          = {2025},
  eprint        = {2511.11068},
  archivePrefix = {arXiv},
  primaryClass  = {math.ST},
}

@article{RinaldoWarma2026Reconstruction,
    AUTHOR = {Rinaldo, Ethan and Warma, Mahamadi},
     TITLE = {Reconstruction algorithms for the fractional {L}aplacian and
              applications to inverse problems},
   JOURNAL = {Inverse Problems},
  FJOURNAL = {Inverse Problems. An International Journal on the Theory and
              Practice of Inverse Problems, Inverse Methods and Computerized
              Inversion of Data},
    VOLUME = {42},
      YEAR = {2026},
    NUMBER = {7},
     PAGES = {Paper No. 075022, 32},
      ISSN = {0266-5611,1361-6420},
   MRCLASS = {35R30 (35R11 35R25)},
  MRNUMBER = {5109907},
       DOI = {10.1088/1361-6420/ae8cb7},
       URL = {https://doi.org/10.1088/1361-6420/ae8cb7},
}

@book{BrezinskiRedivoZaglia1991,
    AUTHOR = {Brezinski, Claude and Redivo Zaglia, Michela},
     TITLE = {Extrapolation methods},
    SERIES = {Studies in Computational Mathematics},
    VOLUME = {2},
 PUBLISHER = {North-Holland Publishing Co., Amsterdam},
      YEAR = {1991},
     PAGES = {x+464},
      ISBN = {0-444-88814-4},
   MRCLASS = {65-02 (65Bxx)},
  MRNUMBER = {1140920},
MRREVIEWER = {Alan\ M.\ Cohen},
}

@article{Uhlmann2009,
    AUTHOR = {Uhlmann, G.},
     TITLE = {Electrical impedance tomography and {C}alder\'on's problem},
   JOURNAL = {Inverse Problems},
  FJOURNAL = {Inverse Problems. An International Journal on the Theory and
              Practice of Inverse Problems, Inverse Methods and Computerized
              Inversion of Data},
    VOLUME = {25},
      YEAR = {2009},
    NUMBER = {12},
     PAGES = {123011, 39},
      ISSN = {0266-5611,1361-6420},
   MRCLASS = {78A48 (35-02 35J25 35R30)},
  MRNUMBER = {3460047},
MRREVIEWER = {Sergey\ G.\ Pyatkov},
       DOI = {10.1088/0266-5611/25/12/123011},
       URL = {https://doi.org/10.1088/0266-5611/25/12/123011},
}

@inproceedings{Calderon1980,
    AUTHOR = {Calder\'on, Alberto-P.},
     TITLE = {On an inverse boundary value problem},
 BOOKTITLE = {Seminar on {N}umerical {A}nalysis and its {A}pplications to
              {C}ontinuum {P}hysics ({R}io de {J}aneiro, 1980)},
     PAGES = {65--73},
 PUBLISHER = {Soc. Brasil. Mat., Rio de Janeiro},
      YEAR = {1980},
   MRCLASS = {35R30 (35K60)},
  MRNUMBER = {590275},
MRREVIEWER = {J.\ R.\ Cannon},
}

@article{Du2012NonlocalDiffusion,
    AUTHOR = {Du, Qiang and Gunzburger, Max and Lehoucq, R. B. and Zhou,
              Kun},
     TITLE = {Analysis and approximation of nonlocal diffusion problems with
              volume constraints},
   JOURNAL = {SIAM Rev.},
  FJOURNAL = {SIAM Review},
    VOLUME = {54},
      YEAR = {2012},
    NUMBER = {4},
     PAGES = {667--696},
      ISSN = {1095-7200,0036-1445},
   MRCLASS = {35R11 (35K20)},
  MRNUMBER = {3023366},
       DOI = {10.1137/110833294},
       URL = {https://doi.org/10.1137/110833294},
}

@article{Du2013NonlocalVectorCalculus,
    AUTHOR = {Du, Qiang and Gunzburger, Max and Lehoucq, R. B. and Zhou,
              Kun},
     TITLE = {A nonlocal vector calculus, nonlocal volume-constrained
              problems, and nonlocal balance laws},
   JOURNAL = {Math. Models Methods Appl. Sci.},
  FJOURNAL = {Mathematical Models and Methods in Applied Sciences},
    VOLUME = {23},
      YEAR = {2013},
    NUMBER = {3},
     PAGES = {493--540},
      ISSN = {0218-2025,1793-6314},
   MRCLASS = {26B12 (26B15 26B20 45P05 46F12)},
  MRNUMBER = {3010838},
MRREVIEWER = {Mikhail\ Korobkov},
       DOI = {10.1142/S0218202512500546},
       URL = {https://doi.org/10.1142/S0218202512500546},
}

@article{Valdinoci2009LongJump,
    AUTHOR = {Valdinoci, Enrico},
     TITLE = {From the long jump random walk to the fractional {L}aplacian},
   JOURNAL = {Bol. Soc. Esp. Mat. Apl. SeMA},
  FJOURNAL = {Bolet\'in de la Sociedad Espa\~nola de Matem\'atica Aplicada.
              SeMA},
    VOLUME = {49},
    NUMBER = {49},
      YEAR = {2009},
     PAGES = {33--44},
      ISSN = {1575-9822},
   MRCLASS = {60G50 (35Q62)},
  MRNUMBER = {2584076},
  URL = {https://arxiv.org/pdf/0901.3261}
}

@article{RailoZimmermann2023counterexamples,
    AUTHOR = {Railo, Jesse and Zimmermann, Philipp},
     TITLE = {Counterexamples to uniqueness in the inverse fractional
              conductivity problem with partial data},
   JOURNAL = {Inverse Probl. Imaging},
  FJOURNAL = {Inverse Problems and Imaging},
    VOLUME = {17},
      YEAR = {2023},
    NUMBER = {2},
     PAGES = {406--418},
      ISSN = {1930-8337,1930-8345},
   MRCLASS = {35R30 (26A33 42B37 46F12)},
  MRNUMBER = {4546955},
       DOI = {10.3934/ipi.2022048},
       URL = {https://doi.org/10.3934/ipi.2022048},
}

@article{CoviKujanpaaRailo2027partial,
    AUTHOR = {Covi, Giovanni and Kujanp\"a\"a, Antti and Railo, Jesse},
     TITLE = {Partial data stability for the inverse fractional conductivity
              problem},
   JOURNAL = {Inverse Probl. Imaging},
  FJOURNAL = {Inverse Problems and Imaging},
    VOLUME = {26},
      YEAR = {2027},
     PAGES = {132--146},
      ISSN = {1930-8337,1930-8345},
   MRCLASS = {35R30 (26A33 42B37)},
  MRNUMBER = {5097243},
       DOI = {10.3934/ipi.2026027},
       URL = {https://doi.org/10.3934/ipi.2026027},
}

@article{BrennerSuperapproximation2020,
    AUTHOR = {Brenner, Susanne C.},
     TITLE = {A general superapproximation result},
   JOURNAL = {Comput. Methods Appl. Math.},
  FJOURNAL = {Computational Methods in Applied Mathematics},
    VOLUME = {20},
      YEAR = {2020},
    NUMBER = {4},
     PAGES = {763--767},
      ISSN = {1609-4840,1609-9389},
   MRCLASS = {65N30},
  MRNUMBER = {4160293},
MRREVIEWER = {Aziz\ Belmiloudi},
       DOI = {10.1515/cmam-2020-0120},
       URL = {https://doi.org/10.1515/cmam-2020-0120},
}

@article{LangerLoescherSteinbachYang2024,
    AUTHOR = {Langer, Ulrich and L\"oscher, Richard and Steinbach, Olaf and
              Yang, Huidong},
     TITLE = {Mass-lumping discretization and solvers for distributed
              elliptic optimal control problems},
   JOURNAL = {Numer. Linear Algebra Appl.},
  FJOURNAL = {Numerical Linear Algebra with Applications},
    VOLUME = {31},
      YEAR = {2024},
    NUMBER = {5},
     PAGES = {Paper No. e2564, 22},
      ISSN = {1070-5325,1099-1506},
   MRCLASS = {65N30 (35J25 49M41 65K10)},
  MRNUMBER = {4804322},
       DOI = {10.1002/nla.2564},
       URL = {https://doi.org/10.1002/nla.2564},
}

@book{ciarlet1978finite,
    AUTHOR = {Ciarlet, Philippe G.},
     TITLE = {The finite element method for elliptic problems},
    SERIES = {Studies in Mathematics and its Applications},
    VOLUME = {4},
 PUBLISHER = {North-Holland Publishing Co., Amsterdam-New York-Oxford},
      YEAR = {1978},
     PAGES = {xix+530},
      ISBN = {0-444-85028-7},
   MRCLASS = {65N30},
  MRNUMBER = {520174},
MRREVIEWER = {Josef\ Nedoma},
}

@misc{DwivediRailoRuppFCP,
  title={A {G}alerkin Finite Element Method for the Fractional {C}alder\'on Problem},
  author={Dwivedi, Mukul and Railo, Jesse and Rupp, Andreas},
  journal={arXiv preprint arXiv:2603.26287},
  year={2026},
  url = {https://arxiv.org/abs/2603.26287}
}

@article{RailoZimmermann2024low,
    AUTHOR = {Railo, Jesse and Zimmermann, Philipp},
     TITLE = {Low regularity theory for the inverse fractional conductivity
              problem},
   JOURNAL = {Nonlinear Anal.},
  FJOURNAL = {Nonlinear Analysis. Theory, Methods \& Applications. An
              International Multidisciplinary Journal},
    VOLUME = {239},
      YEAR = {2024},
     PAGES = {Paper No. 113418, 27},
      ISSN = {0362-546X,1873-5215},
   MRCLASS = {35R30 (26A33 42B37)},
  MRNUMBER = {4658532},
MRREVIEWER = {Enno\ Pais},
       DOI = {10.1016/j.na.2023.113418},
       URL = {https://doi.org/10.1016/j.na.2023.113418},
}

@article{RulandSalo2020low,
    AUTHOR = {R\"uland, Angkana and Salo, Mikko},
     TITLE = {The fractional {C}alder\'on problem: low regularity and
              stability},
   JOURNAL = {Nonlinear Anal.},
  FJOURNAL = {Nonlinear Analysis. Theory, Methods \& Applications. An
              International Multidisciplinary Journal},
    VOLUME = {193},
      YEAR = {2020},
     PAGES = {111529, 56},
      ISSN = {0362-546X,1873-5215},
   MRCLASS = {35R30 (35B35 35B65 35R11 46E35)},
  MRNUMBER = {4062981},
       DOI = {10.1016/j.na.2019.05.010},
       URL = {https://doi.org/10.1016/j.na.2019.05.010},
}

@article{BaersCoviRuland2025instability,
    AUTHOR = {Baers, Hendrik and Covi, Giovanni and R\"uland, Angkana},
     TITLE = {On instability properties of the fractional {C}alder\'on
              problem},
   JOURNAL = {SIAM J. Math. Anal.},
  FJOURNAL = {SIAM Journal on Mathematical Analysis},
    VOLUME = {57},
      YEAR = {2025},
    NUMBER = {5},
     PAGES = {5454--5493},
      ISSN = {0036-1410,1095-7154},
   MRCLASS = {35R30 (35B35)},
  MRNUMBER = {4959954},
MRREVIEWER = {Atsuhide\ Ishida},
       DOI = {10.1137/24M167010X},
       URL = {https://doi.org/10.1137/24M167010X},
}

@book{brenner2008mathematical,
    AUTHOR = {Brenner, Susanne C. and Scott, L. Ridgway},
     TITLE = {The mathematical theory of finite element methods},
    SERIES = {Texts in Applied Mathematics},
    VOLUME = {15},
   EDITION = {Third},
 PUBLISHER = {Springer, New York},
      YEAR = {2008},
     PAGES = {xviii+397},
      ISBN = {978-0-387-75933-3},
   MRCLASS = {65-01 (65-02)},
  MRNUMBER = {2373954},
       DOI = {10.1007/978-0-387-75934-0},
       URL = {https://doi.org/10.1007/978-0-387-75934-0},
}

@article{railo2023fractional,
    AUTHOR = {Railo, Jesse and Zimmermann, Philipp},
     TITLE = {Fractional {C}alder\'on problems and {P}oincar\'e{}
              inequalities on unbounded domains},
   JOURNAL = {J. Spectr. Theory},
  FJOURNAL = {Journal of Spectral Theory},
    VOLUME = {13},
      YEAR = {2023},
    NUMBER = {1},
     PAGES = {63--131},
      ISSN = {1664-039X,1664-0403},
   MRCLASS = {35R30 (26A33 42B37 46F12)},
  MRNUMBER = {4620353},
       DOI = {10.4171/jst/444},
       URL = {https://doi.org/10.4171/jst/444},
}

@article{Covi2020Inverse,
    AUTHOR = {Covi, Giovanni},
     TITLE = {Inverse problems for a fractional conductivity equation},
   JOURNAL = {Nonlinear Anal.},
  FJOURNAL = {Nonlinear Analysis. Theory, Methods \& Applications. An
              International Multidisciplinary Journal},
    VOLUME = {193},
      YEAR = {2020},
     PAGES = {111418, 18},
      ISSN = {0362-546X,1873-5215},
   MRCLASS = {35R11 (35R30)},
  MRNUMBER = {4062967},
       DOI = {10.1016/j.na.2019.01.008},
       URL = {https://doi.org/10.1016/j.na.2019.01.008},
}

@article{covi2022global,
    AUTHOR = {Covi, Giovanni and Railo, Jesse and Zimmermann, Philipp},
     TITLE = {The global inverse fractional conductivity problem},
   JOURNAL = {Calc. Var. Partial Differential Equations},
  FJOURNAL = {Calculus of Variations and Partial Differential Equations},
    VOLUME = {65},
      YEAR = {2026},
    NUMBER = {9},
     PAGES = {Paper No. 258},
      ISSN = {0944-2669,1432-0835},
   MRCLASS = {35R30 (26A33 42B37)},
  MRNUMBER = {5126924},
       DOI = {10.1007/s00526-026-03401-6},
       URL = {https://doi.org/10.1007/s00526-026-03401-6},
}

@article{Covi2024stability,
    AUTHOR = {Covi, Giovanni and Railo, Jesse and Tyni, Teemu and
              Zimmermann, Philipp},
     TITLE = {Stability estimates for the inverse fractional conductivity
              problem},
   JOURNAL = {SIAM J. Math. Anal.},
  FJOURNAL = {SIAM Journal on Mathematical Analysis},
    VOLUME = {56},
      YEAR = {2024},
    NUMBER = {2},
     PAGES = {2456--2487},
      ISSN = {0036-1410,1095-7154},
   MRCLASS = {35R30 (26A33 42B37 46F12)},
       DOI = {10.1137/22M1533542},
       URL = {https://doi.org/10.1137/22M1533542},
}

@incollection{covi2025survey,
    AUTHOR = {Covi, Giovanni},
     TITLE = {The inverse problem for the fractional conductivity equation:
              a survey},
 BOOKTITLE = {Inverse problems: modelling and simulation},
    SERIES = {Trends Math.},
     PAGES = {159--165},
 PUBLISHER = {Birkh\"auser/Springer, Cham},
      YEAR = {2025},
      ISBN = {978-3-031-87212-9; 978-3-031-87213-6},
   MRCLASS = {35R30 (35R11 42B37 46F12)},
  MRNUMBER = {4952816},
      DOI = {10.1007/978-3-031-87213-6_20}
}

@article{Ghosh2020uniqueness,
    AUTHOR = {Ghosh, Tuhin and R\"uland, Angkana and Salo, Mikko and
              Uhlmann, Gunther},
     TITLE = {Uniqueness and reconstruction for the fractional {C}alder\'on
              problem with a single measurement},
   JOURNAL = {J. Funct. Anal.},
  FJOURNAL = {Journal of Functional Analysis},
    VOLUME = {279},
      YEAR = {2020},
    NUMBER = {1},
     PAGES = {108505, 42},
      ISSN = {0022-1236,1096-0783},
   MRCLASS = {35R11},
MRREVIEWER = {Vincenzo\ Ambrosio},
       DOI = {10.1016/j.jfa.2020.108505},
       URL = {https://doi.org/10.1016/j.jfa.2020.108505},
}

@article{Ghosh2020calderon,
    AUTHOR = {Ghosh, Tuhin and Salo, Mikko and Uhlmann, Gunther},
     TITLE = {The {C}alder\'on problem for the fractional {S}chr\"odinger
              equation},
   JOURNAL = {Anal. PDE},
  FJOURNAL = {Analysis \& PDE},
    VOLUME = {13},
      YEAR = {2020},
    NUMBER = {2},
     PAGES = {455--475},
      ISSN = {2157-5045,1948-206X},
   MRCLASS = {35R11 (26A33 35J10 35J70 35R30)},
       DOI = {10.2140/apde.2020.13.455},
       URL = {https://doi.org/10.2140/apde.2020.13.455},
}

@article{Ruland2021single,
    AUTHOR = {R\"uland, Angkana},
     TITLE = {On single measurement stability for the fractional
              {C}alder\'on problem},
   JOURNAL = {SIAM J. Math. Anal.},
  FJOURNAL = {SIAM Journal on Mathematical Analysis},
    VOLUME = {53},
      YEAR = {2021},
    NUMBER = {5},
     PAGES = {5094--5113},
      ISSN = {0036-1410,1095-7154},
   MRCLASS = {35R30 (35B35)},
MRREVIEWER = {Natesan\ Barani Balan},
       DOI = {10.1137/20M1381964},
       URL = {https://doi.org/10.1137/20M1381964},
}

@article{Di2012Hitchhiker,
    AUTHOR = {Di Nezza, Eleonora and Palatucci, Giampiero and Valdinoci,
              Enrico},
     TITLE = {Hitchhiker's guide to the fractional {S}obolev spaces},
   JOURNAL = {Bull. Sci. Math.},
  FJOURNAL = {Bulletin des Sciences Math\'ematiques},
    VOLUME = {136},
      YEAR = {2012},
    NUMBER = {5},
     PAGES = {521--573},
      ISSN = {0007-4497,1952-4773},
   MRCLASS = {46E35 (35A23 35S05 35S30)},
MRREVIEWER = {Lanzhe\ Liu},
       DOI = {10.1016/j.bulsci.2011.12.004},
       URL = {https://doi.org/10.1016/j.bulsci.2011.12.004},
}

@book{Engl1996Regularization,
    AUTHOR = {Engl, Heinz W. and Hanke, Martin and Neubauer, Andreas},
     TITLE = {Regularization of inverse problems},
    SERIES = {Mathematics and its Applications},
    VOLUME = {375},
 PUBLISHER = {Kluwer Academic Publishers Group, Dordrecht},
      YEAR = {1996},
     PAGES = {viii+321},
      ISBN = {0-7923-4157-0},
   MRCLASS = {65J20 (35R30 47A50 47H17)},
MRREVIEWER = {Ulrich\ Tautenhahn},
}

@article{BenningBurger2018,
    AUTHOR = {Benning, Martin and Burger, Martin},
     TITLE = {Modern regularization methods for inverse problems},
   JOURNAL = {Acta Numer.},
  FJOURNAL = {Acta Numerica},
    VOLUME = {27},
      YEAR = {2018},
     PAGES = {1--111},
      ISSN = {0962-4929,1474-0508},
   MRCLASS = {65J22 (35R30 47J25)},
  MRNUMBER = {3826506},
MRREVIEWER = {Anton\ Suhadolc},
       DOI = {10.1017/s0962492918000016},
       URL = {https://doi.org/10.1017/s0962492918000016},
}

@article{AcostaBorthagaray2017,
    AUTHOR = {Acosta, Gabriel and Borthagaray, Juan Pablo},
     TITLE = {A fractional {L}aplace equation: regularity of solutions and
              finite element approximations},
   JOURNAL = {SIAM J. Numer. Anal.},
  FJOURNAL = {SIAM Journal on Numerical Analysis},
    VOLUME = {55},
      YEAR = {2017},
    NUMBER = {2},
     PAGES = {472--495},
      ISSN = {0036-1429,1095-7170},
   MRCLASS = {65N30 (35B65 35R11 65N12 65N15)},
  MRNUMBER = {3620141},
MRREVIEWER = {Hui\ Yu},
       DOI = {10.1137/15M1033952},
       URL = {https://doi.org/10.1137/15M1033952},
}

@article{Acosta2017short,
    AUTHOR = {Acosta, Gabriel and Bersetche, Francisco M. and Borthagaray,
              Juan Pablo},
     TITLE = {A short {FE} implementation for a 2d homogeneous {D}irichlet
              problem of a fractional {L}aplacian},
   JOURNAL = {Comput. Math. Appl.},
  FJOURNAL = {Computers \& Mathematics with Applications. An International
              Journal},
    VOLUME = {74},
      YEAR = {2017},
    NUMBER = {4},
     PAGES = {784--816},
      ISSN = {0898-1221,1873-7668},
   MRCLASS = {65N30 (35R11)},
  MRNUMBER = {3679865},
       DOI = {10.1016/j.camwa.2017.05.026},
       URL = {https://doi.org/10.1016/j.camwa.2017.05.026},
}

\appendix
\section{Quadratic regularization in Hilbert spaces}
\label{app:quadratic-regularization}
We collect the finite-dimensional regularization estimates used in the reconstruction steps; see also \cite[Ch.~4]{Engl1996Regularization}. All adjoints and singular values use the indicated Hilbert inner products.

\begin{lemma}[Shifted Tikhonov regularization]
\label{lem:abstract-shifted-tikhonov}
Let \(\mathsf X,\mathsf Y\) be real Hilbert spaces, let \(E\subset\mathsf X\) be finite dimensional, and let \(T:E\to\mathsf Y\) be linear. For \(b\in\mathsf Y\), \(x_c\in E\), and \(\alpha>0\), the functional
\[
J_b(x):=\tfrac12\|Tx-b\|_{\mathsf Y}^2+\tfrac\alpha2\|x-x_c\|_{\mathsf X}^2,\qquad x\in E,
\]
has a unique minimizer \(x_b\). It satisfies
\begin{equation}\label{eq:abstract-tikhonov-normal}
(Tx_b,Tv)_{\mathsf Y}+\alpha(x_b-x_c,v)_{\mathsf X}=(b,Tv)_{\mathsf Y}\qquad(v\in E).
\end{equation}
For any basis \(\{e_i\}\) of \(E\), the variational matrix \(H_{ij}=(Te_j,Te_i)_{\mathsf Y}+\alpha(e_j,e_i)_{\mathsf X}\) is symmetric positive definite. For every \(x^\dagger\in E\),
\begin{equation}\label{eq:abstract-tikhonov-comparison}
\|Tx_b-b\|_{\mathsf Y}^2+\alpha\|x_b-x_c\|_{\mathsf X}^2
\le\|Tx^\dagger-b\|_{\mathsf Y}^2+\alpha\|x^\dagger-x_c\|_{\mathsf X}^2.
\end{equation}
In particular,
\begin{equation}\label{eq:abstract-tikhonov-apriori}
\|Tx_b-b\|_{\mathsf Y}^2+\alpha\|x_b-x_c\|_{\mathsf X}^2\le\|b-Tx_c\|_{\mathsf Y}^2.
\end{equation}
For fixed \(T,x_c,\alpha\), the minimizers associated with \(b_1,b_2\in\mathsf Y\) satisfy
\begin{equation}\label{eq:abstract-tikhonov-data-stability}
\|x_{b_1}-x_{b_2}\|_{\mathsf X}\le\frac{\|b_1-b_2\|_{\mathsf Y}}{2\sqrt\alpha},\qquad
\|T(x_{b_1}-x_{b_2})\|_{\mathsf Y}\le\|b_1-b_2\|_{\mathsf Y}.
\end{equation}
If \(E\ne\{0\}\), set \(\sigma:=\inf_{0\ne v\in E}\|Tv\|_{\mathsf Y}/\|v\|_{\mathsf X}\ge0\). Then
\begin{align}\label{eq:abstract-tikhonov-comparator}
\|x_b-x^\dagger\|_{\mathsf X}
\le \min\left\{\frac1{2\sqrt\alpha},\frac1{\sqrt{\sigma^2+\alpha}}\right\}\|b-Tx^\dagger\|_{\mathsf Y}+\frac{\alpha}{\sigma^2+\alpha}\|x^\dagger-x_c\|_{\mathsf X}.
\end{align}
For \(\ell\in E^*\), the equation \((Te,Tv)_{\mathsf Y}+\alpha(e,v)_{\mathsf X}=\ell(v)\), \(v\in E\), has a unique solution satisfying
\begin{equation}\label{eq:abstract-tikhonov-load-stability}
\|e\|_{\mathsf X}\le\frac{\|\ell\|_{E^*}}\alpha,\qquad
\|Te\|_{\mathsf Y}\le\frac{\|\ell\|_{E^*}}{2\sqrt\alpha},
\end{equation}
where the dual norm on \(E^*\) is induced by \(\mathsf X\).

Finally, let \(\widetilde x\) be the minimizer with \((T,b)\) replaced by \((\widetilde T,\widetilde b)\), keeping \(E,x_c,\alpha\) fixed. With \(r=Tx_b-b\) and \(\ell_r(v)=(r,(\widetilde T-T)v)_{\mathsf Y}\),
\begin{equation}\label{eq:abstract-tikhonov-operator-perturbation}
\|\widetilde x-x_b\|_{\mathsf X}
\le\frac{\|\widetilde b-b-(\widetilde T-T)x_b\|_{\mathsf Y}}{2\sqrt\alpha}
+\frac{\|\ell_r\|_{E^*}}\alpha.
\end{equation}
\end{lemma}

\begin{proof}
The Hessian \(a_\alpha(u,v):=(Tu,Tv)_{\mathsf Y}+\alpha(u,v)_{\mathsf X}\) satisfies \(a_\alpha(v,v)\ge\alpha\|v\|_{\mathsf X}^2\). Hence the strictly convex functional has a unique minimizer, its first variation gives \eqref{eq:abstract-tikhonov-normal}, and \(\mathbf c^\top H\mathbf c=a_\alpha(\sum_i c_ie_i,\sum_i c_ie_i)>0\) for \(\mathbf c\ne0\). Minimality proves \eqref{eq:abstract-tikhonov-comparison}; choosing \(x^\dagger=x_c\) gives \eqref{eq:abstract-tikhonov-apriori}.

Subtracting two normal equations and testing with \(e=x_{b_1}-x_{b_2}\) yields \(\|Te\|_{\mathsf Y}^2+\alpha\|e\|_{\mathsf X}^2\le D\|Te\|_{\mathsf Y}\), where \(D=\|b_1-b_2\|_{\mathsf Y}\). The inequalities \(Dt-t^2\le D^2/4\) and \(\|Te\|_{\mathsf Y}^2\le D\|Te\|_{\mathsf Y}\) prove \eqref{eq:abstract-tikhonov-data-stability}. For \(B=T^*T\), the normal equation also gives
\[
x_b-x^\dagger=(B+\alpha I)^{-1}T^*(b-Tx^\dagger)-\alpha(B+\alpha I)^{-1}(x^\dagger-x_c).
\]
The eigenvalues of \(B\) are at least \(\sigma^2\). Moreover, for \(t\ge\sigma\), \(t/(t^2+\alpha)\le\min\{1/(2\sqrt\alpha),1/\sqrt{\sigma^2+\alpha}\}\). Spectral decomposition therefore proves \eqref{eq:abstract-tikhonov-comparator}. The same positive Hessian gives unique solvability of the load equation. Testing with \(e\) gives \(\|Te\|_{\mathsf Y}^2+\alpha\|e\|_{\mathsf X}^2\le\|\ell\|_{E^*}\|e\|_{\mathsf X}\); dropping the first term and completing the square in \(\|e\|_{\mathsf X}\) prove \eqref{eq:abstract-tikhonov-load-stability}.

For \(e=\widetilde x-x_b\), subtraction of the normal equations gives
\[
(\widetilde Te,\widetilde Tv)_{\mathsf Y}+\alpha(e,v)_{\mathsf X}
=(\widetilde b-b-(\widetilde T-T)x_b,\widetilde Tv)_{\mathsf Y}-\ell_r(v).
\]
Split \(e\) into the solutions associated with the two terms on the right. Apply \eqref{eq:abstract-tikhonov-data-stability} to the first and \eqref{eq:abstract-tikhonov-load-stability} to the second, both with operator \(\widetilde T\). Their sum gives \eqref{eq:abstract-tikhonov-operator-perturbation}.
\end{proof}

\section{Calibrated and extrapolated DN samples}\label{app:calibrated-dn}
The experiments use reference calibration and radius extrapolation to form the nodal data in \eqref{eq:def-J-h-lambda}. The consistency estimates in this appendix connect the modified data to Theorems~\ref{thm:recon-error} and~\ref{thm:linfty-regularized-reconstruction}.

For each reconstruction node \(x_i\), let \(g_{i,h,r}\in V_h^W\) satisfy \(\|g_{i,h,r}\|_{W^s(\mathbb R^d)}=1\) and \(\operatorname{supp}g_{i,h,r}\subset B(x_i,Cr)\Subset W\), with \(C\) independent of \(i,h,r\). Let \(\mathcal D_{1,h}\) be the DN form assembled on the same mesh for \(\gamma\equiv1\), and set
\[
D_{i,h}(r):=\langle\Lambda_\gamma g_{i,h,r},g_{i,h,r}\rangle,\qquad D^{(1)}_{i,h}(r):=\mathcal D_{1,h}(g_{i,h,r},g_{i,h,r}).
\]
Whenever \(D^{(1)}_{i,h}(r)>0\), define
\begin{equation}\label{eq:calibrated-sample}
\widehat D_{i,h}^\delta(r):=\frac{\mathcal D_h^\delta(g_{i,h,r},g_{i,h,r})}{D^{(1)}_{i,h}(r)}.
\end{equation}
The reference denominator is computable without knowledge of \(\gamma\).

After merging repeated capped radii, let \(r_{i,h,1},\ldots,r_{i,h,m_{i,h}}\) be the distinct radii retained at \(x_i\). Define
\begin{equation}\label{eq:extrapolated-sample}
\widehat D_{i,h}^{\delta,\mathrm{cal}}:=\sum_{k=1}^{m_{i,h}}\omega_{i,h,k}\widehat D_{i,h}^\delta(r_{i,h,k}),\qquad \sum_{k=1}^{m_{i,h}}\omega_{i,h,k}=1.
\end{equation}
For \(m_{i,h}\ge2\), let \(\mathsf V_{i,h}\) have rows \((1,r_{i,h,k}^{2s})\) and take \((\omega_{i,h,1},\ldots,\omega_{i,h,m_{i,h}})=e_1^\top\mathsf V_{i,h}^\dagger\), where \(e_1=(1,0)^\top\). Here \(\mathsf V_{i,h}^{\dagger}\) denotes the Moore--Penrose pseudoinverse, given for \(m_{i,h}\ge2\) by \((\mathsf V_{i,h}^{\top}\mathsf V_{i,h})^{-1}\mathsf V_{i,h}^{\top}\), since the distinct retained radii make \(\mathsf V_{i,h}\) full column rank. Equation~\eqref{eq:extrapolated-sample} is then the intercept of the least-squares fit in \(\{1,r^{2s}\}\); see \cite[Ch.~2]{BrezinskiRedivoZaglia1991}. Since \(\mathsf V_{i,h}\) has full column rank,
\[
\sum_k\omega_{i,h,k}=1,\qquad \sum_k\omega_{i,h,k}r_{i,h,k}^{2s}=0.
\]
For \(m_{i,h}=1\), set \(\omega_{i,h,1}=1\) and retain the single calibrated value. The superscript \(\mathrm{cal}\) in \(\widehat D_{i,h}^{\delta,\mathrm{cal}}\) denotes reference calibration with optional radius fitting.

\begin{lemma}[Consistency of calibrated and extrapolated samples]\label{lem:calibrated-extrapolated-sample}
Assume \(0<\gamma_0\le\gamma\le\gamma_1\) almost everywhere in \(\mathbb R^d\) and \(\sqrt\gamma\in C(\overline W)\). For every retained node and radius, suppose that
\[
|D_{i,h}(r)-\gamma(x_i)|\le\eta_{c,r}(h),\qquad |D^{(1)}_{i,h}(r)-1|\le\eta_{1,r}(h),\qquad D^{(1)}_{i,h}(r)\ge c_0>0,
\]
and \(|\mathcal D_h^\delta(g_{i,h,r},g_{i,h,r})-D_{i,h}(r)|\le\delta\), with \(c_0\) independent of \(i,h,r\). Then
\[
|\widehat D_{i,h}^\delta(r)-\gamma(x_i)|\le c_0^{-1}\bigl(\eta_{c,r}(h)+\gamma_1\eta_{1,r}(h)+\delta\bigr).
\]
If \(\sup_{i,h}\sum_k|\omega_{i,h,k}|\le C_\omega\), the nodal error
\begin{equation}\label{eq:def-Eh-cal-delta}
\widehat E_h^\delta:=\max_{1\le i\le N_h}|\widehat D_{i,h}^{\delta,\mathrm{cal}}-\gamma(x_i)|
\end{equation}
satisfies
\begin{equation}\label{eq:calibrated-extrapolated-error}
\widehat E_h^\delta\le\frac{C_\omega}{c_0}\max_{i,k}\bigl(\eta_{c,r_{i,h,k}}(h)+\gamma_1\eta_{1,r_{i,h,k}}(h)+\delta\bigr).
\end{equation}
\end{lemma}

\begin{proof}
Subtract \(\gamma(x_i)\) from \eqref{eq:calibrated-sample} and use
\begin{align*}
D^{(1)}_{i,h}(r)\bigl(\widehat D_{i,h}^\delta(r)-\gamma(x_i)\bigr)
=\mathcal D_h^\delta(g_{i,h,r},g_{i,h,r})-D_{i,h}(r) +D_{i,h}(r)-\gamma(x_i)+\gamma(x_i)\bigl(1-D^{(1)}_{i,h}(r)\bigr).
\end{align*}
The assumed bounds prove the single-radius estimate. Equation~\eqref{eq:extrapolated-sample}, preservation of constants, and the weight bound imply \eqref{eq:calibrated-extrapolated-error}.
\end{proof}

If \(\max_{i,k}\eta_{1,r_{i,h,k}}(h)\to0\) and \(\sup_{i,h}\sum_k|\omega_{i,h,k}|<\infty\), then the reference denominators are eventually at least \(1/2\), and \eqref{eq:calibrated-extrapolated-error} gives \(\widehat E_h^\delta\lesssim\max_{i,k}(\eta_{c,r_{i,h,k}}+\eta_{1,r_{i,h,k}})+\delta\); distinct radii alone do not ensure the weight bound. Using \(\widehat{\mathbf d}_h^\delta=(\widehat D_{i,h}^{\delta,\mathrm{cal}})_i\) in \eqref{eq:def-J-h-lambda} therefore preserves Theorems~\ref{thm:recon-error} and~\ref{thm:linfty-regularized-reconstruction} with \(E_h^\delta\) replaced by \(\widehat E_h^\delta\), including the rate of Theorem \ref{cor:linfty-fully-discrete-rate} if the reference and concentration errors are \(O(\mathfrak c_h)\). Cancellation by the \(r^{2s}\)-fit requires a proved expansion with a radius-independent \(r^{2s}\) coefficient; Hölder regularity alone gives no such expansion, and the fit does not generally cancel an \(r^{2s}\log r\) term. If noise is added after extrapolation, the input vector has the form \(\widehat{\mathbf d}_h^0+\mathbf e_h\), where the superscript \(0\) denotes the vector before this additional perturbation. Its nodal error is at most \(\widehat E_h^0+\delta_{\rm p}\), where \(\delta_{\rm p}:=\|\mathbf e_h\|_{\ell^\infty}\) is the absolute amplitude of the additional sample noise.

\end{document}